\documentclass[11pt]{amsart}
\usepackage{amscd}
\usepackage[arrow,matrix]{xy}
\usepackage{graphicx}
\usepackage{amsmath}
\usepackage{amsrefs}
\usepackage{mathtools}
\usepackage{xcolor}
\usepackage{listings}
\usepackage{soul}
\usepackage{comment}
\usepackage{latexsym,amssymb}
\usepackage{hyperref}
\usepackage[utf8]{inputenc}

\numberwithin{equation}{section}
\theoremstyle{plain}
\newtheorem{lemma}{Lemma}[section]
\newtheorem{proposition}[lemma]{Proposition}
\newtheorem{theorem}[lemma]{Theorem}
\newtheorem{corollary}[lemma]{Corollary}

\theoremstyle{definition}
\newtheorem{definition}[lemma]{Definition}
\newtheorem{remark}[lemma]{Remark}

\usepackage{tikz}
\usepackage{dsfont}

\definecolor{grey}{RGB}{188,188,188}

\newcommand{\vol}{{\rm vol }}

\DeclareMathOperator{\Var}{Var}

\DeclareMathOperator{\End}{End}
\DeclareMathOperator{\Ric}{Ric}

\begin{document}
	\newcommand{\R}{{\mathbb R}}
	\newcommand{\C}{{\mathbb C}}
	\newcommand{\E}{{\mathbb E}}
	\newcommand{\F}{{\mathbb F}}
	\renewcommand{\O}{{\mathbb O}}
	\newcommand{\Z}{{\mathbb Z}}
	\newcommand{\N}{{\mathbb N}}
	\newcommand{\Q}{{\mathbb Q}}
	\renewcommand{\H}{{\mathbb H}}
	\newcommand{\X}{{\mathfrak X}}

	\newcommand{\Aa}{{\mathcal A}}
	\newcommand{\Bb}{{\mathcal B}}
	\newcommand{\Cc}{{\mathcal C}}    
	\newcommand{\Dd}{{\mathcal D}}
	\newcommand{\Ee}{{\mathcal E}}
	\newcommand{\Ff}{{\mathcal F}}
	\newcommand{\Gg}{{\mathcal G}}    
	\newcommand{\Hh}{{\mathcal H}}
	\newcommand{\Kk}{{\mathcal K}}
	\newcommand{\Ii}{{\mathcal I}}
	\newcommand{\Jj}{{\mathcal J}}
	\newcommand{\Ll}{{\mathcal L}}    
	\newcommand{\Mm}{{\mathcal M}}    
	\newcommand{\Nn}{{\mathcal N}}
	\newcommand{\Oo}{{\mathcal O}}
	\newcommand{\Pp}{{\mathcal P}}
	\newcommand{\Qq}{{\mathcal Q}}
	\newcommand{\Rr}{{\mathcal R}}
	\newcommand{\Ss}{{\mathcal S}}
	\newcommand{\Tt}{{\mathcal T}}
	\newcommand{\Uu}{{\mathcal U}}
	\newcommand{\Vv}{{\mathcal V}}
	\newcommand{\Ww}{{\mathcal W}}
	\newcommand{\Xx}{{\mathcal X}}
	\newcommand{\Yy}{{\mathcal Y}}
	\newcommand{\Zz}{{\mathcal Z}}
	
	\newcommand{\Kb}{{\mathbf K}}
	\newcommand{\Gb}{{\mathbf G}}
	\newcommand{\Lb}{{\mathbf L}}
    \newcommand{\pb}{{\mathbf p}}
    \newcommand{\qb}{{\mathbf q}}
    \newcommand{\Qb}{{\mathbf Q}}
    \newcommand{\yb}{{\mathbf y}}
    \newcommand{\xb}{{\mathbf x}}

	\renewcommand{\a}{{\mathfrak a}}
	\renewcommand{\b}{{\mathfrak b}}
	\newcommand{\e}{{\mathfrak e}}
	\renewcommand{\k}{{\mathfrak k}}
	\newcommand{\m}{{\mathfrak m}}
	\newcommand{\pg}{{\mathfrak p}}
	\newcommand{\g}{{\mathfrak g}}
	\newcommand{\gl}{{\mathfrak gl}}
	\newcommand{\h}{{\mathfrak h}}
	\renewcommand{\l}{{\mathfrak l}}
	\newcommand{\sm}{{\mathfrak m}}
	\newcommand{\n}{{\mathfrak n}}
	\newcommand{\s}{{\mathfrak s}}
	\renewcommand{\o}{{\mathfrak o}}
	\renewcommand{\so}{{\mathfrak so}}
	\renewcommand{\u}{{\mathfrak u}}
	\newcommand{\su}{{\mathfrak su}}

	\newcommand{\ssl}{{\mathfrak sl}}
	\newcommand{\ssp}{{\mathfrak sp}}
	\renewcommand{\t}{{\mathfrak t }}
	
	\newcommand{\zt}{{\tilde z}}
	\newcommand{\xt}{{\tilde x}}
	\newcommand{\Ht}{\widetilde{H}}
	\newcommand{\ut}{{\tilde u}}
	\newcommand{\Mt}{{\widetilde M}}
	\newcommand{\Llt}{{\widetilde{\mathcal L}}}
	\newcommand{\yt}{{\tilde y}}
	\newcommand{\vt}{{\tilde v}}
	\newcommand{\Ppt}{{\widetilde{\mathcal P}}}
	\newcommand{\op}{\mathrm{op}}
	\newcommand{\wh}{\widehat}
	\newcommand{\bp }{{\bar \partial}}
	
	\newcommand{\Remark}{{\it Remark}}
	\newcommand{\Proof}{{\it Proof}}
	\newcommand{\ad}{{\rm ad}}
	\newcommand{\Om}{{\Omega}}
	\newcommand{\om}{{\omega}}
	\newcommand{\eps}{{\varepsilon}}
	\newcommand{\Di}{{\rm Diff}}
	
	\renewcommand{\a}{{\mathfrak a}}
	\renewcommand{\b}{{\mathfrak b}}
	\renewcommand{\k}{{\mathfrak k}}
	\renewcommand{\l}{{\mathfrak l}}
	\renewcommand{\o}{{\mathfrak o}}
	\renewcommand{\u}{{\mathfrak u}}
	\renewcommand{\t}{{\mathfrak t }}
	\newcommand{\Cinf}{C^{\infty}}
	\newcommand{\la}{\langle}
	\newcommand{\ra}{\rangle}
	\newcommand{\half}{\scriptstyle\frac{1}{2}}
	\newcommand{\p}{{\partial}}
	\newcommand{\notsub}{\not\subset}
	\newcommand{\iI}{{I}}               
	\newcommand{\bI}{{\partial I}}      
	\newcommand{\LRA}{\Longrightarrow}
	\newcommand{\LLA}{\Longleftarrow}
	\newcommand{\lra}{\longrightarrow}
	\newcommand{\LLR}{\Longleftrightarrow}
	\newcommand{\lla}{\longleftarrow}
	\newcommand{\INTO}{\hookrightarrow}
	
	\newcommand{\QED}{\hfill$\Box$\medskip}
	\newcommand{\UuU}{\Upsilon _{\delta}(H_0) \times \Uu _{\delta} (J_0)}
	\newcommand{\bm}{\boldmath}

	\newcommand{\GL}{{\rm GL}}
	\newcommand{\SL}{{\rm SL}}
	\newcommand{\SO}{{\rm SO}}
	\newcommand{\G}{{\rm G_2}}
	\newcommand{\Spin}{{\rm Spin(7)}}
	\newcommand{\smallr}{\mathrm{small}}
	\newcommand{\Comm}{{\mathrm Comm}}
	
\newcommand{\Meas}{\mathbf{Meas}\,}

	\newenvironment{nouppercase}{%
		\let\uppercase\relax%
		\renewcommand{\uppercasenonmath}[1]{}}{}

	\setcounter{tocdepth}{1}
	
	\title[Empirical  Hodge Laplacians  and Harmonic forms]{Empirical Hodge Laplacians: Spectral Convergence and Harmonic Forms from Point Clouds}
	
	\author{H\^ong V\^an L\^e}
	\address{Institute of Mathematics of the Czech Academy of Sciences, \v Zitn\'a 25, 115 67 Praha 1, Czech Republic}
	\email{hvle@math.cas.cz}
	\date{July 13, 2026}
	
	\begin{abstract}Let $M^n\subset\mathbb R^d$ be a closed, connected, orientable $C^4$-smooth Riemannian submanifold of dimension $n\ge3$. We construct, for each degree $0\le k\le n$, a family of deformed Hodge Laplacians $\Delta_t^k$, $t>0$, defined in terms of the extrinsic geometry of $M^n$, and prove that $\Delta_t^k$ converges uniformly to the classical Hodge Laplacian $\Delta^k$ as $t\to0^+$. Given an i.i.d. uniformly distributed point cloud $S_m\subset M^n$, we define empirical Hodge operators $\widehat\Delta_{t,S_m}^k$. Under the scaling $t=m^{-1/(2n)}$, we prove uniform consistency in probability and compact Mosco convergence in probability of the associated quadratic forms. Consequently, the empirical spectral cluster near zero contains exactly the $k$-th Betti number $b_k$ of eigenvalues, counted with multiplicity, and converges in the transported discrete $L^2$-sense to the space of harmonic $k$-forms. We also construct consistent empirical estimators of the tangent projection, the second fundamental form, the Riemannian curvature tensor, and the Weitzenb\"ock curvature endomorphisms. As applications, we obtain consistent recovery of the Betti numbers and harmonic representatives of de Rham cohomology, as well as of the Pontryagin forms and their periods. Evaluating these periods on a homology basis recovers the corresponding real Pontryagin classes from sampled data.	
	\end{abstract}

	\keywords{Hodge Laplacian,    extrinsic metric, second fundamental form, Riemannian curvature, Pontryagin characteristic class,  manifold learning}
	\subjclass[2020]{Primary: 62R40; Secondary: 58A14, 55N10, 53C40}
	
	\maketitle
	
	\tableofcontents

\section{Introduction}\label{sec:intro}

The modern paradigm of spectral manifold learning was largely
catalyzed by the seminal work of Belkin and Niyogi
\cites{BN2003,BN2008,BN20062008}, who showed that the scalar
Laplace--Beltrami operator can be approximated from discrete
point-cloud data by graph Laplacians. This approach and its subsequent
refinements have become central tools in dimensionality reduction,
spectral clustering, and geometric data analysis. The nullspace of the
scalar Laplacian identifies the connected components of the underlying
space and therefore represents the degree-zero cohomology group  $H^0$.
 It does not, however, directly provide harmonic
representatives of higher-degree de Rham cohomology.

To access higher-degree invariants, one must pass from scalar functions
to differential $k$-forms. Unlike scalar functions, differential
forms require the coherent estimation of tangent spaces and their
exterior powers from discrete samples. Many existing constructions of
discrete Hodge Laplacians are based on combinatorial complexes, such as
\v{C}ech or Vietoris--Rips complexes, constructed from the data. Such
complexes can exhibit rapid combinatorial growth, and their comparison
with the smooth Hodge Laplacian involves a different discrete-to-smooth
analytic framework.

In this paper, we pursue an extrinsic and operator-theoretic approach.
We estimate the tangent projections and their exterior powers by
localized kernel methods and use smooth spatial cutoff functions to
construct empirical Hodge Laplacians directly on the resulting
projected discrete bundle. This extends the graph-Laplacian paradigm
from scalar functions to differential forms while retaining the
ambient Euclidean representation of the data.

Our first construction is deterministic. For a closed Riemannian
submanifold
\[
M^n\subset\mathbb R^d,
\qquad n\ge2,
\]
we define, in every degree $0\le k\le n$, a one-parameter family of
deformed Hodge operators
\[
\Delta_t^k,\qquad t>0,
\]
in terms of the extrinsic geometry of $M^n$. We prove pointwise convergence to the classical Hodge Laplacian
in Theorem~\ref{thm:pwconvergencel}. For sufficiently regular forms,
we further establish the uniform estimate
\[
\Delta_t^k=\Delta^k+O(t)
\]
as $t\to0^+$; see Theorem~\ref{thm:analytic_limit}.

We then turn to empirical geometry. Let
\[
S_m=\{x_1,\ldots,x_m\}\sim\mu^m
\]
be an i.i.d.\, sample from the uniform probability measure on $M^n$.
Following the local-PCA framework of Aamari and Levrard
\cite{AL2019}, we construct empirical tangent projections and prove
their uniform convergence in probability; see
Proposition~\ref{prop:empitangent_clean}. From these projections and
localized Gaussian moments, we construct empirical estimators of the
second fundamental form, its associated curvature tensors, and the
Weitzenb\"ock curvature endomorphisms. We prove uniform convergence in
probability of these geometric quantities under the standing sampling
regime.

These ingredients are assembled into empirical Hodge operators
\[
\widehat\Delta^k_{t,S_m}.
\]
For $n\ge3$, and under the scaling
\[
t=m^{-1/(2n)},
\]
we prove uniform consistency in probability on $C^4$-smooth
differential forms; see Theorem~\ref{thm:empiconvergence}. We further
establish compact Mosco convergence of the associated empirical
quadratic forms. Consequently, for every fixed degree $k$, the
empirical spectral cluster near zero contains exactly
\[
b_k=\dim\mathcal H^k(M^n)
\]
eigenvalues, counted with multiplicity, with probability tending to
$1$. Moreover, this discrete spectral subspace admits an isometric
identification with the harmonic space
\[
\mathcal H^k(M^n)=\ker\Delta^k
\]
whose discrepancy from the empirical restriction map tends to zero in
the discrete $L^2$-norm; see
Theorem \ref{thm:harmonic_cluster}. In particular, the procedure
consistently recovers the Betti numbers and provides a transported
$L^2$-approximation of the harmonic representatives of de Rham
cohomology.

The curvature estimators also lead to characteristic-class
applications. We construct empirical Pontryagin forms and prove
convergence of the associated periods and characteristic numbers
considered in this paper; see
Theorems~\ref{thm:pontryagin} and
\ref{thm:persistent_cycles}. Thus the same point-cloud framework
recovers both low-energy Hodge spectral information and extrinsic
curvature data.

The paper is organized as follows. In Section \ref{sec:pre}, we fix
the notation and collect the geometric and analytic preliminaries.

In Section \ref{sec:modified}, we construct the deterministic deformed
Hodge operators $\Delta_t^k$, derive their asymptotic expansions, and
prove their convergence to the intrinsic Hodge Laplacians.

In Section \ref{sec:empi}, we develop the empirical tangent,
second-fundamental-form, and curvature estimators. We then define the
empirical Hodge operators, prove uniform consistency and compact Mosco
convergence, and establish convergence of the empirical harmonic
spectral clusters.

In Section \ref{sec:learn}, we apply the curvature convergence results
to the recovery of Pontryagin forms and their associated periods from
uniformly sampled point-cloud data.

In Section \ref{sec:fin}, we compare our construction with related
connection-Laplacian frameworks
\cites{SW2012,SW2017} and discuss further directions. In particular,
although empirical eigenforms admit exact continuous Nystr\"om
extensions, proving uniform convergence of these extensions requires
additional low-energy regularity estimates. Such estimates would also
provide a possible route toward recovering products of harmonic forms
and, ultimately, the de Rham cohomology ring.

Finally, Appendix \ref{sec:empitangent} contains the proof of the
consistency of the empirical tangent projections stated in
Proposition~\ref{prop:empitangent_clean}  and Appendix \ref{sec:density} collects several analytic density lemmas
used in the convergence proofs.

\section*{Acknowledgements}   
This   research  was supported by the Institute of Mathematics,  Czech Academy of Sciences (RVO: 67985840).

\section{Preliminaries}\label{sec:pre}

In this section, we fix notation and recall several standard facts from
Riemannian geometry; see, for instance, \cites{KN1963, KN1969, Jost2017}.
We also collect some facts concerning the extrinsic geometry of
Riemannian submanifolds, which will be used in later sections.

Let $(M,g)$ be a Riemannian submanifold of dimension $n$ of a
Riemannian manifold $(N,\bar g)$, and let $1\le k\le n$. The inner
product $\langle\cdot,\cdot\rangle_g$ on the fibers of $TM$ induces
an inner product, again denoted by $\langle\cdot,\cdot\rangle_g$, on
the fibers of $\Lambda^k T^*M$. We denote by $\nabla$ the
Levi--Civita connection on $M$, as well as the induced connection on
$\Lambda^kT^*M$, and by $\nabla^*$ its formal adjoint. For a vector
bundle $E$ over $M$, we denote by $\Gamma(E)$ the space of smooth
sections of $E$ and by $\Omega^k(M,E)$ the space of $E$-valued
$k$-forms on $M$. We write $\Omega^k(M)$ for
$\Gamma(\Lambda^kT^*M)$  and  $\X (M)$ for  $\Gamma (TM)$.

 The curvature tensor of the Levi--Civita connection will be denoted by
 $R\in \Omega^2(M,\mathfrak{so}(TM))$; thus, for vector fields
 $X,Y,Z$ on $M$,
 \[
 R(X,Y)Z
 =
 \nabla_X\nabla_YZ-\nabla_Y\nabla_XZ-\nabla_{[X,Y]}Z.
 \]
 
 The Ricci tensor is defined by
 \[
 \mathrm{Ric}_p(v,w)
 \coloneqq
 \sum_{i=1}^n
 \langle R(v,e_i)e_i,w\rangle_g,
 \]
 where $\{e_i\}_{i=1}^n$ is an orthonormal basis of $T_pM$.

Denote by $T^\perp M$ the normal bundle over $M$ and by $\mathrm{Sym}(TM \times TM, T^\perp M)$ the vector bundle over $M$ whose fibers consist of all $T^\perp  M$-valued symmetric bilinear forms on $TM$. The difference between the ambient Levi-Civita covariant derivative $\bar{\nabla}$ on $(N, \bar g)$ and the intrinsic Levi-Civita covariant derivative $\nabla$ on $(M, g)$ is captured by the   second fundamental form $B \in \Gamma\big(M, \mathrm{Sym}(TM \times TM, T^\perp M)\big)$, as defined in the Gauss formula.

\begin{proposition}[{\bf The Gauss Formula}]
	\label{prop:gauss_formula}
	For any tangent vector fields $X, Y \in \mathfrak{X}(M)$, the ambient derivative decomposes into tangential and normal components:
	\begin{equation}
		\bar{\nabla}_X Y = \nabla_X Y + B(X, Y). \label{eq:gausseq}
	\end{equation} 
\end{proposition}

\begin{definition}[{\bf The Shape Operator}]
	\label{def:shape_operator}
	For a fixed normal vector field $\xi \in \Gamma  (T^\perp M) $, the {\it shape operator} (or {\it Weingarten Map})\index{Shape Operator} \index{Weingarten Map} $A_\xi: \X(M) \to \X (M)$ is the self-adjoint linear operator defined by the identity
	\begin{equation}
		\la A_\xi X, Y\ra_g = \la B(X, Y), \xi\ra _{\bar g}. \label{eq:shapeoperator}
	\end{equation}
\end{definition}

\begin{proposition}[{\bf The Weingarten Equation}]\label{thm:weingarten}
	For any $X \in \X (M)$ and normal vector field $\xi \in  \Gamma  (T^\perp M)$, the derivative of the normal field is
	\begin{equation}
		\bar\nabla_X \xi = -A_\xi X +  \nabla_X^\perp \xi, \label{eq:weingarten}
	\end{equation}
	where $\nabla_X^\perp \xi \coloneqq \pi^\perp(\bar{\nabla}_X \xi)$ is the {\it normal connection} on $T^\perp M$.
\end{proposition}

\begin{proposition}[{\bf Gauss Equation in Flat Space}]
	\label{prop:gauss_flat}
	In the case where the ambient manifold is the Euclidean space $(\mathbb{R}^d, \langle \cdot, \cdot \rangle)$, the ambient curvature vanishes ($\bar{R} \equiv 0$) and the intrinsic curvature of $M$ is given by
	\begin{equation}
		g(R(X, Y)Z, W) = \langle B(X, W), B(Y, Z)\rangle - \langle B(X, Z), B(Y, W) \rangle . \label{eq:gaussflat}
	\end{equation}
\end{proposition}

\

Let $H \in \Gamma(M^n, T^\perp  M^n)$ denote the mean curvature of $(M^n, g)$, defined for any $x \in M^n$ by
$$ H(x) \coloneqq \sum_{i=1}^n B(e_i, e_i), $$
where $\{e_i\}_{i=1}^n$ is an orthonormal basis of $T_x M^n$ (note that we do not normalize by $1/n$).

Denote by $\| \cdot \|_{\R^d}$ the Euclidean norm in $\R^d$ and by $\| \cdot \|$ the norm on $\Lambda TM$ and $\Lambda T^*M$ induced by the Riemannian metric $g$. If $E, F$ are Euclidean vector bundles over $M$, then $\mathrm{End}(E, F)$ is equipped with the operator norm, denoted by $\| \cdot \|_{\mathrm{op}}$.  Generally, we use $|\cdot|$ for scalar absolute values, $\|\cdot\|$ for Euclidean norms on finite-dimensional tensor spaces, and $\|\cdot\|_{op}$ for operator norms.

Let $d_M$ denote the Riemannian distance on $M$.  For $p \in M$ denote by $\exp_p  ^M$  the exponential map $T_p M \to M$.

\begin{lemma}[{\bf Extrinsic-Intrinsic Distance Lemma}]
	\label{lem:intrextr}  \cite{LMPT2026}*{Lemma D.2.72}
	Let $(M,g) \subset (N^d, \bar g)$ be a  compact   $C^3$-smooth Riemannian submanifold of a  smooth $d$-dimensional Riemannian manifold $(N^d, \bar g)$.
	There exist
	$r>0$ and $C>0$ such that, for every $p\in M$ and every
	$\mathbf x\in T_pM$ with $\|\mathbf x\|<r$, if
	\[
	x=\exp_p^M(\mathbf x),
	\]
	then
\begin{align}
	0
	\leq
	\|\mathbf x\|^2-d_N(p,x)^2
	\leq
	C\|\mathbf x\|^4. \label{eq:intrextr}
	\end{align}	
	Since
	\[
	d_M(p,x)=\|\mathbf x\|
	\]
	in the chosen normal ball, this is equivalently
	\[
	0
	\leq
	d_M(p,x)^2-d_N(p,x)^2
	\leq
	C\,d_M(p,x)^4.
	\]
\end{lemma}

In what follows, we specialize to the case where $(N, \bar g) = (\R^d, \la \cdot, \cdot \ra)$. Unless otherwise stated, we also assume that $(M^n, g)$ is a smooth Riemannian submanifold of $\R^d$ of dimension $n$.

In this case, the geometry of the  second fundamental form $B$ is related to the {\it reach} $\tau_{M}$ of $M$, introduced by Federer \cite{Federer1959}. Denoting by $d(\cdot, M)$ the distance from points in $\R^d$ to $M$, we first define the {\it medial axis} $\mathrm{Med}(M)$ of $M$ by
$$ \mathrm{Med}(M) \coloneqq \{ z \in \R^d : \, \exists p \not= q \in M, \: \|z-p\|_{\R^d} = \|z-q\|_{\R^d}  = d(z, M)\}. $$
Then 
$$ \tau_M \coloneqq   \inf_{z\in \mathrm{Med}(M)} d(z, M) = \inf_{p\in M} d\big(p, \mathrm{Med}(M)\big). $$

 For
$0<r<\tau_M$, the normal map
\[
\mathcal E:\{(p,\xi)\in T^\perp M:\|\xi\|<r\}\longrightarrow\R^d,
\qquad \mathcal E(p,\xi)=p+\xi,
\]
is injective, and its image is the open  tubular neighborhood
\[
\operatorname{Tub}_r(M)=\{z\in\R^d:d(z,M)<r\}.
\]
The nearest-point projection $\pi_0:\operatorname{Tub}_{\tau_M}(M)\to M$ is well
defined. 

For $t > 0$, define the {\it extrinsic Gaussian kernel} $\Phi_t: M^n \times M^n \to \R$ by  
\begin{equation}
	\Phi_t(x, y) \coloneqq \frac{1}{(4\pi t)^{n/2}} \exp\left( -\frac{\|x-y\|_{\R^d}^2}{4t} \right). \label{eq:phit}
\end{equation}

Note that while $\Phi_t$ is defined using the ambient distance in $\mathbb{R}^d$, its restriction to $M^n$ admits the same small-time asymptotic expansion as the intrinsic heat kernel up to higher order (see Remark \ref{rem:extint}).

For $r>0$ and $p\in M$, define the extrinsic closed ball in $M$ as
\begin{equation}
	D_r(p) \coloneqq \{ y \in M : \|y-p\|_{\R^d} \le r \}. \label{eq:extb}
\end{equation}

To ensure that our empirical differential operators remain  continuous with respect to the spatial parameter $p \in M^n$, we avoid hard indicator functions. Instead, we localize our kernels using a smooth ``soft support'' cut-off.
 Fix a smooth nonincreasing bump function $\chi: \mathbb{R} \to [0,1]$ such that $\chi(s) = 1$ for $s \le 1/2$ and $\chi(s) = 0$ for $s \ge 1$. For any fixed radius $\delta > 0$, we define the smooth spatial cut-off function 
\begin{equation}
\chi_\delta(p, y) \coloneqq \chi\Big(\frac{\|y-p\|}{\delta}\Big).\label{eq:cutoff}
\end{equation}

 Denote by $\mu$ the uniform distribution on $M$, i.e., 
$$d\mu(p) \coloneqq \frac{d\vol_g(p)}{\vol_g(M)}.$$

Consequently,  by \cite{Gray1974}*{Corollary (2.10)}
\begin{equation}
d\mu(\exp_p v)
=\frac{1}{\vol_g(M)}
\left(1-\frac16\operatorname{Ric}_p(v,v)+O(\|v\|^3)\right)dv\label{eq:gray}
\end{equation}
where $dv$ is the Lebesgue measure on $T_p M^n$.

\begin{lemma}[\bf Properties of the exponential map and soft support]\label{lem:compact-exponential-open-subset}
	Let $M^n \subset \mathbb{R}^d$ be a compact $C^3$-smooth submanifold with reach $\tau_{M^n} > 0$. Let $\delta \in (0, \tau_{M^n}/4)$.
	
	(a) \textbf{(Geometric bounds):} For any $p \in M^n$, the operator norm of the second fundamental form $B(p)$ is uniformly bounded by the reach:
	\begin{equation}
		\|B(p)\|_{\mathrm{op}} \le \frac{1}{\tau_{M^n}}. \label{eq:nsw}
	\end{equation}
	Furthermore, the intrinsic injectivity radius of the submanifold is bounded below by the reach:
	\begin{equation}
		\mathrm{inj}(M^n) \ge \frac{\tau_{M^n}}{4}. \label{eq:inj}
	\end{equation}

(b) \textbf{(Smooth cut-off properties):}
Let $\chi:\mathbb R\to[0,1]$ be a smooth bump function such that
$\chi(s)=1$ for $s\le 1/2$ and $\chi(s)=0$ for $s\ge 1$. The spatial
cut-off function
\[
\chi_\delta(p,y)\coloneqq \chi(\|y-p\|_{\mathbb R^d}/\delta)
\]
is smooth on $M^n\times M^n$, has support contained in the extrinsic
ball $D_\delta(p)$ in the $y$-variable, and satisfies
$\chi_\delta(p,y)=1$ for all $y\in D_{\delta/2}(p)$.

	(c) \textbf{(Asymptotic integration equivalence):} Let $f: M^n \to \mathbb{R}^N$ be a measurable function with bounded supremum norm. For any $k > 0$, as $t \to 0^+$, the difference between integrating against the smooth localized kernel and the unlocalized global kernel is exponentially suppressed:
	\begin{equation}
		\left\| \int_{M^n} f(y) \Phi_t(p, y) \chi_\delta(p, y) \,d\mu(y) - \int_{M^n} f(y) \Phi_t(p, y) \,d\mu(y) \right\| = O(t^k). \label{eq:bnest1}
	\end{equation}
	Consequently, integrating against the soft support $\chi_\delta(p, y)$ yields the exact same asymptotic Taylor expansion in $t$ as integrating over the entire manifold or integrating against the hard indicator $\mathbf{1}_{D_\delta(p)}(y)$.
\end{lemma}

\begin{proof}
Assertion (a):	Equation \eqref{eq:nsw}    is due to  Niyogi, Smale and Weinberger 	\cite{NSW2008}*{Proposition 6.1}.
Equation  \eqref{eq:inj} follows from \cite{AB2006}*{Corollary 1.4}.

	Assertion (b) follows from standard differential geometry and the definition of the smooth bump function $\chi_{\delta}$. The only possible issue is smoothness along the diagonal $p=y$.
	However, $\chi$ is constant on a neighborhood of $0$, and hence
	$\chi(\|y-p\|_{\mathbb R^d}/\delta)$ is smooth there as well.
	
	To prove (c), we evaluate the difference between the integrals. Since $\chi_\delta(p, y) = 1$ on $D_{\delta/2}(p)$, the integrand vanishes inside this smaller ball. Thus, the integration error is strictly confined to the complement $M^n \setminus D_{\delta/2}(p)$, where the distance satisfies $\|y-p\| \ge \delta/2$.
	
	Using the supremum bound of $f$, the difference is bounded by:
	\begin{align*}
		E_t &\le \int_{M^n \setminus D_{\delta/2}(p)} \|f(y)\| \Phi_t(p,y) \big(1 - \chi_\delta(p,y)\big) \,d\mu(y) \\
		&\le \|f\|_{L^\infty} \int_{M^n \setminus D_{\delta/2}(p)} \frac{1}{(4\pi t)^{n/2}} e^{-\frac{\|y-p\|^2}{4t}} \,d\mu(y).
	\end{align*}
	Because $\|y-p\| \ge \delta/2$ in this domain, we have $e^{-\|y-p\|^2/4t} \le e^{-\delta^2/16t}$. Therefore:
	$$ E_t \le \|f\|_{L^\infty}  \frac{1}{(4\pi t)^{n/2}} e^{-\frac{\delta^2}{16t}}. $$
	Because the exponential term $e^{-c/t}$ decays to zero faster than any polynomial $t^k$ as $t \to 0^+$, we conclude that $E_t = O(t^k)$ for any arbitrarily large integer $k$. This establishes \eqref{eq:bnest1} and completes  the proof of Lemma \ref{lem:compact-exponential-open-subset}.
\end{proof}

\begin{lemma}[\bf Taylor expansion of tangent vector fields in RNC]
	\label{lem:tangentemb}
	Let $M^n\subset\mathbb R^d$ be a Riemannian submanifold, let
	$p\in M^n$, and let $\{e_i\}_{i=1}^n$ be an orthonormal basis of
	$T_pM$. In Riemannian normal coordinates centered at $p$, write
	$y=\exp_p(v)$, where $v=\sum_i v^i e_i\in T_pM$. If $Y$ is a smooth
	tangent vector field on $M^n$, regarded as an $\mathbb R^d$-valued
	function via the inclusion $T_yM^n\subset T_y\mathbb R^d\simeq
	\mathbb R^d$, then, as $v\to0$,
	\begin{align}
		Y(v)
		= & 
		Y(p)
		+
		\sum_j v^j\big(\nabla_jY+B(e_j,Y)\big)
		\nonumber\\
		&\quad
		+
		\frac12\sum_{j,l}v^jv^l
		\Big(
		\nabla_j\nabla_lY
		+
		B(e_j,\nabla_lY)
		+
		\nabla_j^\perp\big(B(e_l,Y)\big)
		-
		A_{B(e_l,Y)}e_j
		\Big)\nonumber\\
		&+
		O(\|v\|^3),
		\label{eq:piY}
	\end{align}
	where all coefficient terms on the right-hand side are evaluated at
	$p$, and $\nabla_j\coloneqq\nabla_{e_j}$, $\nabla_j^\perp \coloneqq\nabla_{e_j}^\perp$.
\end{lemma}

\begin{proof}[Proof of Lemma \ref{lem:tangentemb}]   Let $(v^1,\dots,v^n)$ be Riemannian normal coordinates centered at $p$, associated with the orthonormal basis $\{e_j\}_{j=1}^n \subset T_pM^n$. Then
	\[
	\left. \p_j\coloneqq \frac{\partial}{\partial v^j}\right|_p=e_j.
	\]
	Viewing a tangent vector field $Y$ on $M^n$ as an $\mathbb R^d$-valued function through the embedding $M^n\subset\mathbb R^d$, derivatives at the base point may be computed using the ambient flat connection:
	\[
	\frac{\partial Y}{\partial v^j}(0)
	=
	\bar\nabla_{e_j}Y\big|_p.
	\]

	1) We consider the Taylor expansion 
	\begin{equation}\label{eq:taylorexppi}  
		Y = Y(0) + \sum_j v^j \frac{\partial Y}{\partial v^j}(0) + \frac{1}{2} \sum_{j, l} v^j v^l \frac{\partial^2 Y}{\partial v^j \partial v^l}(0) + O(\|v\|^3).
	\end{equation}
	Let $\bar{\nabla}$ be the flat connection in $\mathbb{R}^d$. The first-order expansion in \eqref{eq:piY} for $Y$ is
	\begin{equation}
		\frac{\partial Y}{\partial v^j} = \bar{\nabla}_j(Y) \stackrel{\eqref{eq:gausseq}}{=} \nabla_j(Y) + B(\p_j, Y). \label{eq:1piY}
	\end{equation}
	The second-order expansion of $Y$ is 
	\begin{align} 
		\frac{\partial^2 Y}{\partial v^j \partial v^l} &= \bar{\nabla}_j(\nabla_l Y + B(\p_l, Y)) \nonumber \\
		&\stackrel{\eqref{eq:gausseq}}{=} \nabla_j \nabla_l Y + B(\p_j, \nabla_l Y) + \bar{\nabla}_j \big(B(\p_l, Y)\big) \nonumber \\
		&\stackrel{\eqref{eq:weingarten}}{=} \nabla_j \nabla_l Y + B(\p_j, \nabla_l Y) - A_{B(\p_l, Y)}\p_j + \nabla_j^\perp \big(B(\p_l, Y)\big). \label{eq:TpiY2}
	\end{align}	
	Now we derive \eqref{eq:piY} from \eqref{eq:taylorexppi}, \eqref{eq:1piY}, and \eqref{eq:TpiY2}, taking into account that the zero-th order of the expansion in the right-hand side of \eqref{eq:taylorexppi} is $Y(0)$.
\end{proof} 

\section{Deformed Hodge Laplacians and their uniform convergence}\label{sec:modified} 

In this section, we introduce a family of deformed Hodge Laplacians $\Delta^k_t$, $t \in \R_{> 0}$, and show the pointwise convergence of $\Delta^k_t$ to the continuous Hodge Laplacian $\Delta^k$ as $t \to 0^+$ (Theorem \ref{thm:pwconvergencel}). Then we establish the convergence in $C^0$-norm of $\Delta^k_t$ to $\Delta^k$ as $t \to 0^+$ (Theorem \ref{thm:analytic_limit}).

Let $M^n \subset \R^d$ be a compact smooth Riemannian submanifold and $0	 \le k \le n$. For $x \in M^n$, we identify $T_x M^n \subset T_x \R^d$ as a subspace in $\R^d$ via the canonical splitting $T\R^d = \R^d \times \R^d$. Let 
\begin{equation}
\Pi_x: \Lambda^k \R^d \to \Lambda^k T_x M^n\label{eq:orthoproj}
\end{equation}
denote the orthogonal projection operator. Denote by
\begin{equation}
\Pi_x^*: \Lambda^k T_x^* M^n \to \Lambda^k (\R^d)^*\label{eq:pullback}
\end{equation}
its adjoint operator. Denote by
\begin{equation}
	R_x^*: \Lambda^k (\R^d)^* \to \Lambda^k T_x^* M^n\label{eq:restr}
\end{equation}
the restriction operator, and by 
\begin{equation}
	i_x: \Lambda^k T_x M^n \to \Lambda^k \R^d\label{eq:inclusion}
\end{equation}
its adjoint, the inclusion operator.

For $v \in T_x M^n$, let $v \wedge: \Lambda^k T_x M^n \to \Lambda^{k+1} T_x M^n$ denote the exterior product with $v$. Denote by $i_v: \Lambda^{k+1} T_x^* M \to \Lambda^k T_x^* M$ the adjoint of $v \wedge$. For $w^* \in T_x^* M^n$, let $w^* \wedge: \Lambda^k T_x^* M^n \to \Lambda^{k+1} T_x^* M^n$ denote the exterior product with $w^*$. Denote by $i_{w^*}: \Lambda^{k+1} T_x M^n \to \Lambda^k T_x M$ the adjoint of $w^* \wedge$.

Recall that $H$ and $B$ denote the mean curvature and the second fundamental form of $(M^n, g)$, respectively. We define a section $\End_H(B) \in \Gamma(\End \Lambda^* T^*M^n)$ as follows. For $\om(x) \in \Lambda^k T^*_x M^n$, we set
\begin{equation}
	\End_H(B) \om(x) \coloneqq \sum_{j,l} \la H, B(e_j, e_l) \ra e_j^* \wedge i_{e_l} \omega(x), \label{eq:endh}
\end{equation}
where $\{e_i\}_{i=1}^n$ is an orthonormal basis of $T_xM ^n$, and $\{e_i^*\}_{i=1}^n$ is its dual basis. 

Under the metric identification $\Lambda^kT^*M^n\simeq\Lambda^kTM^n$, the full Weitzenb\"ock potential is given by Jost \cite{Jost2017}*{Theorem 4.3.3} as
\begin{equation}
	\mathcal{R}_k\omega \coloneqq \sum_{i,j} e_i \wedge i_{e_j^*} (R(e_i, e_j)\omega).
	\label{eq:weitzenboeck}
\end{equation}

To account for the curvature artifact arising from the ambient Euclidean space via the Gauss equation, we define the partial trace operator:
\begin{equation}
	\mathcal{R}_k^{(1)}\omega \coloneqq \sum_{j,l,p} R(e_j,e_p,e_j,e_l) e_p\wedge i_{e_l^*}\omega.
	\label{eq:weitzenboecko}
\end{equation}

\begin{theorem}[\bf Pointwise convergence]
	\label{thm:pwconvergencel}
	Let $M^n\subset\R^d$ be a compact smooth Riemannian submanifold, and let $0\leq k\leq n$. Define the deformed Hodge operator $\Delta_t^k:\Omega^k(M^n)\to\Omega^k(M^n)$ by
	\begin{align}
		\Delta_t^k\omega(x)
		&\coloneqq
		R_x^*\left(
		\frac1t\int_M
		\Phi_t(x,y)
		\bigl(\Pi_x^*\omega(x)-\Pi_y^*\omega(y)\bigr)
		\,d\vol_g(y)
		\right)
		\nonumber\\
		&\qquad
		-\End_H(B)\omega(x)-\mathcal{R}_k^{(1)}\omega(x) - \mathcal{R}_k\omega(x),
		\label{eq:hodgeambient}
	\end{align}
	where $\Phi_t$ is defined by \eqref{eq:phit}. 
	Then, for every $x\in M^n$ and every smooth $k$-form $\omega$,
	\[
	\lim_{t\to 0^+}\Delta_t^k\omega(x)
	=
	\Delta^k\omega(x),
	\]
	where $\Delta^k$ is the classical Hodge Laplacian.
\end{theorem}

The proof of Theorem \ref{thm:pwconvergencel} shall be given after the proof of Lemma \ref{lem:projsecond}. For computational simplicity, we shall prove the adjoint version of Theorem \ref{thm:pwconvergencel} for $k$-vector fields throughout, identifying $\Lambda^k T^*M^n \cong \Lambda^k TM^n$ via the Riemannian metric. This is justified since $\Delta^k$ is self-adjoint and 
\begin{equation}
	(e_j^* \wedge i_{e_l})^* = e_l \wedge i_{e_j^*}. \label{eq:muladjoint}
\end{equation}
Throughout this section, using the Riemannian metric $g$, we identify $T^*M^n$ with $TM^n$, and therefore identify differential $k$-forms with sections of $\Lambda^k TM^n$. 

\begin{remark}[\bf Extrinsic Gaussian kernel vs. the intrinsic one]\label{rem:extint}
	Fix $x\in M^n$ and identify $x$ with the origin in $T_xM^n$. Let
	$y=\exp_x(v)$, where $v\in T_xM^n\cong\R^n$. Lemma~\ref{lem:intrextr}
	gives the basic estimate
	\[
	\|y-x\|_{\R^d}^2=\|v\|^2+O(\|v\|^4).
	\]
	For the $O(t)$ expansion below, we also use the parity-refined
	Euclidean expansion established in the proof of the cited lemma:
	\[
	\|\exp_x(v)-x\|_{\R^d}^2
	=
	\|v\|^2-\frac1{12}\|B_x(v,v)\|^2+O(\|v\|^5).
	\]
	After setting $v=\sqrt t\,u$, this yields
	\begin{align}
	\Phi_t\bigl(x,\exp_x(\sqrt t\,u)\bigr)
	&=
	\frac{e^{-\|u\|^2/4}}{(4\pi t)^{n/2}}
	\Bigl(
	1+\frac{t}{48}\|B_x(u,u)\|^2 \nonumber\\
	&\hspace{38mm}
	+O\bigl(t^{3/2}\|u\|^5+t^2\|u\|^8\bigr)
	\Bigr).
	\label{eq:Phitapp}
	\end{align}
	The quartic correction is even in $u$. Hence its product with the
	leading first-order, odd Taylor term of the section integrates to zero;
	the $O(t^{3/2}\|u\|^5)$ remainder contributes only $O(t)$ after the
	prefactor $t^{-1}$ is taken into account. This parity information is
	used in the proof of Theorem~\ref{thm:analytic_limit}.
\end{remark}

For notational simplicity, we write $\omega(y)$ instead of $i_y \om(y)$; i.e., we treat $\omega(y)$ as a vector-valued function in $\mathbb{R}^N$ where $N = \binom{d}{k}$. We expand $\omega(y)$ around $x=0$:
\begin{equation}
	\omega(y) = \omega(0) + v^j \partial_j \omega(0) + \frac{1}{2} v^j v^l \partial_j \partial_l \omega(0) + O(\|v\|^3), \label{eq:omexp}
\end{equation}
where $\partial_j \coloneqq \partial/\partial v^j$. We define the operator $B: T_x M^n \times \Lambda^k T_x M^n \to \Lambda^k T_x \R^d$ as follows:
\begin{equation}
	B(v, \om) \coloneqq \sum_{p=1}^n B(v, e_p) \wedge i_{e_p^*}\om. \label{eq:Bk}
\end{equation}
(This is the natural derivation extension of $B$ to $\Lambda^k TM^n$).

\begin{lemma}\label{lem:projsecond} 
	Let $\om = i_*\om$ be a $k$-vector field on $M^n$ regarded as an $\R^N$-valued function on $M^n$. Then we have
	\begin{equation}
		\frac{\partial \om}{\partial v^j} = \nabla_j \om + B(e_j, \om), \label{eq:taylor1k}
	\end{equation}
	\begin{equation}
		\Pi_x \Big( \sum_j \frac{\partial^2 \om}{\partial v^j \partial v^j} \Big) = \sum_j \Big( \nabla_j \nabla_j \om - \mathbf{Q}_{j, j}\om \Big), \label{eq:projsecondorder}
	\end{equation}
	where  
	\begin{equation}
		\mathbf{Q}_{j, p} \om \coloneqq A_{B(e_j, \om)} e_p \coloneqq \sum_{l=1}^n A_{B(e_j, e_l)} e_p \wedge i_{e_l^*}\om. \label{eq:weingartenk}
	\end{equation}
\end{lemma}

\begin{proof} 
	1) The first assertion \eqref{eq:taylor1k} for the case $k = 1$ follows from \eqref{eq:1piY}, noting that $\partial/\partial v^j$ and $\nabla_j$ act on $\Gamma(M^n, \Lambda^k \R^d)$ and $\Gamma(M^n, \Lambda^k TM^n)$ respectively, and $B(v, e_p) \wedge i_{e_p^*}$ acts on $\Lambda^k T_x M^n$ as a derivation.
	
	2)  For $k=1$, equation \eqref{eq:TpiY2} gives, after taking the tangential projection and evaluating at the center of the Riemannian normal coordinates,
	\[
	\Pi_x\left(
	\frac{\partial^2Y}{\partial v^j\partial v^j}
	\right)
	=
	\nabla_j\nabla_jY-A_{B(e_j,Y)}e_j.
	\]
	Both the ambient and intrinsic connections extend as derivations to exterior powers. Hence, for a $k$-vector field $\omega$,
	\[
	\Pi_x\left(
	\frac{\partial^2\omega}{\partial v^j\partial v^j}
	\right)
	=
	\nabla_j\nabla_j\omega-\mathbf Q_{j,j}\omega.
	\]
	Summing over $j$ proves \eqref{eq:projsecondorder}.
\end{proof}

\begin{proof}[Proof of Theorem \ref{thm:pwconvergencel}] 
	Set
	\begin{equation}
		\Lb_t \coloneqq \Delta_t^k + \End_H(B) + \mathcal{R}_k^{(1)} + \mathcal{R}_k. \label{eq:lt}
	\end{equation}
	For $t > 0$, let $u \coloneqq v/\sqrt{t}$. By \eqref{eq:gray}, we have
	\begin{equation}
		d\vol_g \Big(\exp_x	 (\sqrt{t}u )\Big)  =  t^{n/2} \Big (1 - \frac{t}{6} \Ric_{ij}(x)u^i u^j + O(t^{3/2}\|u\|^3)\Big) du. \label{eq:taylorvolume}
	\end{equation}
	By \eqref{eq:hodgeambient} and \eqref{eq:omexp}, taking into account \eqref{eq:taylorvolume}, Remark \ref{rem:extint}, and Lemma \ref{lem:compact-exponential-open-subset}, we have:
	\begin{align}
		\Lb_t \om(x) &\stackrel{\eqref{eq:gray}}{=} - \Pi_x \Big( \frac{1}{t(4\pi)^{n/2}} \sum_{j, l} \int_{T_xM} e^{-\frac{\|u\|^2}{4}} \big(1 + O(t\|u\|^4)\big) \frac{t}{2} u^j u^l \partial_j \partial_l \om(x) \nonumber \\
		&\quad \times \big( 1 - \frac{t}{6} \Ric_{ij}(x)u^i u^j + O(t^{3/2}\|u\|^3) \big) du \Big) \label{eq:piyric} \\
		&= -\Pi_x \Big( \frac{1}{(4\pi)^{n/2}} \sum_{j} \int_{\R^n} e^{-\frac{\|u\|^2}{4}} \frac{1}{2} u^j u^j \partial_j \partial_j \om(x) du \Big) + O(t), \label{eq:piyg}
	\end{align}	
	since the Gaussian integral of the first-order term in the Taylor expansion \eqref{eq:omexp} of $\om$ vanishes by symmetry, and
	\[\int_{\R^n} u^j u^l e^{-\frac{\|u\|^2}{4}} du = 0 \quad \text{if } j \neq l.\]
	
	Using the Taylor expansion of $\omega$ up to order four, and using the vanishing of the odd Gaussian moments, the contribution of the third-order term is zero and the first nonzero remainder is of order $t$. Thus, the integral simplifies to:
	\begin{equation}
		\Lb_t \omega(x) = -\Pi_x \left( \sum_j \partial_j \partial_j \omega(x) \right) + O(t). \label{eq:hodgeambientt}
	\end{equation}
	Using \eqref{eq:projsecondorder}, we have
	\begin{equation}
		-\Pi_x \left( \sum_j \partial_j \partial_j \omega(x) \right) = -\sum_j \nabla_j \nabla_j \omega(x) + \sum_j \mathbf{Q}_{j,j}\omega(x). \label{eq:integralexp}
	\end{equation}
	The first term in the right-hand side of \eqref{eq:integralexp} relates to the Connection Laplacian $\nabla^* \nabla \omega$.
	
	For $\omega \in \Lambda^k T_xM^n$, the total drift $\mathbf{Q} \omega$ is defined by:
	\begin{equation}
		\mathbf{Q} \omega \coloneqq \sum_{j=1}^n \mathbf{Q}_{j,j} \omega \stackrel{\eqref{eq:weingartenk}}{=} \sum_{j=1}^n \sum_{l=1}^n A_{B(e_j, e_l)}(e_j) \wedge i_{e^*_l} \om. \label{eq:totaldrift}
	\end{equation}
	
	Using the property \eqref{eq:shapeoperator} of the shape operator, $\langle A_{\xi}(X), Z \rangle = \langle B(X, Z), \xi \rangle$, we express the vector $A_{B(e_j, e_l)}(e_j)$ in the tangent basis $\{e_p\}$ as:
	\begin{equation}
		A_{B(e_j, e_l)}(e_j) = \sum_{p=1}^n \langle B(e_j, e_p), B(e_j, e_l) \rangle e_p.
	\end{equation}
	Substituting this into the drift summation \eqref{eq:totaldrift}:
	\begin{equation}
		\mathbf{Q} \omega = \sum_{j, l, p} \langle B(e_j, e_p), B(e_j, e_l) \rangle e_p \wedge i_{e_l^*}\omega. \label{eq:totaldrift1}
	\end{equation}
	
	As derived in \eqref{eq:integralexp}, taking into account \eqref{eq:totaldrift} and \eqref{eq:totaldrift1}, the projection of the ambient Hessian is:
	\begin{equation}
		\Pi_x \Big( \sum_j \partial_j \partial_j \om \Big) = \sum_j \nabla_j \nabla_j \om - \mathbf{Q}\om(x). \label{eq:term2}
	\end{equation}
	
	\begin{lemma}\label{lem:secondterm}
		We have
		\begin{equation}
			\Pi_x \Big( \sum_j \partial_j \partial_j \om \Big) = \sum_j \nabla_j \nabla_j \om - \mathcal{R}^{(1)}_k \om - \End_H(B) \om. \label{eq:term2a}
		\end{equation}
	\end{lemma}
	
	\begin{proof}[Proof of Lemma \ref{lem:secondterm}]
		Using the Gauss Equation \eqref{eq:gaussflat}:
		\[ R(e_j, e_p, e_j, e_l) =  \langle B(e_j, e_l), B(e_p, e_j)\rangle  - \langle B(e_j, e_j), B(e_p, e_l)\rangle, \]
		we rearrange for the $B \cdot B$ product in the right-hand side of \eqref{eq:totaldrift1}:
		\begin{equation}
			\sum_j \langle B(e_j, e_p), B(e_j, e_l) \rangle = \underbrace{\sum_j \langle B(e_j, e_j), B(e_p, e_l) \rangle}_{\langle H, B(e_p, e_l) \rangle} + \underbrace{\sum_j R(e_j, e_p, e_j, e_l)}_{\text{Curvature Term}}. \label{eq:BB}
		\end{equation}
		
		Substituting \eqref{eq:BB} back into the right-hand side of \eqref{eq:totaldrift1}, taking into account \eqref{eq:weitzenboecko}, the defining equation \eqref{eq:endh} of $\End_H(B)$, and noting that $H = \sum_j B(e_j, e_j)$, we obtain \eqref{eq:term2a} from \eqref{eq:term2}. 
		This completes the proof of Lemma \ref{lem:secondterm}.
	\end{proof}
	
	Now we derive Theorem \ref{thm:pwconvergencel}. From \eqref{eq:hodgeambientt} and Lemma \ref{lem:secondterm}, taking the limit as $t \to 0^+$, the integral operator converges exactly to:
	\[ \lim_{t \to 0^+} \Lb_t \omega(x) = -\sum_j \nabla_j \nabla_j \omega(x) + \mathcal{R}_k^{(1)}\omega(x) + \End_H(B)\omega(x). \]
	Recall from \eqref{eq:lt} that $\Delta_t^k \omega(x) = \Lb_t \omega(x) - \End_H(B)\omega(x) - \mathcal{R}_k^{(1)}\omega(x) - \mathcal{R}_k\omega(x)$. Substituting the limit of $\Lb_t \omega(x)$ yields:
	\[ \lim_{t \to 0^+} \Delta_t^k \omega(x) = \nabla^* \nabla \omega(x) - \mathcal{R}_k(\omega(x)). \]
	By the classical Weitzenb\"ock formula (see, e.g., \cite{Jost2017}*{Theorem 4.3.3}), this establishes:
	\[ \lim_{t \to 0^+} \Delta_t^k \omega(x) = \Delta^k \omega(x). \]
\end{proof}

\begin{theorem}[\bf Uniform convergence of deformed operators]
	\label{thm:analytic_limit}
	Let $M^n$ be a compact smooth Riemannian submanifold of $\mathbb{R}^d$
	and let $0\le k\le n$. There exist constants $t_0>0$ and
	$C(M^n)>0$ such that, for every
	$\omega\in C^4(\Lambda^k TM^n)$ and every $0<t<t_0$,
	\begin{equation}
		\|\Delta_t^k\omega-\Delta^k\omega\|_{C^0(M^n)}
		\le
		C(M^n)t\|\omega\|_{C^4(M^n)}.
		\label{eq:limnormc}
	\end{equation}
\end{theorem}

\begin{proof}
	By the Taylor expansion of the Gaussian integral in \eqref{eq:hodgeambientt}, and taking into account the curvature substitution from Lemma \ref{lem:secondterm}, the pointwise error satisfies:
	\begin{align}
		\|\Delta^k_t \om(x) - \Delta^k \om(x)\| \le C(x, \om) t \label{eq:limnormca}
	\end{align}		
	for sufficiently small $t$. 
	
	The constant $C(x, \om)$ encapsulates the remainder terms of the expansion. By \eqref{eq:piyric}, \eqref{eq:Phitapp}, and Lemma \ref{lem:intrextr}, the odd-order terms in the Gaussian integral vanish by symmetry. The first non-vanishing remainder arises from the fourth-order derivatives in the Taylor expansion \eqref{eq:omexp} integrating against the fourth moments of the Gaussian, which scale exactly as $\mathcal{O}(t)$. Consequently, the remainder is strictly controlled by the fourth spatial derivatives of $\omega$, yielding:
	\begin{equation}
		C(x, \om) \le C_1(x) \| \om \|_{C^4} \label{eq:limnormcb}
	\end{equation}   
	for a strictly positive continuous function $C_1(x)$ depending only on the local geometry of $M^n$. Since $M^n$ is compact, $C_1(x)$ achieves a uniform global maximum $C(M^n) = \sup_{x \in M^n} C_1(x) < \infty$. This establishes the uniform bound \eqref{eq:limnormc} and completes the proof.
\end{proof}

\begin{remark}[\bf Higher norm convergence]\label{rem:higher}
	The preceding theorem is stated only in the $C^0$-norm, which is the
	form needed in the empirical convergence arguments below. By
	differentiating the same local small-time expansion in Riemannian
	normal coordinates, one obtains analogous $C^l$-estimates. More
	precisely, for every integer $l\ge0$ there exist constants
	$t_l>0$ and $C_l(M^n)>0$ such that
	\[
	\|\Delta_t^k\omega-\Delta^k\omega\|_{C^l(M^n)}
	\le
	C_l(M^n)t\|\omega\|_{C^{l+4}(M^n)}
	\]
	for all $\omega\in C^{l+4}(\Lambda^k TM^n)$ and all $0<t<t_l$.
	We shall not use these higher norm estimates in the sequel.
\end{remark}

\section{Empirical Hodge Laplacians and their spectral convergence}\label{sec:empi}

Let $(M^n,g)\subset\mathbb R^d$ be a closed, oriented,
$n$-dimensional $C^3$-smooth Riemannian submanifold, where   $n \ge 3$, endowed with
the induced metric $g$. Denote by $\mu$ the uniform probability
measure on $M^n$, that is,
\[
d\mu=\frac{d\vol_g}{\vol_g(M^n)}.
\]
The tangent, second-fundamental-form, and curvature estimators
constructed below require only $C^3$-regularity. For
Theorem~\ref{thm:empiconvergence} and the harmonic-cluster convergence
result of Theorem~\ref{thm:harmonic_cluster}, we additionally assume
that $M^n$ is $C^4$-smooth; for the latter theorem, we also assume
that $M^n$ is connected. 
We define the empirical Hodge operators
\[
\widehat\Delta^k_{t,S_m}
\]
and prove their uniform consistency on $C^4$-smooth differential
forms in Theorem~\ref{thm:empiconvergence}.
Finally, we establish compact Mosco convergence of the corresponding empirical
quadratic forms and deduce convergence in probability of the
empirical harmonic spectral cluster; see
Theorem~\ref{thm:harmonic_cluster}.

 Throughout this section, $S_m=\{x_1,\ldots,x_m\}\sim\mu^m$ is an
i.i.d. sample. Unless otherwise stated, in all asymptotic convergence
results we use the scaling
\[
t=t_m=m^{-\frac{1}{2n}}.
\]
In particular, $t_m\to0$ as $m\to\infty$. The empirical
estimators and operators introduced below are nevertheless defined for
every $t>0$.

Note that the dimension of a compact submanifold $M^n \subset \R^d$ can be estimated directly from a uniformly sampled point cloud $S_m$ \cite{SW2012}*{Section 2, p.7}, so we assume in this section that the intrinsic dimension $n$ is known. Singer and Wu also proposed an algorithm to detect the orientability of $M^n$ from finite point data sets \cite{SW2011}, so we assume that $M^n$ is oriented.

Throughout this section, the symbols $C, C_1, C_2, \ldots$ denote positive constants that may depend on the intrinsic dimension $n$, the ambient dimension $d$, and the geometry of $M^n$ (e.g., the reach $\tau_{M^n}$), but are independent of the sample size $m$ and the deformation parameter $t \in \R^+$.

\subsection{Empirical projection $(\hat{\Pi}_{t,S_m})_x$}\label{subs:eproj}

Estimating the orthogonal projection $\Pi_x : \mathbb{R}^d \to T_x M^n$ is equivalent to estimating the tangent space $T_x M^n$. 
Let $S_m = (x_1,\ldots,x_m) \in (M^n)^m$.
For $p \in M^n$ and $\delta > 0$, denote by $D_\delta(p) \subset M^n$ the extrinsic ball of radius $\delta$ centered at $p$. 
Following \cite{AL2019}*{Section 3.1}, we define the local covariance matrix at $p \in M^n$:
\begin{equation}
	\Sigma_{t,S_m}(p) \coloneqq \frac{1}{m} \sum_{j=1}^m \Phi_t(p,x_j) (x_j-p)(x_j-p)^\top \chi _\delta(p, x_j), \label{eq:LPCA_clean}
\end{equation}
where $\Phi_t$ is defined in \eqref{eq:phit}: 
\begin{equation*}
	\Phi_t(x, y) \coloneqq \frac{1}{(4\pi t)^{n/2}} \exp\left( -\frac{\|x-y\|_{\R^d}^2}{4t} \right)
\end{equation*}
and 		the empirical projection
		\begin{equation}
			(\hat{\Pi}_{t,S_m})_p : \R^d \to \R^d \label{eq:empiproj_clean}
		\end{equation}
		is defined as the orthogonal projection onto the span of the top $n$ eigenvectors of $\Sigma_{t,S_m}(p)$. Note that we regard both the empirical projection $(\hat{\Pi}_{t,S_m})_p$ and the true projection $\Pi_p: \R^d \to T_p M^n$ as linear operators from $\R^d$ to $\R^d$.
		
		\medskip
		
		\begin{remark}[\bf Gaussian vs.\ compactly supported kernels]  \label{rem:empitangent}
			In the local PCA literature, compactly supported kernels localized to a radius $h \asymp \sqrt{t}$ are standard. While the Gaussian kernel $\Phi_t$ shares this characteristic scale, its moments are defined by its infinite tails. If one were to truncate the Gaussian kernel precisely at a shrinking radius $D_{\sqrt{t}}(p)$, the truncation would alter the kernel's higher-order moments, introducing non-negligible bias into the second-order expansions required for curvature estimation.
			
			To resolve this, our construction of $\Sigma_{t, S_m}$ utilizes a \emph{fixed} radius $\delta > 0$ that is independent of $t$. Because the Gaussian tail decays exponentially as $\exp(-\delta^2 / 4t)$, the truncation error is $o(t^k)$ for any integer $k \ge 1$. This implies the kernel ``self-localizes'':
			\begin{itemize}
				\item The local covariance matrix captures the full un-truncated Gaussian moments up to exponentially small corrections, avoiding truncation bias in the expected values.
				\item The effective region contributing to the covariance remains concentrated in an $O(\sqrt{t})$-neighborhood, preserving the concentration rates.
			\end{itemize}
		\end{remark}
		
		\medskip

	\begin{proposition}[\bf Properties of empirical projections]
		\label{prop:empitangent_clean}
		Let $M^n\subset\mathbb R^d$ be a compact $C^3$-smooth submanifold of
		dimension $n\ge2$ with reach $\tau_M>0$, and let
		$\delta\in(0,\tau_{M^n}/4)$ be fixed. There exist constants
		$t_0>0$, $C_0>0$, and $C>0$, depending only on the geometry of $M^n$,
		such that the following holds.
		
		Let $S_m=(x_1,\ldots,x_m)$ be an i.i.d. sample from the uniform
		probability measure $\mu$ on $M^n$. Suppose that $0<t<t_0$ and
		\begin{equation}
			t=m^{-1/(2n)}.
			\label{eq:optimal_scaling}
		\end{equation}
		Then, with probability at least $1-m^{-2/n}$,
		\begin{equation}
			\sup_{p\in M}
			\|(\hat\Pi_{t,S_m})_p-\Pi_p\|_{\mathrm{op}}
			\le Ct.
			\label{eq:empoj_clean}
		\end{equation}
		On the same high-probability event, the map
		\[
		p\longmapsto (\hat\Pi_{t,S_m})_p
		\]
		is continuous on $M^n$.
		
		Assume in addition that $n\ge3$.
		Then, on an event of probability at least \(1-2m^{-2/n}\), the following Lipschitz
		transition estimate also holds: for all sufficiently small $t$, there
		exists a constant $C'>0$, depending only on the geometry of $M^n$, such
		that
		\begin{equation}
			\Big\|
			\big(\hat\Pi_p\hat\Pi_y-\Pi_p\Pi_y\big)
			-
			\big(\hat\Pi_p^2-\Pi_p^2\big)
			\Big\|_{\mathrm{op}}
			\le
			C't\,\|y-p\|_{\mathbb R^d}
			\label{eq:empirical_lipschitz_transition}
		\end{equation}
		uniformly in $p,y\in M^n$, where
		\[
		\hat\Pi_q\coloneqq(\hat\Pi_{t,S_m})_q.
		\]
	\end{proposition}

	\begin{proof} 
		Our proof follows \cite{AL2019} but is self-contained because the fixed-radius truncation differs slightly from their framework. Since multiplying the covariance matrix by a positive scalar does not change its eigenspaces, the normalization conventions used in \cite{AL2019} and in \eqref{eq:LPCA_clean} are equivalent for tangent space estimation. Although \cite{AL2019} employs a compactly supported kernel localized at scale $h$, our covariance matrix uses the Gaussian kernel restricted to the fixed neighborhood $D_\delta(p)$. 
		
		To guarantee that the empirical projection matrix $\hat{\Pi}_{t, S_m}$ is well-defined, we must ensure a strict spectral separation between the tangential and normal subspaces of the local covariance matrix. By the local Taylor expansion of the manifold, the tangential eigenvalues scale as $O(t)$ while the normal eigenvalues, driven by the extrinsic curvature, scale as $O(t^2 \|B\|_{L^\infty}^2)$. Therefore, there exists a critical bandwidth threshold $t_0 > 0$, depending entirely on the reach and maximum curvature of $M^n$, such that for all $t < t_0$, the spectral gap $\lambda_n - \lambda_{n+1} \ge c t > 0$ is  bounded below by $ct$ uniformly in $p \in M^n$. We assume hereafter that $t < t_0$ is sufficiently small to maintain this eigengap, allowing us to apply the Davis-Kahan theorem to bound the projection error. We postpone the detailed proof to Appendix \ref{sec:empitangent}. 
	\end{proof}
		
		\begin{remark}\label{rem:scaling} 
			Note that our convergence rate is better than that in \cite{AL2019}*{Theorem 2}, as we achieve an $O(t)$-rate of convergence compared to their $O(\sqrt{t})$-rate. The trade-off is that for practical computation, their compactly supported ball $D_{\sqrt{t}}(p)$ is shrinking as $t \to 0^+$, whereas ours requires integrating over a fixed radius. Furthermore, our condition \eqref{eq:optimal_scaling}  differs from the condition that $t^{n/2} \asymp \frac{\log m}{m}$ in \cite{AL2019}.
		\end{remark}   
		
\

We identify $\mathbb{R}^d$ with $(\mathbb{R}^d)^*$ via the Euclidean metric, and therefore identify the restriction operator $R_x^*$ with $\Pi_x$.

\medskip

\begin{corollary}[Consistency and continuity of empirical transition operators]\label{cor:empitr_clean}
	Under the assumption  \eqref{eq:optimal_scaling} of Proposition \ref{prop:empitangent_clean}, with probability at least $1 - m^{-2/n}$ over i.i.d. $S_m\sim \mu^m$, we have
	\begin{align}
		\sup_{x,y \in M}
		\|\Lambda^k\big((\hat{\Pi}_{t,S_m})_x (\hat{\Pi}_{t,S_m})_y\big ) - \Lambda^k(\Pi_x \Pi_y)\|_{\mathrm{op}}
		&\le 
		2	k C t .
		\label{eq:empoj2_clean}
	\end{align}
		
	Furthermore, on the same high-probability event, the  mapping
	\[
	\hat{\Pi}_{t, S_m}: M^n \to \End (\R^d), \quad p \mapsto (\hat{\Pi}_{t, S_m})_p
	\]
	is continuous.  
\end{corollary}
\begin{proof}  Using $\Lambda ^k (AB) = \Lambda ^k (A) \Lambda ^k (B)$  by functoriality  of exterior power, we write 
	\begin{align}
		\Lambda^k\big((\hat{\Pi}_{t,S_m})_x (\hat{\Pi}_{t,S_m})_y\big ) - \Lambda^k(\Pi_x \Pi_y)  & =  \big(\Lambda^k (\hat{\Pi}_{t,S_m})_x -\Lambda ^k \Pi_x\big)\Lambda ^k  (\hat{\Pi}_{t,S_m})_y\nonumber\\
		&+  \Lambda  ^k \Pi_x \big(\Lambda ^k  (\hat{\Pi}_{t,S_m})_y-  \Lambda ^k  \Pi_y \big).\label{eq:projkdecomp}
	\end{align}
	Using multilinearity of the exterior power,
	\[
	\|\Lambda^k A-\Lambda^k B\|_{\mathrm{op}}
	\le
	k\max(\|A\|,\|B\|)^{k-1}\|A-B\|,
	\] 
	we derive \eqref{eq:empoj2_clean}    from \eqref{eq:empoj_clean}  and \eqref{eq:projkdecomp}, taking into account
	$$ \| \Lambda^k A\|_{\mathrm{op}} \le \|A \|^k_{\mathrm{op}},$$
	$$\| \Lambda ^k  \Pi_x\|_{\mathrm{op}}  = 1, $$
	$$\|(\hat{\Pi}_{t,S_m})_x\|_{\mathrm{op}} =\|\Pi_x\|_{\mathrm{op}} =1.$$
	The continuity  statement  follows from     the similar   assertion in Proposition \ref{prop:empitangent_clean}.
\end{proof}

\subsection{The Empirical Construction of $\hat{B}$ and $\widehat\End_H(B)$}\label{subs:B}

Our construction of the empirical second fundamental form $\hat{B}$ of $M^n$ is based on the following observation.
\begin{lemma}\label{lem:B}
	Assume that $Y$ is a smooth vector field on a compact
	$C^3$-smooth submanifold $M^n\subset\R^d$.  Let $\Pi_x^\perp$ denote the projection to the normal space $(T_x M^n)^\perp \subset \R^d$.   Let $\Phi_t$ denote the extrinsic Gaussian kernel defined in \eqref{eq:phit}. Then
	\begin{equation}
\lim_{t\to0^+}
		\Pi_x\otimes\Pi_x^\perp
		\left(
		\frac1{2t}
		\int_M
		\Phi_t(x,y)(y-x)\otimes Y(y)\,	d\vol_g(y)
		\right)
		=		
		\sum_{j=1}^ne_j\otimes B(e_j,Y(x)).
		\label{eq:B}
	\end{equation}
\end{lemma}

\begin{proof} 
	As in the proof of Theorem~\ref{thm:pwconvergencel}, using the Taylor
	expansion \eqref{eq:piY}, we obtain
	\begin{align*}
	&\Pi_x\otimes\Pi_x^\perp
	\left(
	\frac1t\int_{M^n}\Phi_t(x,y)(y-x)\otimes Y(y)\,d\vol_g(y)
	\right)\\
	&\qquad=
	2\sum_{j=1}^n e_j\otimes B(e_j,Y(x))+O(\sqrt t).
	\end{align*}
	Taking the limit as $t\to0$ completes the proof.
\end{proof}

For $v \in \R^d$, we define the contraction operator
$$ v^\#: \R^d \otimes \R^d \to \R^d, \qquad (w_1 \otimes w_2) \mapsto \la v, w_1 \ra w_2. $$
The normal-coordinate expansion used in the proof of Lemma~\ref{lem:B}, together with the corresponding Gaussian-moment estimates, gives, uniformly for $p\in M^n$ and $v,w\in T_pM^n$,
\[
\mathcal B_t(p)(v,w)
=
B_p(v,w)+R_t(p;v,w),
\]
where
\begin{equation}
	\|R_t(p;v,w)\|
	\le
	C\sqrt t\,\|v\|\,\|w\|	\label{eq:B_uniform_bias}
\end{equation}
 and $C$  is  a constant  depending only on  $M^n$. \footnote{Simplifying bookkeeping,  we   denote  this constant  by $C$   although we used  $C$ in \eqref{eq:empoj_clean}.}
Indeed, since $M^n$ is compact and $C^3$-smooth, the local coordinate charts, the coefficients occurring in the Taylor expansions, and their remainders can be controlled uniformly in $p$; the contribution from the complement of a fixed normal neighborhood is exponentially small in $t^{-1}$.

Since $B_p$ is symmetric, symmetrization does not change the limit,
and hence
\[
\sup_{p\in M^n}
\|\mathcal B_t^{\mathrm{sym}}(p)-B_p\|_{\mathrm{op}}
\le C\sqrt t.
\]
The same estimates hold for the ambient extensions:
\begin{equation}
	\sup_{p\in M^n}
	\|\widetilde{\mathcal B}_t(p)-\widetilde B_p\|_{\mathrm{op}}
	+
	\sup_{p\in M^n}
	\|\widetilde{\mathcal B}^{\mathrm{sym}}_t(p)-\widetilde B_p\|_{\mathrm{op}}
	\le C\sqrt t.
	\label{eq:B_extended_uniform_bias}
\end{equation}

\begin{corollary}[\bf Deformation of the second fundamental form]\label{cor:B}
	For $t \in \R_{+}$, let $\mathcal{B}_t(x): T_x M^n \times T_x M^n \to (T_x M^n)^\perp$ be the linear operator defined by
	\begin{align}
\mathcal B_t(x)(v,w)
&\coloneqq
\Pi_x^\perp v^\#\Bigg(
\frac{\vol_g(M^n)}{2t}
\int_{M^n}\Phi_t(x,y)(y-x)\nonumber\\
&\hspace{35mm}\otimes\Pi_y(i_xw)\,
\chi_\delta(x,y)\,d\mu(y)
\Bigg).
\label{eq:bintegral}
\end{align}
	where $i_x: T_x M^n \to \R^d$ is the canonical inclusion mapping.  
	Then we have
	\begin{equation}
		\Bb_t(v, w) = B(v, w) + O(\sqrt t). \label{eq:corB}
	\end{equation}
\end{corollary}

\medskip

For notational simplicity, we shall omit $i_x$ in the formulas below, identifying a vector $v \in T_x M^n$ with its image $i_x v$ in $\R^d$. 

Based on Corollary \ref{cor:B}, we define the empirical second fundamental form $\hat{B}$ and its symmetrization $\hat B ^{sym}$ at any point $p \in M^n$ as follows. For $t \in \R_{+}$ and a point cloud $S_m = \{x_1, \dots, x_m\} \subset M^n$, recall that the empirical orthonormal basis $\{(\hat{e}_i)_{t, S_m}(p)\}_{i=1}^n$ consists of the top $n$ eigenvectors of the covariance matrix $\Sigma_{t, S_m}(p)$ defined in \eqref{eq:LPCA_clean}. We set
For brevity, write
\[
\hat e_i(p)\coloneqq(\hat e_i)_{t,S_m}(p).
\]
Then
\begin{align}
(\hat B_{t,S_m})_p(\hat e_i(p),\hat e_k(p))
&\coloneqq
(\hat\Pi_{t,S_m})_p^\perp\circ\hat e_i(p)^\#
\Bigg[
\frac{\vol_g(M^n)}{2mt}
\sum_{j=1}^m\Phi_t(p,x_j)(x_j-p)\nonumber\\
&\hspace{34mm}\otimes
\hat\Pi_{x_j}\hat e_k(p)\,\chi_\delta(p,x_j)
\Bigg],
\label{eq:empiB}
\end{align}
and
\begin{equation}
(\hat B^{\mathrm{sym}}_{t,S_m})_p(\hat e_i(p),\hat e_k(p))
\coloneqq
\frac12\Bigl(
(\hat B_{t,S_m})_p(\hat e_i(p),\hat e_k(p))
+
(\hat B_{t,S_m})_p(\hat e_k(p),\hat e_i(p))
\Bigr).
\label{eq:empiBs}
\end{equation}
We also symmetrize
	\begin{equation}
	\Bb^{sym}_t(v, w) \coloneqq \frac{1}{2}\big (\Bb_t(v, w) + \Bb_t (w, v)\big) . \label{eq:corBs}
\end{equation}

For $p \in M^n$, $t \in \R_{+}$, and $S_m \subset M^n$, we extend $\Bb_t(p)$, $B(p)$,  $(\hat{B}_{t, S_m})_p$,  and their symmetrized versions to act as linear operators from $\R^d \times \R^d \to \R^d$ as follows. For $u, v \in \R^d$, we set
\begin{equation}
	\tilde{B}_p(u, v) \coloneqq B_p(\Pi_p u, \Pi_p v), \label{eq:Bext}
\end{equation}
\begin{equation}
	\tilde{\Bb}_t(p)(u, v) \coloneqq \Pi_p^\perp u^\# \frac{\vol_g(M^n)}{2t} \int_{M^n} \Phi_t(p, y)(y-p) \otimes \Pi_y \Pi_p(v)\chi_\delta (p, y) d\mu(y), \label{eq:Btext}
\end{equation} 

\begin{equation}
	\tilde{\Bb}^{sym}_t(p)(u, v) \coloneqq \frac{1}{2}\Big ( \tilde{\Bb}_t(p)(u, v) + \tilde{\Bb}_t (p)(v, u)\Big). \label{eq:Btexts}
\end{equation} 

\begin{align} 
	\widetilde{(\hat{B}_{t, S_m})}_p(u, v) &\coloneqq (\hat{\Pi}_{t, S_m})^\perp_p u^\# \Big[ \frac{\vol_g(M^n)}{2mt} \sum_{j=1}^m \Phi_t(p, x_j) (x_j - p)\nonumber\\
	& \otimes (\hat{\Pi}_{t, S_m})_{x_j} (\hat{\Pi}_{t, S_m})_p \chi_{\delta} (p, x_j) v \Big], \label{eq:Bempiext}
\end{align}
\begin{equation} 
	\widetilde{(\hat{B}^{sym}_{t, S_m})}_p(u, v) \coloneqq \frac{1}{2}\Big ( 	\widetilde{(\hat{B}_{t, S_m})}_p(u, v) + 	\widetilde{(\hat{B}_{t, S_m})}_p(v, u) \Big ). \label{eq:Bempiexts}
\end{equation}

Denote by $\hat{T}_p M^n$ the empirical tangent space.

\begin{lemma}\label{lem:Bextend} 
	The restriction of $\tilde{B}_p$ and $\tilde{\Bb}^{sym}_t$ to $T_p M^n \times T_p M^n$, and of $\widetilde{(\hat{B}^{sym}_{t, S_m})}_p$ to $\hat{T}_p M^n \times \hat{T}_p M^n$, is exactly equal to $B_p$, $\Bb_t ^{sym}(p)$, and $(\hat{B}^{sym}_{t, S_m})_p$, respectively. Furthermore, we have
	\begin{equation}
		\lim_{t \to 0^+} \| \tilde{\Bb}^{sym}_t(p) - \tilde B_p \|_{\mathrm{op}} = 0 
	\end{equation}
	for any $p \in M^n$. Consequently, letting 
	
$$	\tilde H^{sym}_t(p)
	\coloneqq
	\sum_{i=1}^d
	\tilde{\mathcal B}^{sym}_t(p)(e_i,e_i), $$
	then  
	$$ \lim_{t \to 0^+} \tilde{H}^{sym}_t(p) = H(p). $$
\end{lemma}

\begin{proof} 
	The first assertion of Lemma \ref{lem:Bextend} is straightforward from the definitions. The second assertion follows from the first, taking into account Lemma \ref{lem:tangentemb}. The final assertion regarding the mean curvature follows immediately from the uniform convergence of the extended fundamental form.    
\end{proof}

\medskip

\begin{proposition}[\bf The empirical second fundamental form]
	\label{prop:probempib}
	Assume that $M^n\subset\mathbb R^d$ is a compact $C^3$-smooth
	submanifold of dimension $n\ge2$ with reach $\tau_{M^n}>0$, and fix
	$\delta\in(0,{\tau_{M^n}}/4)$. Let $S_m\sim\mu^m$ be i.i.d. and set
	$t=m^{-1/(2n)}$. Then, for all sufficiently large $m$, with probability
	at least $1-3m^{-2/n}$,
	\begin{equation}
		\sup_{p\in M^n}
		\| \widetilde{(\hat B_{t,S_m})}_p-\tilde B_p\|_{\mathrm{op}}
		\le C_B\sqrt t.
		\label{eq:probempib}
	\end{equation}
where $C_B$ depends only on the geometry of $M^n$. Furthermore, on the
same high-probability event, for all sufficiently
small $t$, the maps
\[
p\longmapsto \hat\Pi_{t,S_m}(p)
\]
and
\[
p\longmapsto \widetilde{(\hat B_{t,S_m})}_p
\in \End(\mathbb R^d\times\mathbb R^d,\mathbb R^d)
\]
are continuous.
Consequently, on the same high-probability event, the map
\[
p\longmapsto \widetilde{(\hat B^{sym}_{t,S_m})}_p
\in \End(\mathbb R^d\times\mathbb R^d,\mathbb R^d)
\]
is continuous and satisfies
\[
\sup_{p\in M}
\left\|
\widetilde{(\hat B^{sym}_{t,S_m})}_p-\tilde B_p
\right\|_{\mathrm{op}}
\le C_B\sqrt t .
\]
\end{proposition}

\begin{proof}
	Work on the high-probability projector event of
	Proposition~\ref{prop:empitangent_clean}. On this event, for
	$t\le t_0$, the map
	\[
	\widetilde{(\hat B_{t,S_m})}:M^n\longrightarrow
	\mathrm{Lin}(\R^d\times\R^d,\R^d)
	\]
	is continuous.

	For each $p \in M^n$, we decompose the estimation error into an analytical bias and a stochastic fluctuation:
	\begin{equation}
		(\widetilde{\hat B_{t, S_m}})_p - \tilde B_p = \underbrace{(\widetilde{\hat B_{t, S_m}})_p - \tilde\Bb_t(p)}_{\text{Stochastic error}} + \underbrace{(\tilde\Bb_t(p) - \tilde B_p)}_{\text{Bias}}. \label{eq:error}
	\end{equation}
	
	By Lemma \ref{lem:Bextend} and  Eq. \eqref{eq:B_extended_uniform_bias}, there exists a constant $C_1 > 0$ such that for $t$ sufficiently small we have
	\begin{equation}
		\|\tilde \Bb_t(p) - \tilde B_p\|_{\mathrm{op}} \le C_1 \sqrt t. \label{eq:anbias_new}
	\end{equation}
	
	For a fixed sample $S_m = (x_1, \ldots, x_m) \in (M^n)^m$, we isolate the internal summation operators:
	\begin{align}
		\mathrm{Sum}_{t, S_m} &\coloneqq \frac{1}{2m} \sum_{j=1}^m \Phi_t(p,x_j) \frac{(x_j - p)}{t} \otimes \hat{\Pi}_{x_j}\hat{\Pi}_{p} \chi_\delta (p, x_j), \label{eq:sum}\\
		\mathrm{Sum}^{\mathrm{true}}_t &\coloneqq \frac{1}{2m} \sum_{j=1}^m \Phi_t(p,x_j) \frac{(x_j - p)}{t} \otimes \Pi_{x_j} \Pi_{p}   \chi_\delta (p,x_j). \label{eq:sumtrue}
	\end{align}
	
	We split the stochastic error of the fundamental form into three components:
	Define
\begin{align*}
E_1&\coloneqq
\|(\hat\Pi_p^\perp-\Pi_p^\perp)\mathrm{Sum}_{t,S_m}\|_{\mathrm{op}},\\
E_2&\coloneqq
\|\Pi_p^\perp(\mathrm{Sum}_{t,S_m}-\mathrm{Sum}^{\mathrm{true}}_t)\|_{\mathrm{op}},\\
E_3&\coloneqq
\left\|\Pi_p^\perp\mathrm{Sum}^{\mathrm{true}}_t
-\frac{\tilde\Bb_t(p)}{\vol_g(M^n)}\right\|_{\mathrm{op}}.
\end{align*}
Then
\begin{equation}
\|(\widetilde{\hat B_{t,S_m}})_p-\tilde\Bb_t(p)\|_{\mathrm{op}}
\le\vol_g(M^n)(E_1+E_2+E_3).
\label{eq:sterror}
\end{equation}
	\medskip
	\noindent
	\underline{Step 1.} {\it Bounding $\mathrm{Sum}^{\mathrm{true}}_t$  and $E_3$.}
	Let $F_p :  M^n \to \R^d \otimes  \End (\R^d)$ be defined  by
	$$ F_p(y) \coloneqq \Phi_t(p,y) \frac{y-p}{t} \otimes \Pi_y\Pi_p \chi_\delta(p, y).$$ Then by \eqref{eq:Btext}
	
	\begin{equation}
		\tilde \Bb_t  (p) =  \frac{\vol_g(M^n)}{2} \Pi ^\perp_p \E_\mu [F_p] ,\label{eq:tbbt}
	\end{equation}
	\begin{equation}
		\mathrm{Sum}^{\mathrm{true}}_t = \frac{1}{2}\left \{\E_{\mu}[F_p] + \left( \frac{1}{m}\sum_{j=1}^m F_p(x_j) - \E_{\mu}[F_p] \right)\right\}. \label{eq:sumtrue2}
	\end{equation}
	
	Since  $\|\Pi ^\perp_p\|_{\mathrm{op}} \le 1$, by \eqref{eq:tbbt}, \eqref{eq:sumtrue2}:	
	\begin{equation}
		E_3=  \Big\|  \Pi_p^\perp  \mathrm{Sum }_t^{\mathrm{true}}- \frac{\tilde \Bb_t (p)}{\vol_g(M^n)}\Big\|_{\mathrm{op}}\le \frac{1}{2} \left\| \frac{1}{m}\sum_{j=1}^m F_p(x_j) - \E_{\mu}[F_p] \right\|.\label{eq:lastterm}
		\end{equation}
	
	In normal coordinates $y = \exp_p(\sqrt{t}u)$, taking into account the volume distortion $d\mu(y) = (\vol_g (M^n) ^{-1}) t^{n/2}(1+O(t|u|^2))du$, the expected value evaluates to:
	\begin{equation}
		\E_\mu[F_p] = \frac1{\vol_g(M^n)}\int_{\R^n} \frac{1}{(4\pi)^{n/2}} e^{-|u|^2/4} \left( t^{-1/2}u + O(1) \right) \otimes \big(\Pi_p + O(\sqrt{t})\big) du.
	\end{equation}
	Because the leading odd term $t^{-1/2}u$ integrates to exactly zero against the symmetric Gaussian measure, the first non-vanishing contribution is bounded by a constant. Hence, 
	\begin{equation}
		\|\E_\mu[F_p]\|_{\mathrm{op}} \le C_2. \label{eq:expf}
	\end{equation}
	
	By Lemma \ref{lem:MC_kernel_vector}, the Monte Carlo error is bounded by:
	\begin{equation}
		\Big \|\frac{1}{m}\sum_{j=1}^m F_p(x_j) - \E_{\mu}[F_p] \Big \| = O\left(\sqrt{\frac{\log m}{m t^{n/2+1}}}\right) \label{eq:fluchtf} 
	\end{equation} 
	with probability at least $1-m^{-2}$. Because $n \ge 2$, we have $1-m^{-2} \ge 1 - m^{-\frac{2}{n}}$. Furthermore, under the scaling $t = m^{-1/(2n)}$, this fluctuation decays as $o(\sqrt{t})$. Combining \eqref{eq:sumtrue2}, \eqref{eq:expf}, and \eqref{eq:fluchtf}, we conclude that with probability at least $1-m^{-2}$,
	\begin{equation}
		\|\mathrm{Sum}^{\mathrm{true}}_t\|_{\mathrm{op}} \le C_3 \label{eq:sumtrueest}
	\end{equation}
	and  by \eqref{eq:lastterm}, \eqref{eq:fluchtf}

	\begin{equation}
		E_3
		=
		\|\Pi_p^\perp\mathrm{Sum}^{\mathrm{true}}_t(p)
		-
		\frac{\tilde \Bb_t (p)}{\vol_g (M^n)}\|_{\mathrm{op}}
		\le
	\frac{1}{2}	\left\|
		\frac1m\sum_{j=1}^mF_p(x_j)-\E_\mu[F_p]
		\right\|_{\mathrm{op}}
		 = o (\sqrt t) \label{eq:extlastterm}
	\end{equation}
	for $t$ sufficiently small.
	
	\medskip
	\noindent
	\underline{Step 2.} {\it Defining the High-Probability Geometric Event.}
	
	Define the projector difference tensor:
	\[
	\Delta_{t, S_m} \Pi (x,y) \coloneqq (\hat\Pi_{t, S_m})_x(\hat\Pi_{t, S_m})_y - \Pi_x\Pi_y.
	\]
	Let $\Omega_{t, m}$ be the event that the empirical projectors are uniformly well-behaved over nearby points  and $ \hat \Pi_{t, S_m}$ is   continuous:
	\begin{equation}
\begin{aligned}
\Omega_{t,m}\coloneqq\bigg\{S_m\in(M^n)^m:\;&
\sup_{\substack{x,y\in M^n\\ \|x-y\|\le\delta}}
\|\Delta_{t,S_m}\Pi(x,y)\|_{\mathrm{op}}\le2Ct,\\
&\hat\Pi_{t,S_m}\in C\bigl(M^n,\End(\R^d)\bigr)
\bigg\}.
\end{aligned}
\label{omtm}
\end{equation}
	where $C$ is the universal bound constant from Corollary \ref{cor:empitr_clean}. By Corollary \ref{cor:empitr_clean}, for sufficiently small $t$, 
	$$\mu^m(\Omega_{t, m}) \ge 1-m^{-2/n}.$$
	
	\medskip
	\noindent
	\underline{Step 3.} {\it Decoupling and Bounding $E_2$ and $E_1$.}
	
	We evaluate $E_2$ exclusively conditional on the event $\Omega_{t,m}$. Applying the triangle inequality to the empirical sum to bring the operator norm inside the integral, we obtain:
	\begin{align}
		E_2 &\le \|\mathrm{Sum}_{t, S_m} - \mathrm{Sum}^{\mathrm{true}}_t\|_{\mathrm{op}} \nonumber \\
		&\le \frac{1}{2m} \sum_{j=1}^m \Phi_t(p,x_j) \frac{\|x_j - p\|}{t} \|\Delta_{t, S_m}\Pi(x_j,p)\|_{\mathrm{op}}\chi_{\delta} (p, x_j) \nonumber \\
		&\le (C t) \left( \frac{1}{m} \sum_{j=1}^m \Phi_t(p,x_j) \frac{\|x_j - p\|}{t} \chi_{\delta}(p, x_j) \right). \label{eq:E2_decoupled}
	\end{align}  
	Let 
	\[
	A_p(x) \coloneqq \Phi_t(p,x) \frac{\|x - p\|}{t} \chi_{\delta}(p, x).
	\]
	By integrating in normal coordinates against the Gaussian measure (as detailed in the proof of Theorem \ref{thm:pwconvergencel} and taking into account Lemma \ref{lem:compact-exponential-open-subset}), we obtain
	\begin{equation}
		\E_\mu[A_p] \le C_4 t^{-1/2}. \label{eq:Apexp}
	\end{equation}
	By Lemma \ref{lem:MC_scalar_kernel}, the empirical sum $\frac{1}{m} \sum A_p(x_j)$ converges to its expectation uniformly over $p$ with a stochastic error of $o(\sqrt{t})$ with probability $1-m^{-2} \ge 1-m^{-\frac{2}{n}}$. Taking into account \eqref{eq:Apexp} and \eqref{eq:E2_decoupled}, we conclude that
	\begin{equation}
		E_2 \le C t \cdot \left( C_4 t^{-1/2} + o(\sqrt{t}) \right) \le C_5 \sqrt{t} \label{eq:E2_final}
	\end{equation}
	with probability at least $1 - 2m^{-\frac{2}{n}}$ (the sum of the failure probabilities of $\Omega_{t,m}$ and Lemma \ref{lem:MC_scalar_kernel}).
	
	Returning to $E_1$, we observe that on $\Omega_{t, m}$, we specifically have $\|\hat{\Pi}_p^\perp - \Pi_p^\perp\|_{\mathrm{op}} \le C t$. Therefore, taking into account  \eqref{eq:sumtrueest}, \eqref{eq:E2_decoupled}, \eqref{eq:E2_final},  we obtain
	\begin{align}
E_1
&\le
\|\hat\Pi_p^\perp-\Pi_p^\perp\|_{\mathrm{op}}
\Bigl(
\|\mathrm{Sum}^{\mathrm{true}}_t\|_{\mathrm{op}}
+
\|\mathrm{Sum}_{t,S_m}-\mathrm{Sum}^{\mathrm{true}}_t\|_{\mathrm{op}}
\Bigr)\nonumber\\
&\le Ct\bigl(C_3+C_5\sqrt t\bigr)
\le C_6t.
\label{eq:e1}
\end{align}
	with probability at least $1 - 3m^{-\frac{2}{n}}$ (adding the failure probability of Lemma \ref{lem:MC_kernel_vector} required to bound $\mathrm{Sum}^{\mathrm{true}}_t$).  
	
\medskip
	\noindent
	{\it Conclusion.}
	Taking into account the error decomposition \eqref{eq:error}, the
	analytical bias bound \eqref{eq:anbias_new}, the bound
	\eqref{eq:extlastterm} for $E_3$, and the uniform bounds for $E_1$
	and $E_2$, and absorbing the fixed factor
	$\vol_g(M^n)$ into the constants, we obtain
	\[
	\begin{aligned}
		\sup_{p\in M^n}
		\|\widetilde{(\hat B_{t, S_m})}_p-\tilde B_p\|_{\mathrm{op}}
		&\le
		C_1\sqrt t+o(\sqrt t)+C_6t+C_5\sqrt t\nonumber\\
		&\le C\sqrt t
	\end{aligned}
	\]
	for all sufficiently small $t$. By the union bound applied to the
	geometric event $\Omega_{t,m}$ and the scalar- and tensor-valued
	Monte Carlo events, this estimate holds simultaneously with probability
	at least
	\[
	1-3m^{-2/n}.
	\]
	This proves the first assertion of
	Proposition~\ref{prop:probempib}.
	
	On the same high-probability event, the map
	\[
	p\longmapsto\hat\Pi_{t,S_m}(p)
	\]
	is continuous by Proposition \ref{prop:empitangent_clean}. The
	definition \eqref{eq:Bempiext} is a finite sum of continuous
	expressions involving $p$, $\hat\Pi_{t,S_m}(p)$, and the smooth
	kernel and cut-off functions. Consequently,
	\[
	p\longmapsto
	\widetilde{(\hat B_{t,S_m})}_p
	\]
	is continuous on the same event.
	
	Finally, symmetrization preserves continuity. Since $B_p$ is
	symmetric, for $u,v\in\mathbb R^d$,
	\[
	\begin{aligned}
		\bigl(
		\widetilde{(\hat B^{\mathrm{sym}}_{t, S_m})}_p-\tilde B_p
		\bigr)(u,v)
		&=
		\frac12
		\bigl((\widetilde{\hat B_{t,S_m}})_p-\tilde B_p\bigr)(u,v)\nonumber\\
		&\quad+
		\frac12
		\bigl( (\widetilde{\hat B_{t,S_m}})_p-\tilde B_p\bigr)(v,u).
	\end{aligned}
	\]
	Therefore,
	\[
	\sup_{p\in M^n}
	\left\|
	\widetilde{(\hat B^{\mathrm{sym}}_{t,S_m})}_p-\tilde B_p
	\right\|_{\mathrm{op}}
	\le
	\sup_{p\in M}
	\left\|
	\widetilde{(\hat B_{t,S_m})}_p-\tilde B_p
	\right\|_{\mathrm{op}}
	\le C\sqrt t.
	\]
	This proves the remaining assertions.
	\end{proof}

\

For $S_m \in (M^n)^m$ and $t \in \R^+$, we set
$$  (\hat H^{sym}_{t, S_m})_p \coloneqq \sum_{i=1}^n (\hat{B}^{sym}_{t, S_m})_p (\hat{e}_i(p) , \hat{e}_i(p)), $$
where $\{ \hat{e}_i(p) \}_{i=1}^n$ is an orthonormal basis of the empirical tangent space $(\hat{\Pi}_{t, S_m})_p$.  The definition of $ (\hat H^{sym}_{t, S_m})_p $ is independent of the empirical orthonormal basis.
We define the empirical version $(\widehat{\Ww}_{t, S_m})_p$ of $\End_H(B)$ acting on $\Lambda^k \R^d$ by (cf. \eqref{eq:endh}  and \eqref{eq:muladjoint}):

\begin{equation}
	(\widehat{\Ww}_{t,S_m})_p(\omega) \coloneqq \sum_{j,l = 1} ^n \left\langle (\hat H^{sym}_{(t, S_m)})_p, (\hat{B}^{sym}_{(t,S_m)})_p \big( \hat{e}_j, \hat{e}_l \big) \right\rangle  \hat{e}_l \wedge  i_{(\hat{e}_j)^*} \hat{\Pi}_p \omega. \label{eq:wtsm}
\end{equation}
Here, $\hat{\Pi}_p$ is the shorthand notation for $\Lambda^k \hat{\Pi}_p$, and $(\hat{e}_j)^*$ denotes the dual covector with respect to the ambient Euclidean metric. This is a straightforward matrix multiplication involving the components of $\hat{B}^{sym}_{t, S_m}$ and $\hat H^{sym}_{t, S_m}$.

We also identify $\End_H(B)$ acting on $\Lambda^k TM^n$ with its ambient extension,  denoted by $\widetilde \End_H(B)$, acting on the space of $\Lambda^k (\R^d)$-valued functions on $M^n$ as follows: 
\begin{equation}
\widetilde\End_H(B)_p(\om) \coloneqq i_*\End_H(B)_p(\Pi_p \om). \label{eq:tendhb}
\end{equation} 
Here, $\Pi_p$ is the shorthand notation for $\Lambda^k \Pi_p$.

\begin{theorem}[\bf Consistency and continuity of $\widehat{\Ww}_{t, S_m}$]\label{thm:empiendb} 
	Let $M^n\subset\mathbb R^d$ be a closed $C^3$-smooth submanifold
	of dimension $n\ge2$, and let $\mu$ be the uniform distribution on
	$M^n$. Assume that $(t,m)$ satisfies \eqref{eq:optimal_scaling}, i.e.,
	$t=m^{-1/(2n)}$. Then, for sufficiently small $t$, with
	$\mu^m$-probability at least $1-3m^{-2/n}$ over the choice of
	$S_m\in(M^n)^m$, the estimator $\widehat{\Ww}$ satisfies:
	\begin{equation}
		\sup_{p \in M^n} \| (\widehat{\mathcal{W}}_{t, S_m})_p - \widetilde\End_H(B)_p \|_{\mathrm{op}} \le C_7 \sqrt{t}, \label{eq:curvest_sup}
	\end{equation}
	where $C_7$ depends only on the geometry of $M^n$ and on $k$. Furthermore,   on the same high-probability event,   the maps $\hat\Pi_{t, S_m}$ and  $\widehat{\mathcal{W}}_{t, S_m}: M^n \to \End (\Lambda ^k \R^d)$  are continuous.  
\end{theorem}

\begin{proof}
	Note that the operator $\End_H(B)_p$ at $p \in M^n$ is defined via the contraction of the mean curvature $H$ and the second fundamental form $B$. 
	Namely, for any orthonormal basis $\{e_i\}_{i=1}^n$ of the range of $\Pi_p$: 
	\begin{equation}
	\widetilde	\End_H(B)_p = \sum_{j, l=1}^n \langle H_p, B_p(e_j, e_l) \rangle e_l \wedge   i_{ e_j^*}\Pi_p.\label{eq:tendhbc}
	\end{equation}
	
	This sum is a canonical tensor contraction and is invariant under an orthogonal change of basis $\{e_i\}_{i=1}^n \to \{e_i'\}_{i=1}^n$. Consequently, the operator is a smooth function of the triplet $(\Pi_p, B_p, H_p)$ viewed as operators on the ambient space $\mathbb{R}^d$. Specifically, we can write:
	\begin{equation}
		\widetilde\End_H(B)_p(\omega) = \mathcal{C} \big( H_p \otimes \tilde B_p \otimes \Pi_p \otimes \omega \big),
	\end{equation}
	where $\mathcal{C}$ is a multilinear map representing the internal contractions and exterior/interior products. Clearly, we also have
	$$ (\widehat{\Ww}_{t, S_m})_p(\omega) = \mathcal{C} \big( (\hat{H}^{sym}_{t, S_m})_p \otimes \widetilde{(\hat{B}^{sym}_{t, S_m})_p} \otimes \hat{\Pi}_p \otimes \omega \big). $$
	
	The continuity assertion follows from
	Proposition~\ref{prop:probempib}: continuity of
	$p\mapsto\widetilde{(\hat B^{\mathrm{sym}}_{t,S_m})}_p$ implies
	continuity of $p\mapsto(\hat H^{\mathrm{sym}}_{t,S_m})_p$; together
	with the continuity of $p\mapsto\hat\Pi_p$ from
	Proposition~\ref{prop:empitangent_clean}, the contraction formula above
	shows that $p\mapsto\widehat{\mathcal W}_{t,S_m}(p)$ is continuous.
	
	By our previous results, if $(t, m)$ satisfy the scaling condition \eqref{eq:optimal_scaling}, then:
	\begin{itemize}
		\item $\sup_{p} \| \hat{\Pi}_p - \Pi_p \|_{\text{op}} \le Ct$ with probability at least $1 - m^{-2/n}$ (from Proposition \ref{prop:empitangent_clean}).
		\item $\sup_{p} \| \widetilde{(\hat{B}^{sym}_{t, S_m})_p} - \tilde B_p \|_{\text{op}} \le C_4\sqrt{t}$ with probability at least $1 - 3m^{-2/n}$ (from Proposition \ref{prop:probempib}).
	\end{itemize}
		
	Using the ambient extensions, for any ambient orthonormal basis
	$\{e_i\}_{i=1}^d$, we have
	\[
	(\hat H^{sym}_{t, S_m})_p
	=
	\sum_{i=1}^d \widetilde{(\hat{B}^{sym}_{t, S_m})_p}(e_i,e_i),
	\qquad
	H_p
	=
	\sum_{i=1}^d \tilde B_p(e_i,e_i).
	\]
	Hence
	\[
	\|(\hat H^{sym}_{t,S_m})_p-H_p\|
	\le
	d\,
	\|\widetilde{(\hat B^{sym}_{t,S_m})_p}-\tilde B_p\|_{\mathrm{op}}.
	\]
	Therefore, by Proposition~\ref{prop:probempib},
	\[
	\sup_{p\in M^n}\|(\hat H^{sym}_{t, S_m})_p-H_p\|
	\le C\sqrt t 
	\]
	with probability  at least $1- 3 m^{-2/n}$.

	Let $\Delta \Pi_p = \hat{\Pi}_p - \Pi_p$, $\Delta B_p = \widetilde{(\hat{B}^{sym}_{t, S_m})_p} - \tilde B_p$, and $\Delta H_p = (\hat H^{sym}_{t, S_m})_p- H_p$. Because $M^n$ is compact and smooth, the operators $H, B$, and $\Pi$ are uniformly bounded in norm by a constant $K(M^n)$.
	
	By multilinearity and the uniform boundedness of all operators involved, the difference $\widehat{\Ww}_p - \widetilde\End_H(B)_p$ expands into a finite sum of terms, each containing at least one factor among $\Delta H, \Delta B, \Delta \Pi$. Hence:
	\begin{align}
		\|\widehat{\Ww}_p - \widetilde\End_H(B)_p\|_{\mathrm{op}} &\le \| \Cc\big ((\hat{H}^{sym}_{t, S_m})_p, \widetilde{(\hat{B}^{sym}_{t, S_m})_p}, \hat{\Pi}_p\big) - \Cc(H, \tilde B, \Pi)_p \|_{\mathrm{op}} \nonumber \\
		&\le \| \Cc(\Delta H, \tilde B, \Pi)_p \|_{\mathrm{op}} + \| \Cc(H, \Delta B, \Pi)_p \|_{\mathrm{op}} \nonumber \\
		&\quad + \| \Cc(H, \tilde B, \Delta \Pi)_p \|_{\mathrm{op}} + O(\|\Delta\|^2).
	\end{align}
	where, for brevity,
	\[
	\|\Delta\|
	\coloneqq
	\|\Delta H\|+\|\Delta B\|_{\mathrm{op}}+\|\Delta\Pi\|_{\mathrm{op}}.
	\]
	Each term on the RHS is bounded by the product of the norms of its constituents. Since 
	\[
	\|\Delta H\|=O(\sqrt t),\qquad
	\|\Delta B\|_{\mathrm{op}}=O(\sqrt t),\qquad
	\|\Delta\Pi\|_{\mathrm{op}}=O(t).
	\]
	all quadratic terms are $o(\sqrt{t})$. Substituting the uniform bounds into the linear terms yields:
	\begin{equation}
		\|\widehat{\Ww}_p - \widetilde\End_H(B)_p\|_{\mathrm{op}} \le C_8 \|\Delta H\| + C_9 \|\Delta B\|_{\mathrm{op}} + K_3 \|\Delta \Pi\|_{\mathrm{op}}.
	\end{equation}
	Substituting the rates $O(\sqrt{t})$, $O(\sqrt{t})$, and $O(t)$ respectively, we obtain:
	\begin{equation}
		\sup_{p \in M^n} \| \widehat{\Ww}_p - \widetilde\End_H(B)_p \|_{\mathrm{op}} \le C_7\sqrt{t}
	\end{equation}
	with probability at least $1 - 3m^{-\frac{2}{n}}$ for $t$ sufficiently small. This completes the proof of the  first assertion of  Theorem \ref{thm:empiendb}.
\end{proof}

\subsection{Recovering the Riemannian curvature tensor and Weitzenb\"ock potentials}\label{subs:weitz}
Using the Gauss equation \eqref{eq:gaussflat}, we define for each $t \in \R_+$, $p \in M^n$, and $S_m \in (M^n)^m$ the empirical Riemannian curvature tensor
$$ \widetilde{(\hat R_{t, S_m})_p} : \R^d \times \R^d \times \R^d \times \R^d \to \R $$ 
as follows:
\begin{align}\label{eq:empirc}
	\widetilde{(\hat R_{t, S_m})_p} (X, Y, Z, W) \coloneqq  &   \la \widetilde{(\hat{B}^{sym}_{t, S_m})_p} (X,W), \widetilde{(\hat{B}^{sym}_{t, S_m})_p} (Y, Z)\ra\nonumber\\
	& -\la \widetilde{(\hat{B}^{sym}_{t, S_m})_p}(X, Z), \widetilde{(\hat{B}^{sym}_{t, S_m})_p} (Y, W) \ra ,
\end{align}
where $\la \cdot , \cdot \ra$ is the ambient Euclidean metric.

We also extend the Riemannian curvature $R_p: T_p M^n \times T_p M^n \times T_p M^n \times T_p M^n \to \R$ to an operator
$$ \tilde R_p : \R^d \times \R^d \times \R^d \times \R^d \to \R$$ 
by 
\[
\tilde R_p(X,Y,Z,W)
\coloneqq
R_p(\Pi_pX,\Pi_pY,\Pi_pZ,\Pi_pW).
\]

Then the Gauss equation \eqref{eq:gaussflat} also holds for $\tilde R$, i.e., 
\begin{equation*}
	\tilde R_p (X, Y, Z, W) =  \la \tilde B_p (X,W), \tilde B_p (Y, Z)\ra - \la \tilde B_p (X, Z), \tilde B_p (Y, W) \ra.
\end{equation*}

Recalling the degree-preserving Weitzenb\"ock potential acting on $\Lambda^\ast TM^n$ \eqref{eq:weitzenboeck}, and the partial trace operator \eqref{eq:weitzenboecko}:
\begin{align*}
	&\mathcal{R}\omega = \sum_{i,j = 1}  e_i \wedge i_{e_j^*} \left( \sum_{p,l} R(e_i, e_j, e_p, e_l) e_l \wedge i_{e_p^*}\omega \right),\\
	&\Rr^{(1)}\omega 
	=
	\sum_{j,l,s}
	R(e_j,e_s,e_j,e_l)\,e_s\wedge i_{e_l^\ast}\omega.\nonumber
\end{align*}
For $0\le k\le n$, we denote their restriction to $\Lambda^kTM^n$ by
\[
\Rr_k
\coloneqq
\Rr\big|_{\Lambda^kTM^n}, \qquad  \Rr_k^{(1)}
\coloneqq
\Rr^{(1)}\big|_{\Lambda^kTM^n}.
\]
When the degree is clear from the context, we simply write $\Rr$ and $\Rr^{(1)}$, respectively.

For $p \in M^n$ we extend the Weitzenb\"ock potential and the partial trace operator $\Rr_p, \Rr^{(1)}_p \in \End (\Lambda^* T_p M^n)$ to operators 
$\widetilde {\Rr}_p ,  \widetilde {\Rr}^{(1)}_p \in \End (\Lambda^k \R^d)$ as follows (cf. \eqref{eq:tendhb}):
\begin{align}	
	\widetilde {\Rr}_p (\om) \coloneqq i_* (\Rr_p) (\Pi_p \om),\qquad 	\widetilde {\Rr}^{(1)}_p (\om) \coloneqq i_* (\Rr^{(1)}_p) (\Pi_p \om) \label{eq:twrk}
\end{align}
where, as before, abusing notation, $\Pi_p$ denotes the extension of the projection $\Pi_p$ to $\Lambda^* \R^d$. 

We define the empirical Weitzenb\"ock potential and the partial trace operator $(\widehat {\Rr}_{t, S_m})_p, (\widehat {\Rr}^{(1)}_{t, S_m})_p \in \End (\Lambda^* \R^d)$ by (cf. \eqref{eq:wtsm}):
\begin{align}
	&	(\widehat {\Rr}_{t, S_m})_p (\om) \coloneqq \sum_{a, b = 1}^n \hat e_a\wedge i_{\hat e_b^*} \Big (\sum_{c,d = 1}^n \widetilde{(\hat R_{t, S_m})_p} (\hat e_a, \hat e_b, \hat e_c, \hat e_d)\hat e_d \wedge i_{\hat e_c^*} (\om)\Big),\label{eq:ewf}\\	
	&	(\widehat {\Rr}^{(1)}_{t, S_m})_p (\om) \coloneqq \sum_{a, b, c =1}^n 	\widetilde{(\hat R_{t, S_m})_p} (\hat e_a, \hat e_b, \hat e_a, \hat e_c)\hat e_b \wedge i_{\hat e_c^*} (\om).\label{eq:ewp}
\end{align}
where $\{\hat e_a\}_{a=1}^n$ is an orthonormal basis of the empirical tangent space $(\hat \Pi _{t, S_m} )_p$ and $\{(\hat e_a)^*\}$ denote the dual covectors with respect to the ambient Euclidean metric.

\begin{remark}\label{rem:wind} 
	The definitions in \eqref{eq:ewf}, \eqref{eq:ewp} are independent of the chosen empirical orthonormal basis. A compact way to make this transparent is to introduce a basis-independent ambient contraction.
	Let $S$ be a four-linear form on $\R^d$, and let $P$ be an orthogonal projection on $\R^d$. Fix an ambient orthonormal basis $\{E_\alpha\}_{\alpha=1}^d$ and define
	\begin{align}	
		\mathfrak B_k(S,P)(\omega)
		\coloneqq
		&\sum_{\alpha,\beta=1}^d PE_\alpha \wedge i_{(PE_\beta)^*}
		\Big (	\sum_{\gamma, \delta =1}^d S(PE_\alpha,PE_\beta,PE_\gamma,PE_\delta)\nonumber\\
		&\quad\cdot
		PE_\delta\wedge
		i_{(PE_\gamma)^*}
		\bigl((\Lambda^kP)\omega\bigr)\Big)
		\label{eq:contractionf}\\		
		\mathfrak C_k(S,P)(\omega)
		\coloneqq
		&\sum_{\alpha,\beta,\gamma=1}^d
		S(PE_\alpha,PE_\beta,PE_\alpha,PE_\gamma)
		\nonumber\\
		&\quad\cdot
		PE_\beta\wedge
		i_{(PE_\gamma)^*}
		\bigl((\Lambda^kP)\omega\bigr).
		\label{eq:contraction}
	\end{align}
	The right-hand sides of \eqref{eq:contractionf} and \eqref{eq:contraction} are independent of the ambient orthonormal basis because every repeated index represents contraction with the Euclidean metric. If $\{e_a\}_{a=1}^n$ is an orthonormal basis of $\operatorname{range}P$, then \eqref{eq:contraction} is equal to
	\[
	\sum_{a,b,c=1}^n
	S(e_a,e_b,e_a,e_c)
	e_b\wedge i_{e_c^*}\bigl((\Lambda^kP)\omega\bigr).
	\]
	Consequently, setting $	\widetilde{\Rr} ^{(1)}_{k,q}\coloneqq \widetilde \Rr_q ^{(1)} \big | _{\Lambda ^k \R^d}$, we have
	\[
	\widetilde{\Rr}^{(1)}_{k,q}
	=\mathfrak C_k(\widetilde R_q,\Pi_q)
	\]
	and, setting $\widehat{\Rr}^{(1)}_{k,t,S_m}(q)\coloneqq (\widehat {\Rr}^{(1)}_{t, S_m})_q \big | _{\Lambda ^k \R^d}$, we have
	\[
	\widehat{\Rr}^{(1)}_{k,t,S_m}(q)
	=\mathfrak C_k\left(
	\widetilde{(\widehat R_{t,S_m})_q},\widehat\Pi_q
	\right).
	\]
	Similarly, 
	\[
	\widehat{\Rr}_{k,t,S_m}(q)
	=\mathfrak B_k\left(
	\widetilde{(\widehat R_{t,S_m})_q},\widehat\Pi_q
	\right).
	\]
\end{remark}

\begin{theorem}[\bf Consistency of $(\hat R_{t, S_m})_p$]\label{thm:Rcurv} 
	Let $M^n$ be a closed $C^3$-smooth submanifold in $\R^d$ of dimension $n \ge 2$, and let $t = m ^{- \frac{1}{2n}}$. 
	
	1) Then, for all sufficiently large $m$, with probability at least $1 - 3m^{-2/n}$ over the choice of i.i.d. $S_m\sim \mu^m$, the estimator $\hat R_{t, S_m}$ satisfies:
	\begin{equation}
		\sup_{p \in M^n} \| \widetilde{(\hat R_{t, S_m})_p} - \tilde R_p \|_{\mathrm{op}} \le C \sqrt{t}, \label{eq:curvest_supr}
	\end{equation}
	where $C$ depends only on the geometry of $M^n$. 
	
	2) Consequently there exists a constant $C'$ such that
	\begin{align}
		\sup_{p \in M^n} \| (\widehat {\Rr}_{t, S_m})_p- \widetilde {\Rr}_p \|_{\mathrm{op}} \le C' \sqrt{t} \label{eq:weitz_supr}\\
		\sup_{p \in M^n} \| (\widehat {\Rr}^{(1)}_{t, S_m})_p- \widetilde {\Rr}^{(1)}_p \|_{\mathrm{op}} \le C' \sqrt{t} \label{eq:weitz_supro}
	\end{align}
	where $C'$ depends only on the geometry of $M^n$ and $k$.
\end{theorem}

\begin{proof}
	1) Set
	\[
	D_q
	\coloneqq
	(\widetilde{\widehat B^{\mathrm{sym}}_{t,S_m}})_q
	-\widetilde B_q.
	\]
	For unit vectors $X,Y,Z,W\in\R^d$, subtraction of the two Gauss formulas gives
	\begin{align*}
		&\left|
		\widetilde{(\widehat R_{t,S_m})_q}(X,Y,Z,W)
		-\widetilde R_q(X,Y,Z,W)
		\right|\\
		&\le
		\left|\la D_q(X,W),
		(\widetilde{\widehat B^{\mathrm{sym}}_{t,S_m}})_q(Y,Z)\ra\right|
		+
		\left|\la\widetilde B_q(X,W),D_q(Y,Z)\ra\right|\\
		&\quad+
		\left|\la D_q(X,Z),
		(\widetilde{\widehat B^{\mathrm{sym}}_{t,S_m}})_q(Y,W)\ra\right|
		+
		\left|\la\widetilde B_q(X,Z),D_q(Y,W)\ra\right|.
	\end{align*}
	Since $\widetilde B$ is uniformly bounded and
	$\widetilde{(\widehat B^{\mathrm{sym}}_{t,S_m})}$ is uniformly bounded on the good event, this yields
	\[
	\sup_{q\in M^n}
	\left\|
	\widetilde{(\widehat R_{t,S_m})_q}-\widetilde R_q
	\right\|_{\mathrm{op}}
	\le C\sqrt t.
	\]
	
	2) The second assertion follows from the first one using the same argument in the proof of Theorem \ref{thm:empiendb}.  
	The maps $\mathfrak B_k$ and $\mathfrak C_k$ in \eqref{eq:contraction} are polynomial and multilinear in the coefficients of $S$ and $P$. Hence, on uniformly bounded sets, there is a constant $C_{d,k}$ such that
	
	\[
	\|\mathfrak B_k(S,P)-\mathfrak B_k(T,Q)\|_{\mathrm{op}}
	\le
	C_{d,k}
	\left(
	\|S-T\|_{\mathrm{op}}+\|P-Q\|_{\mathrm{op}}
	\right),
	\]
	\[
	\|\mathfrak C_k(S,P)-\mathfrak C_k(T,Q)\|_{\mathrm{op}}
	\le
	C_{d,k}
	\left(
	\|S-T\|_{\mathrm{op}}+\|P-Q\|_{\mathrm{op}}
	\right).
	\]
	Apply this with
	\[
	S=\widetilde{(\widehat R_{t,S_m})_q},
	\quad
	T=\widetilde R_q,
	\quad
	P=\widehat\Pi_q,
	\quad
	Q=\Pi_q.
	\]
	Then
	\begin{align*}
		&\sup_{q\in M^n}
		\left\|
		\widehat{\mathcal R}_{k,t,S_m}(q)
		-\widetilde{\mathcal R}_{k,q}
		\right\|_{\mathrm{op}}\\
		&\qquad\le
		C_{d,k}
		\left(
		\sup_{q\in M^n}
		\left\|
		\widetilde{(\widehat R_{t,S_m})_q}-\widetilde R_q
		\right\|_{\mathrm{op}}
		+
		\sup_{q\in M^n}\|\widehat\Pi_q-\Pi_q\|_{\mathrm{op}}
		\right)\\
		&\qquad\le C_{d,k}(\sqrt t+t)
		\le C_k\sqrt t.
	\end{align*}
	This estimate holds on the same event of probability at least
	$1-3m^{-2/n}$. This proves \eqref{eq:weitz_supr}. In the same way we obtain \eqref{eq:weitz_supro}.
	
\end{proof}

\subsection{Empirical Hodge Laplacians and their convergence}

We identify $\Lambda^k (\R^d)^*$ with $\Lambda^k \R^d$ via the Euclidean metric.
Identifying $\om \in \Lambda^k T_p M^n$ with its image $i_* \om \in \Lambda^k \R^d$, we extend the Laplacian operator $\Delta^k : \Gamma (\Lambda^k TM^n) \to \Gamma (\Lambda^k TM^n)$ to an operator, denoted by $\widetilde \Delta^k$, acting on smooth $\Lambda^k \R^d$-valued functions on $M^n$ as follows:
\begin{equation}
	\widetilde \Delta^k \om \coloneqq i_*\Delta^k \Pi (\om), \label{eq:deltaext}
\end{equation}
where $\Pi(\omega)(p)\coloneqq\Pi_p\omega(p)$ and $i_*$ is the
map induced by the inclusion $i$, as in \eqref{eq:inclusion}.

Denote by $\Meas(\Xx, \Yy)$ the space of measurable mappings from a measurable space $\Xx$ to a measurable space $\Yy$. For $t \in \R_{+}$, $S_m \in (M^n)^m$, and $\om \in C^4 (M^n, \Lambda^k \R^d)$, we define the empirical Hodge Laplacian
\begin{align}
	(\hat{\Delta}_{t, S_m}) &: C(M^n, \Lambda^k \R^d) \to \Meas(M^n, \Lambda^k\R^d),\nonumber\\
	(\hat{\Delta}_{t, S_m})_p \om &\coloneqq (\hat{\Lb}_{t, S_m})_p \om - (\widehat{\Ww}_{t, S_m})_p \om -  (\widehat{\Rr}_{t, S_m})_p \om -  (\widehat{\Rr}^{(1)}_{t, S_m})_p \om, 
	\label{eq:hodgeempi}
\end{align}
where the empirical diffusion operator incorporates the volume scaling and the distance cutoff:
\begin{align}
(\hat{\Lb}_{t,S_m})_p\omega
&\coloneqq
\frac{\vol_g(M^n)}{mt}
\sum_{j=1}^m\Phi_t(p,x_j)\chi_\delta(p,x_j)\nonumber\\
&\qquad\times
\Bigl(
(\hat\Pi_{t,S_m})_p\omega(p)
-(\hat\Pi_{t,S_m})_p(\hat\Pi_{t,S_m})_{x_j}\omega(x_j)
\Bigr).
\label{eq:lbnew}
\end{align}
and $\Phi_t$ is defined in \eqref{eq:phit}.

Next, we extend the operator $\Delta^k_t$ defined in \eqref{eq:hodgeambient} to an operator,  denoted by $\widetilde\Delta^k_t$, acting on smooth $\Lambda^k \R^d$-valued functions on $M^n$ as follows, cf. \eqref{eq:deltaext}:
\begin{equation}
\widetilde	\Delta^k_t \om \coloneqq i_*\Delta^k_t \Pi (\om). \label{eq:deltakext}
\end{equation}
We also extend the operator $\Lb_t$ to an operator,  denoted by $\widetilde\Lb_t$, acting on the space of $\Lambda^k\R^d$-valued functions on $M^n$ as follows:
$$ \widetilde\Lb_t \omega \coloneqq i_*\Lb_t \Pi \om. $$
Since projections $\Pi_x: \R^d \to T_xM^n$ do not increase the norm, Theorem \ref{thm:analytic_limit} is also valid for the extension $\widetilde\Delta^k_t$.

\begin{theorem}[\bf Consistency and continuity of empirical Hodge Laplacians]\
	\label{thm:empiconvergence}	Let $M^n\subset\mathbb R^d$ be a closed orientable $C^4$-smooth
	submanifold of dimension $n\ge 3$, and let $\mu$ be the uniform
	distribution on $M^n$.  Let
	$\omega\in C^4(M^n,\Lambda^k\mathbb R^d)$, and assume that $(t,m)$
	satisfy \eqref{eq:optimal_scaling}, i.e.,
	\[
	t=m^{-\frac{1}{2n}}.
	\]
	Then, for all sufficiently small $t$, with probability at least
	$1-6m^{-2/n}$ over i.i.d. samples $S_m\sim\mu^m$, we have
	\begin{equation}
		\sup_{p\in M}
		\|\hat\Delta^k_{t,S_m}\omega(p)-\widetilde\Delta^k\omega(p)\|
		\le
		C_{12}\sqrt t\,\|\omega\|_{C^4}.
		\label{eq:hodgesup}
	\end{equation}
	Here $C_{12}$ depends only on the geometry of $M^n$. Furthermore, on
	the same high-probability event,
	\[
	\hat\Delta^k_{t,S_m}:
	C^4(M^n,\Lambda^k\mathbb R^d)
	\to
	C(M^n,\Lambda^k\mathbb R^d)
	\]
	is a continuous linear operator.
\end{theorem}
\begin{proof}[Proof of Theorem \ref{thm:empiconvergence}]
	We split the error into an analytical bias and an empirical error:
	\begin{equation}
		\hat{\Delta}_{t,S_m}\omega - \widetilde \Delta^k \omega = \underbrace{\big(\hat{\Delta}_{t,S_m}\omega - \widetilde\Delta_t^k \omega\big)}_{\text{empirical error}} + \underbrace{\big(\widetilde\Delta_t^k \omega - \widetilde\Delta^k \omega\big)}_{\text{analytic bias}}.
		\label{eq:main_split}
	\end{equation}
	
	By Theorem \ref{thm:analytic_limit} (see the remark after \eqref{eq:deltakext}), we have:
	\begin{equation}
		\sup_{p \in M^n} \|\widetilde \Delta_t^k \omega(p) - \widetilde\Delta^k \omega(p)\| \le C_2(M^n) t \|\om\|_{C^4}.
		\label{eq:analytic_bias}
	\end{equation}
	
	To handle the empirical error, we define an intermediate operator using the true geometric projectors evaluated on the empirical sample:
	\begin{equation}
		\tilde{\Lb}_{t, S_m}\omega(p) \coloneqq \frac{\vol_g(M^n)}{m t} \sum_{j=1}^m \Phi_t(p, x_j) \Big( \Pi_p \om(p) -\Pi_{p}\Pi_{x_j} \om(x_j) \Big) \chi_\delta(p, x_j).
		\label{eq:lbtilde}
	\end{equation}			
	We split the error of the diffusion part into a projection error and a Monte Carlo error:
	\begin{align}
		\hat{\Lb}_{t,S_m}\omega - \widetilde\Lb_t \omega &= \underbrace{\Big( \hat{\Lb}_{t,S_m}\omega - \tilde{\Lb}_{t,S_m}\omega \Big)}_{\text{projection error}} + \underbrace{\Big( \tilde{\Lb}_{t,S_m}\omega - \widetilde\Lb_t \omega \Big)}_{\text{Monte Carlo error}}.\label{eq:deltaker}
	\end{align}
	
	\medskip
	\noindent

	\underline{Step 1.} {\it Estimating the projection error.}
		
	Let
	\[
	G_p(y)
	\coloneqq
	(\hat\Pi_p-\Pi_p)\omega(p)
	-
	(\hat\Pi_p\hat\Pi_y-\Pi_p\Pi_y)\omega(y).
	\]
	Then
	\[
	\hat {\Lb}_{t,S_m}\omega(p)-\widetilde\Lb_{t,S_m}\omega(p)
	=
	\frac{\operatorname{vol}(M^n)}{mt}
	\sum_{j=1}^m
	\Phi_t(p,x_j)\chi_\delta(p,x_j)G_p(x_j).
	\]
	Moreover,
	\[
	G_p(p)=0,
	\]
	because
	\[
	(\hat\Pi_p-\Pi_p)\omega(p)
	-
	(\hat\Pi_p^{2}-\Pi_p^{2})\omega(p)
	=
	(\hat\Pi_p-\Pi_p)\omega(p)
	-
	(\hat\Pi_p-\Pi_p)\omega(p)
	=
	0.
	\]
	On the event of Corollary~\ref{cor:empitr_clean}, the map
	\[
	y\mapsto \hat\Pi_p\hat\Pi_y-\Pi_p\Pi_y
	\]
	is uniformly \(O(t)\) in operator norm. 
	Let
	\[
	A_p(y)\coloneqq
	\hat\Pi_p\hat\Pi_y-\Pi_p\Pi_y .
	\]
	Then,  $A_p (p) \om (p) = 	(\hat\Pi_p-\Pi_p)\omega(p)$, and 
	\[
	G_p(y)=A_p(p)\omega(p)-A_p(y)\omega(y).
	\]
	Hence
	\[
	G_p(y)
	=
	A_p(y)\big(\omega(p)-\omega(y)\big)
	+
	\big(A_p(p)-A_p(y)\big)\omega(p).
	\]
	By Corollary \ref{cor:empitr_clean}, we have
	\[
	\|A_p(y)\|_{\mathrm{op}}\le Ct
	\]
	uniformly in $p,y$.  
	Under the hypotheses of the final part of
	Proposition \ref{prop:empitangent_clean}, one also has
	
	\begin{align}
	\Big \|		\Big [	\Lambda^k(\widehat\Pi_p\widehat\Pi_y)		-		
		\Lambda^k(\Pi_p\Pi_y)	\Big ]		-		
		\big[		\Lambda^k(\widehat\Pi_p^2)		-		
		\Lambda^k(\Pi_p^2)
		\big ]	\Big \|_{\mathrm{op}}
		\le C_k t\,\|y-p\|_{\R^d}.
		\label{eq:apmean_exterior}
	\end{align}
	
	Indeed, this follows from the multilinearity of the exterior power,
	the estimate \eqref{eq:apmean}, the uniform $O(t)$-projection error,
	and the uniform Lipschitz continuity of $p\mapsto\Pi_p$.

	 By \eqref{eq:apmean_exterior}
	\[
	\|A_p(y)-A_p(p)\|_{\mathrm{op}}
	\le Ct\|y-p\|.
	\]

	Hence
	\[
	\|G_p(y)\|
	\le
	Ct\|\omega(y)-\omega(p)\|
	+
	Ct\|y-p\|\|\omega\|_{C^0}
	\le
	Ct\|y-p\|\|\omega\|_{C^1}.
	\]

	Therefore
	\[
	\|\hat\Lb_{t,S_m}\omega(p)-\widetilde\Lb_{t,S_m}\omega(p)\|
	\le
	Ct\|\omega\|_{C^1}
	\frac1m
	\sum_{j=1}^m
	\Phi_t(p,x_j)
	\frac{\|x_j-p\|}{t}
	\chi_\delta(p,x_j).
	\]
	By Lemma~\ref{lem:MC_scalar_kernel}, the last empirical average is
	bounded uniformly by \(C t^{-1/2}+o(\sqrt t)\). Hence
	\begin{equation}
	\sup_{p\in M}
	\|\hat\Lb_{t,S_m}\omega(p)-\widetilde\Lb_{t,S_m}\omega(p)\|
	\le
	C\sqrt t\,\|\omega\|_{C^1}.  \label{eq:proj_error}
	\end{equation}

	\medskip
	\noindent
	\underline{Step 2.} {\it Estimating the Monte Carlo error}. 
	For each $x\in M^n$ and $t>0$, define the vector-valued function
	\[
	f_{x,t}(y)
	\coloneqq
	\frac1t\Phi_t(x,y)
	\bigl(
	\Pi_x\omega(x)-\Pi_x\Pi_y\omega(y)
	\bigr)
	\chi_\delta(x,y), 	\qquad
	y\in M^n.
	\]
Introduce
\[
\widetilde\Lb _{t, \delta} \om (x) \coloneqq \vol_g (M^n)  \E_\mu [f_{x, t}].
\]
	We have
	\begin{equation}\label{eq:decomptl}
\widetilde \Lb_{t, S_m}  -\widetilde \Lb_t =  (\widetilde \Lb_{t, S_m}- \widetilde  \Lb_{t, \delta} )  + (\widetilde  \Lb_{t, \delta} - \widetilde \Lb_t).
\end{equation}	
By Lemma~\ref{lem:compact-exponential-open-subset}(c), the last term on the right-hand side of \eqref{eq:decomptl} is exponentially small in $t^{-1}$. Now we estimate the first term in the right-hand side of \eqref{eq:decomptl}.		
			\begin{equation}
\widetilde{\Lb}_{t,S_m}\omega(x)-\widetilde{\Lb}_{t,\delta}\omega(x)
=
\vol_g(M^n)\left[
\frac1m\sum_{j=1}^m f_{x,t}(x_j)
-\mathbb E_{y\sim\mu}[f_{x,t}(y)]
\right].
\label{eq:MCerror}
\end{equation}
	By Lemma \ref{lem:densitybounded}, and after absorbing the fixed factor
	$\vol_g(M^n)$ into the constant, there exists
	$C_{11}>0$, depending only on the geometry of $M^n$, such that
	\begin{equation}
		\sup_{x\in M^n}
		\left\|
		\frac1m\sum_{j=1}^m f_{x,t}(x_j)
		-	
		\mathbb E_{y\sim\mu}[f_{x,t}(y)]
		\right\|
		\le
		C_{11}\|\omega\|_{C^1}
		\sqrt{\frac{\log m}{m\,t^{n/2+1}}}
		\label{eq:mc_error1}
	\end{equation}
	with probability at least $1-m^{-2}$. Under the scaling
	$t=m^{-1/(2n)}$, Lemma \ref{lem:densitybounded} also gives
	\[
	\sqrt{\frac{\log m}{m\,t^{n/2+1}}}
	=o(\sqrt t)
	\]
	for every $n\ge2$.

	Combining \eqref{eq:deltaker}, \eqref{eq:proj_error}, \eqref{eq:decomptl}, \eqref{eq:MCerror}, and \eqref{eq:mc_error1}, the total diffusion operator error is bounded by:
	\begin{align}
&\sup_{x\in M^n}
\|\hat\Lb_{t,S_m}\omega(x)-\widetilde\Lb_t\omega(x)\|\nonumber\\
&\qquad\le
C\sqrt t\,\|\omega\|_{C^1}
+C_{11}\|\omega\|_{C^1}
\sqrt{\frac{\log m}{m\,t^{n/2+1}}}
+O(e^{-c/t})\|\omega\|_{C^1}\nonumber\\
&\qquad=O(\sqrt t)\|\omega\|_{C^1}.
\label{eq:L_total}
\end{align}

	\underline{Step 3.} {\it Zeroth-order term and conclusion.}
	
	Recalling \eqref{eq:curvest_sup} and \eqref{eq:weitz_supr}, we have under the scaling $t = m^{-\frac{1}{2n}}$:
	\begin{align}
		\sup_{x\in M^n} \| \widehat{\Ww}_{t,S_m}(x) -\widetilde \End_H(B)(x) \|_{\mathrm{op}} &\le C_7 \sqrt{t}	
		\label{eq:W_total}\\
		\sup_{x \in M^n} \| \widehat{\Rr}_{k, t,S_m}(x) -\widetilde \Rr_k(x) \|_{\mathrm{op}} &\le C_k \sqrt{t} \label{eq:Rktotal}\\
		\sup_{x \in M^n} \| \widehat{\Rr}^{(1)}_{k, t,S_m}(x) -\widetilde \Rr^{(1)}_k(x) \|_{\mathrm{op}} &\le C_k \sqrt{t} \label{eq:Rktotalo}
	\end{align}
	for sufficiently small $t$ with probability at least $1 - 3m^{-\frac{2}{n}}$ over the choice of $S_m$.		
	
	Combining the split \eqref{eq:main_split}, the analytic bias \eqref{eq:analytic_bias} ($O(t)$), and using the decomposition of the empirical error:
	\begin{align*}
		\hat{\Delta}_{t,S_m} - \widetilde \Delta_t^k &= (\hat{\Lb}^k_{t,S_m} -\widetilde {\Lb}^k_{t}) - \big( \widehat{\Ww}_{t,S_m} -\widetilde \End_H(B)\big)\\
		&\quad - \big(\widehat{\Rr}_{k, t,S_m} -\widetilde \Rr_k\big) - \big(\widehat{\Rr}^{(1)}_{k, t,S_m} -\widetilde \Rr^{(1)}_k\big),
	\end{align*}
	taking into account the decoupled diffusion error \eqref{eq:L_total} ($O(\sqrt t)$), and the zeroth-order errors \eqref{eq:W_total}, \eqref{eq:Rktotal}, and \eqref{eq:Rktotalo} ($O(\sqrt{t})$), we obtain:
	\begin{equation}
		\sup_{x\in M^n} \| \hat{\Delta}_{t,S_m}\omega(x) - \widetilde\Delta^k \omega(x) \| \le C_{12} \sqrt{t} \|\om\|_{C^4}
	\end{equation}
	for sufficiently small $t$, with probability at least $1 - 6m^{-\frac{2}{n}}$ over the choice of $S_m$. This holds on the intersection of the following high-probability events:
	\begin{itemize}
		\item the differentiated projection event: failure at most $2 m^{-2/n}$,
		\item the concentration event from Lemma \ref{lem:densitybounded}: failure at most $m^{-2} \le m^{-2/n}$,
		\item the common event for $\hat B, \widehat \Ww, \widehat \Rr, \widehat \Rr^{(1)}$: failure at most $3 m^{-2/n}$.
	\end{itemize}
	This completes the proof of the first assertion of Theorem \ref{thm:empiconvergence}.
	
	\medskip
	
	On the same event, the maps
	\[
	p\mapsto\widehat\Pi_{t,S_m}(p),\qquad
	p\mapsto\widehat{\mathcal W}_{t,S_m}(p),\qquad
	p\mapsto\widehat{\mathcal R}_{k,t,S_m}(p),\qquad
	p\mapsto\widehat{\mathcal R}^{(1)}_{k,t,S_m}(p)
	\]
	are continuous. It follows from the finite-sum definition that
	$\widehat{\Lb}^k_{t,S_m}\omega$, and hence
	$\widehat\Delta^k_{t,S_m}\omega$, is continuous in $p$. Moreover,
	for fixed $t$ and $S_m$,
	\[
	\|\widehat\Delta^k_{t,S_m}\omega\|_{C^0}
	\le
	C_{t,S_m}\|\omega\|_{C^0}
	\le
	C_{t,S_m}\|\omega\|_{C^4}.
	\]
	Thus
	\[
	\widehat\Delta^k_{t,S_m}:
	C^4(M^n,\Lambda^k\mathbb R^d)
	\longrightarrow
	C(M^n,\Lambda^k\mathbb R^d)
	\]
	is a continuous linear operator.  This completes the  proof of Theorem \ref{thm:empiconvergence}.
	\end{proof}

To prove the  convergence theorem \ref{thm:harmonic_cluster}  for empirical harmonic cluster  we need some preparation. 
For each sample point \(x_i\in S_m\), let
\[
P_i^{(k)}
\coloneqq
\Lambda^k\widehat\Pi_{x_i},
\qquad
E_i^{(k)}
\coloneqq
\operatorname{Ran}P_i^{(k)}.
\]
For an arbitrary \(x\in M^n\), set
\[
P_x^{(k)}
\coloneqq
\Lambda^k\widehat\Pi_x,
\qquad
E_x^{(k)}
\coloneqq
\operatorname{Ran}P_x^{(k)}.
\]

Set
\begin{equation}
	\Hh_{m,k}:=\bigoplus_{i=1}^m E_i^{(k)},
	\qquad
	\langle u,v\rangle_{m}
	:=\frac1m\sum_{i=1}^m\langle u(x_i),v(x_i)\rangle.
	\label{eq:discrete-tangent-space}
\end{equation}

For every continuous differential $k$-form $\omega$ on $(M^n,g)$,
identified with a $k$-vector field via the metric $g$, define its
empirical restriction by
\begin{equation}
	R_m^k\omega\in\Hh_{m,k},
	\qquad
	(R_m^k\omega)(x_i)
	\coloneqq
	P_{x_i}^{(k)}\omega(x_i).
	\label{eq:rkm}
\end{equation}

Denote by $\mu_m = \frac{1}{m}\sum_{i=1}^m \delta_{x_i}$ the empirical probability measure associated with the sample $S_m$, and let $\mu$ be the uniform probability measure on $M^n$. To rigorously compare discrete sections on $S_m$ with continuum sections on $M^n$, we utilize the $TL^2$ framework introduced by Garc\'ia Trillos and Slep\v{c}ev \cite{GarciaTrillos2015} (see also \cite{Trillos2020}). 

Let $T_m:M^n\to S_m$ be an optimal transport map realizing the
$\infty$-Wasserstein distance $W_\infty(\mu,\mu_m)$; thus
\[
(T_m)_*\mu=\mu_m,
\qquad
\sup_{x\in M^n}d_M(x,T_m(x))=W_\infty(\mu,\mu_m).
\]
For a sequence of discrete ambient sections $v_m \in L^2(S_m, \Lambda^k \mathbb{R}^d, \mu_m)$ and a continuum section $\omega \in L^2(M^n, \Lambda^k \mathbb{R}^d, \mu)$, we say that $v_m$ converges to $\omega$ in the \emph{transported $L^2(\mu)$-topology} (or strong $TL^2$ sense) if
\begin{equation}
	\| v_m \circ T_m - \omega \|_{L^2(M^n, \Lambda^k \mathbb{R}^d, \mu)} \longrightarrow 0 \quad \text{as } m \to \infty.
\end{equation}
Similarly, $v_m$ converges weakly to $\omega$ in the transported $L^2(\mu)$-topology if $v_m \circ T_m \rightharpoonup \omega$ weakly in $L^2(M^n, \Lambda^k \mathbb{R}^d, \mu)$. In the remainder of this section, we assume that $M^n$ is a connected manifold.	

Following Kuwae and Shioya
\cite{Kuwae2003}*{Section~2.5, Definitions~2.11--2.13},
we use 
the following notion of compact Mosco convergence adapted
to the transported $L^2$-topology. Let
\[
F_m:L^2(S_m,\mu_m)\to(-\infty,+\infty]
\]
and
\[
F:L^2(M^n,\mu)\to(-\infty,+\infty]
\]
be lower semicontinuous quadratic forms. We say that $F_m$ converges
compactly to $F$ in the Mosco sense with respect to the transported
$L^2(\mu)$-topology if the following three conditions hold.

\begin{enumerate}
	\item[\rm (i)] \emph{Liminf inequality.} If $u_m\in L^2(S_m,\mu_m)$
	converges weakly to $u\in L^2(M^n,\mu)$ in the transported
	$L^2(\mu)$-sense, then
	\[
	F(u)\le \liminf_{m\to\infty}F_m(u_m).
	\]	
	\item[\rm (ii)] \emph{Recovery sequence.} For every
	$u\in L^2(M^n,\mu)$, there exists a sequence
	$u_m\in L^2(S_m,\mu_m)$ converging strongly to $u$ in the
	transported $L^2(\mu)$-sense such that
	\[
	F(u)\ge \limsup_{m\to\infty}F_m(u_m).
	\]	
	\item[\rm (iii)] \emph{Compactness.} If a sequence
	$u_m\in L^2(S_m,\mu_m)$ satisfies
	\[
	\sup_m\bigl(\|u_m\|_{L^2(\mu_m)}^2+F_m(u_m)\bigr)<+\infty,
	\]
	then it has a subsequence which converges strongly in the transported
	$L^2(\mu)$-sense to some $u\in L^2(M^n,\mu)$.
\end{enumerate}
Thus, in the terminology of
\cite{Kuwae2003}*{Definition~2.13}, compact convergence consists
of Mosco convergence together with asymptotic compactness.

\begin{remark}[\bf Probabilistic compact Mosco convergence and spectral limits]\
	\label{rem:prob_mosco}
	For each $m$ and each realization
	$S_m\in (M^n)^m$, let $F_m^{S_m}$ denote the corresponding empirical
	quadratic form, and choose a transport map $T_m^{S_m}$ satisfying the
	transport estimate used below.
	
	We say that $F_m$ converges compactly in the Mosco sense to $F$
	\emph{in probability}, in the uniform-good-event sense used in this
	paper, if there exist measurable events
	\[
	\mathcal G_m\subset (M^n)^m,
	\qquad
	\mu^m(\mathcal G_m)\longrightarrow1,
	\]
	such that the following holds: for every deterministic sequence of
	realizations
	\[
	S_m\in\mathcal G_m
	\]
	for all sufficiently large $m$, the corresponding deterministic
	sequence of transported forms $F_m^{S_m}$ converges compactly in the
	Mosco sense to $F$, namely, it satisfies conditions {\rm (i)}, {\rm
		(ii)}, and {\rm (iii)} above.
	
	This definition does not assert that a randomly chosen sequence of
	samples belongs to $\mathcal G_m$ eventually almost surely; no
	Borel--Cantelli argument is required. Rather, it asserts deterministic
	compact Mosco convergence uniformly over all sufficiently large
	realizations belonging to the good events.
	
	The deterministic spectral-convergence results of Kuwae and Shioya
	\cite{Kuwae2003} may therefore be applied to every deterministic
	selection $S_m\in\mathcal G_m$. In particular, for each fixed
	eigenvalue index $j$,
	\[
	\sup_{S_m\in\mathcal G_m}
	\left|
	\lambda_j(F_m^{S_m})-\lambda_j(F)
	\right|
	\longrightarrow0.
	\]
	Indeed, if this were false, one could choose a deterministic sequence
	$S_m\in\mathcal G_m$ along which the convergence fails, contradicting
	the defining property above. Consequently, for every $\varepsilon>0$,
	\[
	\mu^m\left\{
		S_m:
		\left|
		\lambda_j(F_m^{S_m})-\lambda_j(F)
		\right|>\varepsilon
		\right\}
	\le
	\mu^m(\mathcal G_m^c)
	\longrightarrow0.
	\]
	The same argument applies to spectral projections associated with a
	fixed isolated spectral cluster. Thus compact Mosco convergence in the
	above sense implies convergence in probability of the corresponding
	eigenvalues and isolated spectral subspaces.
\end{remark}

\begin{proposition}[\bf Scalar compact convergence for the localized
	Gaussian energy]\
	\label{prop:scalar_gaussian_mosco}
	Let $M^n\subset\mathbb R^d$ be a closed, connected, orientable
	$C^4$-smooth submanifold of dimension $n\ge3$, let $\mu$ be its
	uniform probability measure, let $S_m\sim\mu^m$, and set
	\[
	t=m^{-1/(2n)},
	\qquad
	h_m\coloneqq\sqrt t=m^{-1/(4n)}.
	\]
The scalar empirical energies 
	\[
	\mathcal E_{m,t}(u)
	\coloneqq
	\frac{\vol_g(M^n)}{2m^2t}
	\sum_{i,j=1}^m
	\Phi_t(x_i,x_j)\chi_\delta(x_i,x_j)
	|u(x_i)-u(x_j)|^2
	\]
converge compactly in the Mosco sense, in probability, with respect to the transported
	$L^2(\mu)$-topology, to the functional
	\[
	\mathcal E(u)
	\coloneqq
	\begin{cases}
		\displaystyle\int_{M^n}\|\nabla u\|^2\,d\mu,
		& u\in H^1(M^n,\mu) \\
		+\infty,
		&u\in L^2(M^n,\mu)\setminus H^1(M^n,\mu).
	\end{cases}
	\]
\end{proposition}

\begin{proof}  By the infinity-transport estimate
	\cite{Trillos2020}*{Theorem 2,   Eq. (1.14)}, on an event
	of probability tending to  $1$ \footnote{More precisely, \cite{Trillos2020}*{Theorem~2} shows that,
		for every $\beta>1$, the transport estimate holds with probability
		at least $1-Cm^{-\beta}$, where the constants depend on the geometric
		data and on $\beta$. Combining this estimate with the quantitative
		probability bound in Theorem~\ref{thm:empiconvergence} yields a
		corresponding quantitative probability statement in
		Proposition~\ref{prop:scalar_gaussian_mosco}. Since $n\ge3$, one may
		choose $\beta>1>2/n$, so the transport failure probability is of
		smaller order than $m^{-2/n}$.},
	 there exists a transport map
	\(T_m:M^n\to S_m\) such that
	\begin{equation}
	\varepsilon_m
	\coloneqq
	\sup_{x\in M^n}d_{M^n}(x,T_m(x))
	\le
	C\frac{(\log m)^{1/n}}{m^{1/n}}. \label{eq:trillos2020}
\end{equation}

	Since \(h_m=\sqrt t=m^{-1/(4n)}\), it follows from
	\eqref{eq:trillos2020} that
	\[
	\frac{\varepsilon_m}{h_m}
	\le
	C(\log m)^{1/n}m^{-3/(4n)}
	\longrightarrow0.
	\]
	
	Thus the transport displacement is negligible compared with the
	interaction scale $\sqrt t$ of the Gaussian kernel.
	
	The transport and interpolation estimates below hold on measurable
	events $\mathcal G_m^{\mathrm{tr}}$ with
	$\mu^m(\mathcal G_m^{\mathrm{tr}})\to1$. The liminf and compactness
	arguments are deterministic for every sequence of realizations lying in
	these events eventually. Universal good events for the recovery
	condition are constructed at the end of the recovery argument.
	
	Write
	\[
	\eta(z)
	\coloneqq
	(4\pi)^{-n/2}e^{-|z|^2/4},
	\qquad
	\Phi_t(x,y)
	=	
	h_m^{-n}
	\eta\left(\frac{x-y}{h_m}\right).
	\]
	For each $R>0$, choose a nonincreasing cutoff
	$\vartheta_R:[0,\infty)\to[0,1]$ such that
	\[
	\vartheta_R=1\quad\text{on }[0,R],
	\qquad
	\vartheta_R=0\quad\text{on }[R+1,\infty),
	\]
	and set
	\begin{equation}
	\eta_R(z)
	\coloneqq
	\eta(z)\vartheta_R(|z|).\label{eq:etar}
	\end{equation}
	Let
	\begin{equation}
	\sigma_R
	\coloneqq
	\int_{\mathbb R^n}z_1^2\eta_R(z)\,dz.\label{eq:surface_tension}
	\end{equation}
	Then $0\le\eta_R\le\eta$, and dominated convergence gives
	\[
	\sigma_R\longrightarrow
	\int_{\mathbb R^n}z_1^2\eta(z)\,dz
	=2.
	\]
	
	For fixed $R$ and sufficiently large $m$,
	\[
	(R+1)h_m<\frac{\delta}{2}.
	\]
	Consequently, $\chi_\delta(x_i,x_j)=1$ whenever
	$\eta_R((x_i-x_j)/h_m)\ne0$. Define
	\[
	\mathcal E_{m,t}^{(R)}(u)
	\coloneqq
	\frac{\vol_g(M^n)}
	{2m^2h_m^{n+2}}
	\sum_{i,j=1}^m
	\eta_R\left(\frac{x_i-x_j}{h_m}\right)
	|u(x_i)-u(x_j)|^2.
	\]
	Since the integrands are nonnegative and $\eta_R\le\eta$,
	\[
	\mathcal E_{m,t}(u)
	\ge
	\mathcal E_{m,t}^{(R)}(u).
	\]
	 For each fixed $R\ge1$,  recalling \eqref{eq:etar}, set
	\[
	a_R\coloneqq\int_{\mathbb R^n}\eta_R(z)\,dz,
	\qquad
	H_{m,R}\coloneqq(R+1)h_m,
	\]
	and define
	\[
	\overline\eta_R(z)
	\coloneqq
	\frac{(R+1)^n}{a_R}\,\eta_R\bigl((R+1)z\bigr).
	\]
	Then $\overline\eta_R$ is a nonnegative, radial, decreasing,
	Lipschitz kernel supported in the unit ball and satisfying
	\[
	\int_{\mathbb R^n}\overline\eta_R(z)\,dz=1.
	\]
	Moreover,
	\[
	h_m^{-n}\eta_R\left(\frac{x-y}{h_m}\right)
	=	
	a_RH_{m,R}^{-n}
	\overline\eta_R\left(\frac{x-y}{H_{m,R}}\right),
	\]
	and,  recalling \eqref{eq:surface_tension}, we have
	\[
	\sigma_{\overline\eta_R}
	=	
	\frac{\sigma_R}{a_R(R+1)^2}.
	\]

	Since
	\[
	\frac{\varepsilon_m}{H_{m,R}}
	=	
	\frac{\varepsilon_m}{(R+1)h_m}
	\longrightarrow0,
	\]
	the hypotheses of the interpolation estimates in
	\cite{Trillos2020}*{Lemma~14(ii) and equation~(4.7)} apply to the
	kernel $\bar\eta_R$ with bandwidth $H_{m,R}$. Hence there exist
	interpolation operators
	\[
	I_{m,R}:L^2(S_m,\mu_m)\longrightarrow H^1(M^n,\mu)
	\]
	such that
	\[
	\left\|
	I_{m,R}u_m-u_m\circ T_m
	\right\|_{L^2(\mu)}^2
	\le
	C_RH_{m,R}^2\,\mathcal G_{m,R}(u_m),
	\]
	and
	\[
	\frac{\sigma_{\overline\eta_R}}{2}
	\int_{M^n}\|\nabla I_{m,R}u_m\|^2\,d\mu
	\le
	(1+o(1))\mathcal G_{m,R}(u_m),
	\]
	where $\mathcal G_{m,R}$ denotes the graph energy formed with the
	normalized kernel $\bar\eta_R$ at scale $H_{m,R}$.
	
	By the above rescaling,
	\[
	\mathcal G_{m,R}(u_m)
	=		\frac{1}{a_R(R+1)^2}\,
	\mathcal E_{m,t}^{(R)}(u_m).
	\]
	Multiplying the gradient estimate by $a_R(R+1)^2$, and absorbing the
	factor $H_{m,R}^2=(R+1)^2h_m^2$ into the $R$-dependent constant in
	the $L^2$-estimate, we obtain
	\begin{equation}
		\left\|
		I_{m,R}u_m-u_m\circ T_m
		\right\|_{L^2(\mu)}^2
		\le
		C_Rh_m^2\mathcal E_{m,t}^{(R)}(u_m),
		\label{eq:scalar-interpolation-L2}
	\end{equation}
	and
	\begin{equation}
		\frac{\sigma_R}{2}
		\int_{M^n}\|\nabla I_{m,R}u_m\|^2\,d\mu
		\le
		(1+o(1))\mathcal E_{m,t}^{(R)}(u_m).
		\label{eq:scalar-interpolation-energy}
	\end{equation}
	Here the $o(1)$-term is taken as $m\to\infty$ for fixed $R$.
	Since $R$ is fixed and $h_m\to0$, the support of the truncated
	kernel lies inside the region where the spatial cutoff
	$\chi_\delta$ is identically one for all sufficiently large $m$.

	\smallskip
	\noindent
	{\it Mosco liminf inequality.}
	Suppose that
	\[
	u_m\circ T_m\rightharpoonup u
	\quad\text{weakly in }L^2(M^n,\mu).
	\]
	We may assume that
	\[
	\liminf_{m\to\infty}\mathcal E_{m,t}(u_m)<\infty
	\]
	and pass to a subsequence realizing this liminf. For every fixed
	$R>0$, the sequence
	$\mathcal E_{m,t}^{(R)}(u_m)$ is bounded. It follows from \eqref{eq:scalar-interpolation-L2}, the boundedness of
	$\mathcal E_{m,t}^{(R)}(u_m)$, and $h_m\to0$ that
	\[
	I_{m,R}u_m-u_m\circ T_m
	\longrightarrow0
	\quad\text{strongly in }L^2(\mu).
	\]
	Since
	\[
	u_m\circ T_m\rightharpoonup u
	\quad\text{weakly in }L^2(\mu),
	\]
	the transported functions are uniformly bounded in $L^2(\mu)$.
	The preceding strong convergence therefore implies that
	${I_{m,R}u_m}$ is also uniformly bounded in $L^2(\mu)$, and,
	moreover,
	\[
	I_{m,R}u_m\rightharpoonup u
	\quad\text{weakly in }L^2(\mu).
	\]
	Together with \eqref{eq:scalar-interpolation-energy}, this shows that
	${I_{m,R}u_m}$ is bounded in $H^1(M^n)$.
	 Weak lower
	semicontinuity therefore gives
	\[
	\frac{\sigma_R}{2}
	\int_{M^n}\|\nabla u\|^2\,d\mu
	\le
	\liminf_{m\to\infty}
	\mathcal E_{m,t}^{(R)}(u_m)
	\le
	\liminf_{m\to\infty}
	\mathcal E_{m,t}(u_m).
	\]
	Letting $R\to\infty$ and using $\sigma_R\to2$, we obtain
	\[
	\mathcal E(u)
	\le
	\liminf_{m\to\infty}\mathcal E_{m,t}(u_m).
	\]
	
	\smallskip
	\noindent
	{\it Recovery sequence.}
	Let $u\in C^\infty(M^n)$, and define
	\[
	u_m(x_i)\coloneqq u(x_i).
	\]
	Since $u$ is uniformly continuous and $\varepsilon_m\to0$,
	\[
	u_m\circ T_m\longrightarrow u
	\quad\text{strongly in }L^2(M^n,\mu).
	\]
	Moreover, in degree $k=0$, the empirical Hodge operator is precisely
	the scalar localized Gaussian diffusion operator since all  operators  $(\widehat\Ww _{t, S_m})_p$, $(\widehat \Rr_{t, S_m})_p$,  $(\widehat \Rr^{(1)}_{t, S_m})_p$ vanish for $k =0$. 
	Therefore,
	\begin{equation}
	\mathcal E_{m,t}(u_m)
	=
	\left\langle
	\widehat\Delta^0_{t,S_m}u,u
	\right\rangle_m.\label{eq:energy0}
	\end{equation}
	Indeed, writing $u_i=u(x_i)$ and using the symmetry
	\[
	\Phi_t(x_i,x_j)\chi_\delta(x_i,x_j)
	=
	\Phi_t(x_j,x_i)\chi_\delta(x_j,x_i),
	\]
	we obtain
	\[
	\begin{aligned}
		\left\langle
		\widehat\Delta^0_{t,S_m}u,u
		\right\rangle_m
		&=
		\frac{\vol_g(M^n)}{m^2t}
		\sum_{i,j=1}^m
		\Phi_t(x_i,x_j)\chi_\delta(x_i,x_j)
		(u_i-u_j)u_i\\
		&=
		\frac{\vol_g(M^n)}{2m^2t}
		\sum_{i,j=1}^m
		\Phi_t(x_i,x_j)\chi_\delta(x_i,x_j)
		|u_i-u_j|^2\\
		&=
		\mathcal E_{m,t}(u|_{S_m}).
	\end{aligned}
	\]
	
	By \eqref{eq:energy0}
	\[
	\mathcal E_{m,t}(u_m)
	=  \langle
	\widehat\Delta^0_{t,S_m}u,u
	\rangle_m = 	\frac{1}{m}\sum_{ i =1} ^m  u(x_i) (\widehat\Delta^0_{t,S_m}u)(x_i).
	\]
Hence
	\[
	\begin{aligned}
		&\left|
		\mathcal E_{m,t}(u_m)
		-	
		\int_{M^n}u\,\Delta_gu\,d\mu
		\right|\\
		&\quad\le
		\|u\|_{C^0(M^n)}
		\left\|
		\widehat\Delta^0_{t,S_m}u-\Delta_gu
		\right\|_{C^0(S_m)}\\
		&\qquad+
		\left|
		\frac1m\sum_{i=1}^m
		u(x_i)\Delta_gu(x_i)
		-
		\int_{M^n}u\,\Delta_gu\,d\mu
		\right|.
	\end{aligned}
	\]
	The first term converges to zero in probability by
	Theorem~\ref{thm:empiconvergence}, and the second converges to zero
	almost surely by the law of large numbers. Therefore,
	\[
	\mathcal E_{m,t}(u_m)
	\longrightarrow
	\int_{M^n}u\,\Delta_gu\,d\mu
	\]
	in probability. Since $M^n$ is closed and
	$\Delta_g=-\operatorname{div}\nabla$, integration by parts gives
	\[
	\int_{M^n}u\,\Delta_gu\,d\mu
	=	
	\int_{M^n}\|\nabla u\|^2\,d\mu.
	\]
	Consequently,
	\[
	\mathcal E_{m,t}(u_m)
	\longrightarrow
	\int_{M^n}\|\nabla u\|^2\,d\mu
	\]
	in probability.
	
	To obtain the uniform-good-event formulation of
	Remark~\ref{rem:prob_mosco}, fix a countable set
	\[
	\mathcal D=\{u^\ell:\ell\ge1\}\subset C^\infty(M^n)
	\]
	which is dense in $H^1(M^n,\mu)$. For each $L\ge1$, the transport
	estimate and the preceding empirical-operator and law-of-large-numbers
	estimates for the finitely many functions $u^1,\ldots,u^L$ yield an
	integer $M_L$ and measurable events $\mathcal G_{m,L}$ such that, for
	all $m\ge M_L$,
	\[
	\mu^m(\mathcal G_{m,L})\ge1-L^{-1},
	\]
	and all corresponding recovery errors are at most $L^{-1}$. Choose
	$M_L$ increasing, set
	\[
	N_m\coloneqq\max\{L:M_L\le m\},
	\qquad
	\mathcal G_m\coloneqq\mathcal G_{m,N_m},
	\]
	and enlarge $M_L$, if necessary, so that the deterministic transport
	errors for the first $L$ core functions are also at most $L^{-1}$.
	Then $N_m\to\infty$ and $\mu^m(\mathcal G_m)\to1$. 
	
	Fix $u\in H^1(M^n,\mu)$. Choose $u^{\ell_q}\in\mathcal D$ such that
	\[
	\|u^{\ell_q}-u\|_{H^1}\le q^{-1}.
	\]
	For each $q$, choose $m_q$ sufficiently large that $m_q<m_{q+1}$,
	$N_m\ge\ell_q$ for $m\ge m_q$, and, for every realization
	$S_m\in\mathcal G_m$ and every $m\ge m_q$,
	\[
	\bigl|(R_m u^{\ell_q})\circ T_m-u^{\ell_q}\bigr|_{L^2}
	\le q^{-1},
	\qquad
	\mathcal E_{m,t}(R_m u^{\ell_q})
	\le
	\mathcal E(u^{\ell_q})+q^{-1}.
	\]
	Set
	\[
	q(m)\coloneqq\max{q:m_q\le m},
	\qquad
	v_m\coloneqq R_m u^{\ell_{q(m)}}.
	\]
	Then $q(m)\to\infty$,
	\[
	v_m\circ T_m\longrightarrow u
	\quad\text{strongly in }L^2(M^n,\mu),
	\]
	and, since $u^{\ell_q}\to u$ in $H^1(M^n,\mu)$,
	\[
	\limsup_{m\to\infty}\mathcal E_{m,t}(v_m)
	\le
	\lim_{q\to\infty}\mathcal E(u^{\ell_q})
	=
	\mathcal E(u).
	\]
	Thus ${v_m}$ is a recovery sequence for $u$.

 If
	$u\in L^2(M^n,\mu)\setminus H^1(M^n,\mu)$, choose any strongly
	transported approximation of $u$ by restrictions of smooth functions;
	the recovery inequality is automatic because $\mathcal E(u)=+\infty$.
	
	\smallskip
	\noindent
	{\it Asymptotic compactness.}
	Suppose that
	\[
	\sup_m
	\left(
	\|u_m\|_{L^2(\mu_m)}^2+
	\mathcal E_{m,t}(u_m)
	\right)<\infty.
	\]
	Fix $R_0>0$ with $\sigma_{R_0}>0$. Since
	\[
	\mathcal E_{m,t}^{(R_0)}(u_m)
	\le
	\mathcal E_{m,t}(u_m),
	\]
	\eqref{eq:scalar-interpolation-energy} shows that
	${I_{m,R_0}u_m}$ is bounded in $H^1(M^n)$.
	By the Rellich compactness theorem, it has a subsequence converging
	strongly in $L^2(M^n,\mu)$. In view of
	\eqref{eq:scalar-interpolation-L2}, the corresponding transported
	functions $u_m\circ T_m$ converge strongly to the same limit.
	Thus ${u_m}$ is precompact in the transported
	$L^2(\mu)$-topology.
\end{proof}

For $t, C  > 0$  and $v \in \Hh_{m, k}$  we set
\[	
\mathcal{Q}_{m,t}(v) \coloneqq \langle \widehat{\Delta}^k_{t,S_m}v, v \rangle_m, \qquad 
\mathcal{Q}_{m,t} ^C(v) \coloneqq  \mathcal{Q}_{m,t}(v) + C\|v\|_m^2.
\]
Recall that we identify  $TM^n$ with $T^* M^n$ via the Riemannian metric $g$.
\begin{lemma}[\bf Mosco convergence of the empirical Hodge forms]
	\label{lem:Hodge_Mosco}\
	Let $M^n\subset\mathbb R^d$ be a closed connected orientable
	$C^4$-smooth submanifold of dimension $n\ge3$, let $0\le k\le n$,
	and set $t=m^{-1/(2n)}$. Then the transported empirical quadratic
	forms on the bundle space $\Hh_{m,k}$:
	\[
	\mathcal{Q}_{m,t}(v) 
	\]
	converge in probability, in the Mosco sense, to the continuous Hodge energy 
	\begin{equation}
		\mathcal{Q}_k(\omega) \coloneqq \langle\Delta^k\omega, \omega\rangle_{L^2(\mu)} = \int_{M^n} \|\nabla \omega\|^2 \, d\mu - \int_{M^n} \langle \mathcal{R}_k \omega, \omega \rangle \, d\mu \label{eq:qk}
	\end{equation}
	on $L^2(\Lambda^k TM^n,\mu)$. The form domain of $\mathcal Q_k$ is
	\[
	\operatorname{Dom}(\mathcal Q_k)
	=
	H^1(\Lambda^kTM^n,\mu).
	\]
	There exists $C>0$ such that every sequence
	$\{v_m\}\subset\Hh_{m,k}$ satisfying
	\[
	\sup_m
	\left(
	\|v_m\|_m^2+
	\Qq^{C}_{m,t}(v_m)
	\right)<\infty
	\]
	is precompact in the transported $L^2(\mu)$-topology.
\end{lemma}
\begin{proof}
	By the definition of the empirical Hodge operator, its quadratic form splits into the principal empirical diffusion energy and the zeroth-order potential terms.
	Indeed, writing
	\[
	v_i=v(x_i),
	\qquad
	K_{ij}=\Phi_t(x_i,x_j)\chi_\delta(x_i,x_j),
	\]
	and recalling that
	$v_i\in E_i^{(k)}=\operatorname{Ran}P_i^{(k)}$, we obtain
	\[
	\begin{aligned}
		\bigl\langle\widehat{\Lb}_{t,S_m}^kv,v\bigr\rangle_m
		&=
		\frac{\vol_g(M^n)}{m^2t}
		\sum_{i,j=1}^m
		K_{ij}
		\bigl\langle P_i^{(k)}(v_i-v_j),v_i\bigr\rangle\\
		&=
		\frac{\vol_g(M^n)}{m^2t}
		\sum_{i,j=1}^m
		K_{ij}\langle v_i-v_j,v_i\rangle.
	\end{aligned}
	\]
	Here we used the self-adjointness of $P_i^{(k)}$ and the identity
	$P_i^{(k)}v_i=v_i$. Since $K_{ij}=K_{ji}$, symmetrization in
	$i$ and $j$ yields
	\[
	\begin{aligned}
		2\sum_{i,j=1}^m
		K_{ij}\langle v_i-v_j,v_i\rangle
		&=
		\sum_{i,j=1}^m
		K_{ij}
		\left(
		\langle v_i-v_j,v_i\rangle
		+
		\langle v_j-v_i,v_j\rangle
		\right)\\
		&=
		\sum_{i,j=1}^m
		K_{ij}|v_i-v_j|^2.
	\end{aligned}
	\]
	Consequently,
	\[
	\bigl\langle\widehat{\Lb}_{t,S_m}^kv,v\bigr\rangle_m
	=	
	\frac{\vol_g(M^n)}{2m^2t}
	\sum_{i,j=1}^m
	\Phi_t(x_i,x_j)\chi_\delta(x_i,x_j)
	|v(x_i)-v(x_j)|^2.
	\]
	Substituting this identity into the definition of
	$\mathcal Q_{m,t}$ gives
	\begin{align}
		\mathcal{Q}_{m,t}(v)
		&=
		\frac{\vol_g(M^n)}{2m^2t}
		\sum_{i,j=1}^m
		\Phi_t(x_i,x_j)
		|v(x_i)-v(x_j)|^2
		\chi_\delta(x_i,x_j)
		\nonumber\\
		&\qquad
		-	
		\bigl\langle
		(\widehat{\Ww}_{t,S_m}
		+\widehat{\Rr}_{k,t,S_m}
		+\widehat{\Rr}^{(1)}_{k,t,S_m})v,
		v
		\bigr\rangle_m.
		\label{eq:mosco_split}
	\end{align}
	
	The first term evaluates differences directly in the flat ambient Euclidean space $\Lambda^k \mathbb{R}^d$. Let $\{e_\alpha\}_{\alpha=1}^D$ be a fixed orthonormal basis for $\Lambda^k \mathbb{R}^d$ (where $D = \binom{d}{k}$). For any discrete section $v \in \Hh_{m,k}$, we decompose it globally into its scalar components $v(x) = \sum_\alpha v^\alpha(x) e_\alpha$. The ambient difference energy term in  the right-hand side of \eqref{eq:mosco_split} splits  into a finite sum of scalar discrete Dirichlet energies:
	\begin{equation}
		\mathcal{Q}_{\mathrm{diff}, m, t}(v) = \sum_{\alpha=1}^D \left( \frac{\vol_g (M^n)}{2m^2t} \sum_{i,j=1}^m \Phi_t(x_i,x_j) |v^\alpha(x_i) - v^\alpha(x_j)|^2 \chi_\delta(x_i, x_j) \right).\label{eq:qdmt}
	\end{equation}
	
	Let $T_m: M^n \to S_m$ with $(T_m)_*\mu = \mu_m$ be the optimal transport maps defining the transported $L^2(\mu)$-topology, and set
	\begin{equation}
	\widetilde{v}_m(x) \coloneqq v_m(T_m(x)).\label{eq:tildevm}
	\end{equation}
	
	Recalling \eqref{eq:trillos2020}, on an event of probability tending to $1$,
	\begin{equation}
		\varepsilon_m
		\coloneqq
		\sup_{x\in M^n}d_M(x,T_m(x))
		\le
		C\frac{(\log m)^{1/n}}{m^{1/n}}. 
	\end{equation}
	
	Because the discrete forms satisfy the empirical fiber constraint at the sample points, by \eqref{eq:tildevm}, we have
	\[
	P_{T_m(x)}^{(k)}\widetilde{v}_m(x) = \widetilde{v}_m(x).
	\]
	Inserting this into the continuum constraint and applying the triangle inequality, we obtain:
	\begin{align}
\|(I-\Pi_x^{(k)})\widetilde v_m(x)\|
&\le
\|(P_{T_m(x)}^{(k)}-\Pi_{T_m(x)}^{(k)})\widetilde v_m(x)\|
+\|(\Pi_{T_m(x)}^{(k)}-\Pi_x^{(k)})\widetilde v_m(x)\|\nonumber\\
&\le
\Bigl(
\|P_{T_m(x)}^{(k)}-\Pi_{T_m(x)}^{(k)}\|_{\mathrm{op}}
+\|\Pi_{T_m(x)}^{(k)}-\Pi_x^{(k)}\|_{\mathrm{op}}
\Bigr)\|\widetilde v_m(x)\|.
\label{eq:Iminuspkx}
\end{align}
 The first term in the right-hand side of \eqref{eq:Iminuspkx} tends uniformly to zero by the empirical tangent convergence established in Proposition \ref{prop:empitangent_clean}. The second term tends uniformly to zero because the optimal transport distance satisfies
	\[
	\sup_{x \in M^n} d_M(x, T_m(x)) \longrightarrow 0,
	\]
	and the projector $x \mapsto \Pi_x^{(k)}$ is smooth, hence Lipschitz on the compact manifold $M^n$. 
	
	Since $\widetilde{v}_m$ is uniformly bounded in $L^2(M^n, \mu)$, the right-hand side of \eqref{eq:Iminuspkx} vanishes in $L^2(M^n, \mu)$ as $m \to \infty$. Consequently, if $\widetilde{v}_m \to v$ strongly (or weakly) in $L^2(M^n, \mu)$, then the limit satisfies
	\[
	(I-\Pi_x^{(k)})v(x) = 0 \quad \text{for a.e. } x \in M^n.
	\]
	Thus, the transported limit $v$ is an intrinsic $k$-form.

We next identify the limiting contribution of the diffusion term in
\eqref{eq:mosco_split}. Since
\[
\mathcal Q_{\mathrm{diff},m,t}(v)
=
\sum_{\alpha=1}^D
\mathcal E_{m,t}(v^\alpha),
\]
the scalar compact Mosco convergence of
Proposition~\ref{prop:scalar_gaussian_mosco}, applied componentwise,
shows that the scalar limit of the diffusion part is the functional
\[
\mathcal Q_{\mathcal L}(\zeta)
\coloneqq
\sum_{\alpha=1}^D
\int_{M^n}\|\nabla_M\zeta^\alpha\|^2\,d\mu,
\qquad
\zeta=\sum_{\alpha=1}^D\zeta^\alpha e_\alpha
\in L^2(M^n,\Lambda^k\mathbb R^d).
\]
Equivalently, if $\overline{\nabla}$ denotes the flat connection on
the trivial bundle $M^n\times\Lambda^k\mathbb R^d$, then
\[
\mathcal Q_{\mathcal L}(\zeta)
=
\int_{M^n}\|\overline{\nabla}\zeta\|^2\,d\mu
\]
whenever $\zeta\in H^1(M^n,\Lambda^k\mathbb R^d)$.

By the fiber-constraint argument above, every transported limit of
discrete sections in $\Hh_{m,k}$ lies in the intrinsic subbundle
$\Lambda^kTM^n$. We now identify
$\mathcal Q_{\mathcal L}$ on this intrinsic subbundle. Let
$\omega\in H^1(\Lambda^kTM^n,\mu)$, and first suppose that
$\omega$ is smooth. Choose a local orthonormal tangent frame
$\{E_i\}_{i=1}^n$. Since $\{e_\alpha\}_{\alpha=1}^D$ is a constant
orthonormal basis of the flat ambient space
$\Lambda^k\mathbb R^d$, we have
\[
\sum_{\alpha=1}^D\|\nabla_M\omega^\alpha\|^2
=
\sum_{i=1}^n
\|\overline{\nabla}_{E_i}\omega\|^2.
\]
The Euclidean covariant derivative decomposes into tangential and
normal components:
\[
\overline{\nabla}_{E_i}\omega
=
\nabla_{E_i}\omega
+
\nabla_{E_i}^{\perp}\omega,
\]
where $\nabla$ is the intrinsic Levi-Civita connection on
$\Lambda^kTM^n$. These two components are orthogonal in
$\Lambda^k\mathbb R^d$. Therefore
\[
\|\overline{\nabla}\omega\|^2
=
\|\nabla\omega\|^2
+
\sum_{i=1}^n\|\nabla_{E_i}^{\perp}\omega\|^2.
\]
The normal contribution is a zeroth-order expression determined by the
second fundamental form. With the notation used in the pointwise
expansion of the deterministic operator, it is
\[
\sum_{i=1}^n\|\nabla_{E_i}^{\perp}\omega\|^2
=
\bigl\langle
(\End_H(B)+\mathcal R^{(1)}_k)\omega,\omega
\bigr\rangle.
\]
Consequently,
\begin{align}
	\mathcal Q_{\mathcal L}(\omega)
	&=
	\sum_{\alpha=1}^D
	\int_{M^n}\|\nabla_M\omega^\alpha\|^2\,d\mu
	\nonumber\\
	&=
	\int_{M^n}\|\nabla\omega\|^2\,d\mu
	+
	\int_{M^n}
	\bigl\langle
	(\End_H(B)+\mathcal R^{(1)}_k)\omega,\omega
	\bigr\rangle\,d\mu.
	\label{eq:q1limit}
\end{align}
By density, the same identity holds for every
$\omega\in H^1(\Lambda^kT^*M^n,\mu)$.

Thus the diffusion term in \eqref{eq:mosco_split} has as its continuum
limit the intrinsic rough gradient energy together with the extrinsic
zeroth-order contribution
\[
\End_H(B)+\mathcal R^{(1)}_k.
\]
The remaining terms in \eqref{eq:mosco_split} are precisely the
empirical zeroth-order potentials which will cancel this extrinsic
contribution and leave the intrinsic Hodge energy.

	Choose $C>0$ so large that, on the common good event,
	\[
	CI - \widehat{\mathcal W}_{t,S_m} - \widehat{\mathcal R}_{k,t,S_m} - \widehat{\mathcal R}^{(1)}_{k,t,S_m}
	\]
	is nonnegative at every sample point. Then
	\begin{align}
		\Qq^C_{m, t} (v)  &= \Qq_{m, t} (v) + C \| v\|^2_m\nonumber\\
		&= \Qq_{\mathrm{diff},m, t}(v) + \la (CI - \widehat{\mathcal W}_{t,S_m} - \widehat{\mathcal R}_{k,t,S_m} - \widehat{\mathcal R}^{(1)}_{k,t,S_m}) v, v \ra_m\nonumber\\
		&\ge \Qq_{\mathrm{diff},m, t}(v). \label{eq:q2limit}
	\end{align}
	
	We prove Mosco convergence of $\mathcal Q_{m,t}^C$; subtracting the fixed term $C\|\cdot\|_m^2$ then gives the asserted convergence of $\mathcal Q_{m,t}$.
	
	\smallskip
	\noindent
	{\it Liminf inequality.} 
	Suppose that $v_m\in\Hh_{m,k}$ converges weakly in the transported $L^2$-topology to $v$, and assume without loss of generality that 
	\[
	\liminf_{m\to\infty}\mathcal Q_{m,t}^C(v_m)<\infty.
	\]
	Pass to a subsequence (still denoted by $v_m$) that realizes this $\liminf$ as an actual limit. By \eqref{eq:q2limit}, the ambient diffusion energies $\mathcal Q_{\mathrm{diff},m,t}(v_m)$ are uniformly bounded along this subsequence. 
	
	Writing $v_m = \sum_{\alpha=1}^{D} v_m^\alpha e_\alpha$ in a fixed ambient orthonormal basis, Proposition \ref{prop:scalar_gaussian_mosco} applies to each scalar component. Since $D<\infty$, we may extract a further subsequence along which $v_m$ converges \emph{strongly} in the transported $L^2$-topology. Because the original sequence converged weakly to $v$, the strong limit of this sub-subsequence must also be $v$.
	
	The empirical fiber constraint gives $P_{T_m(x)}^{(k)}\widetilde{v}_m(x) = \widetilde{v}_m(x)$.  
	Since
	\[
	(I-\Pi^{(k)})\widetilde v_m\longrightarrow0
	\quad\text{strongly in }L^2,
	\]
	while
	\[
	\widetilde v_m\longrightarrow v
	\quad\text{strongly in }L^2,
	\]
	the boundedness of the multiplication operator $(I-\Pi^{(k)})$ implies
	\[
	(I-\Pi^{(k)})v=0.
	\]
	Thus, $v$ is an intrinsic $k$-form.
	
	Because we are now operating on a strongly convergent subsequence, the scalar $\liminf$ inequality in Proposition \ref{prop:scalar_gaussian_mosco} applies component-wise to the diffusion energy:
	\begin{equation}
		\liminf_{m\to\infty} \mathcal Q_{\mathrm{diff},m,t}(v_m) \ge \sum_{\alpha=1}^{D} \int_{M^n}\|\nabla_Mv^\alpha\|^2\,d\mu = \mathcal{Q}_{\mathcal{L}}(v). \label{eq:q3limit}
	\end{equation}
	We claim that  
	\begin{equation}
		\begin{aligned}
			&\left\langle \bigl( \widehat{\mathcal W}_{t,S_m} + \widehat{\mathcal R}_{k,t,S_m} + \widehat{\mathcal R}^{(1)}_{k,t,S_m} \bigr)v_m,v_m \right\rangle_m \\
			&\qquad\longrightarrow \int_{M^n} \left\langle \bigl( \End_H(B) + \mathcal R_k + \mathcal R^{(1)}_k \bigr)v,v \right\rangle\,d\mu.\label{eq:q5limit}
		\end{aligned}
	\end{equation}
	
	To show \eqref{eq:q5limit}, let 
	\[ Z_m \coloneqq \widehat{\Ww}_{t, S_m} + \widehat{\mathcal R}_{k,t,S_m} + \widehat{\mathcal R}^{(1)}_{k,t,S_m}, \qquad Z\coloneqq \End_H(B) + \mathcal R_k + \mathcal R^{(1)}_k. \]
	Then 
	\[
	\begin{aligned}
		\sup_{x\in M^n} \|Z_m(T_m(x))-Z(x)\|_{\mathrm{op}}
		&\le
		\sup_{y\in M^n}\|Z_m(y)-Z(y)\|_{\mathrm{op}}\\
		&\quad+
		\sup_{x\in M^n}\|Z(T_m(x))-Z(x)\|_{\mathrm{op}}
		\longrightarrow0.
	\end{aligned}
	\]
	Here the first term tends to zero by uniform empirical convergence, whereas the second tends to zero by continuity of $Z$ and $\varepsilon_m\to0$. Since $\widetilde v_m\to v$ strongly in $L^2(\mu)$, it follows that
	\[
	\left\langle Z_mv_m,v_m\right\rangle_m \longrightarrow \int_{M^n}\langle Z(x)v(x),v(x)\rangle\,d\mu(x)
	\]
	which is exactly \eqref{eq:q5limit}.	
	Combining the estimates in \eqref{eq:q1limit}, \eqref{eq:q3limit}, and \eqref{eq:q5limit}, and observing the perfect algebraic cancellation of the $\End_H(B)$ and $\mathcal{R}^{(1)}_k$ trace artifacts, we recover exactly \eqref{eq:qk}:
	
	\[
	\liminf_{m\to\infty} \mathcal Q_{m,t}^C(v_m) \ge \mathcal Q_k(v)+C\|v\|_{L^2(\mu)}^2.
	\]
	
	\smallskip
	\noindent
	{\it Recovery sequence.}
	Let $\omega\in C^\infty(\Lambda^kTM^n)$, and  recall  \eqref{eq:rkm}
	\[
	(R_m^k\omega)(x_i) = P_{x_i}^{(k)}\omega(x_i).
	\]
	Set
	\[
	v_m\coloneqq R_m^k\omega, \qquad v_m(x_i)=P_{x_i}^{(k)}\omega(x_i).
	\]
	Recalling definition  \eqref{eq:tildevm}  of  $\tilde v_m$,    then  we have
	\[
	\| \tilde  v_m  (x)-\om (x)\|\le \|\big( P^{(k)}_{T_m (x)}- \Pi^{(k)}_{T_m (x)}\big)\om \big(T_m (x)\big)\|  + \|\om \big(T_m (x)\big) - \om (x)\|.
	\] 
The first term  vanishes   by Proposition \ref{prop:empitangent_clean}, and the second  by smoothness of $\om$ and $\eps_m\to 0$.	Hence $v_m\to\om$ strongly in the transported $L^2(\mu)$-topology. Moreover, because the empirical operator depends only on the projected sample values,   and
\[
\hat \Pi_{x_j} v_m  (x_j) = \hat \Pi_{x_j} \hat \Pi_{x_j} \om (x_j) = \hat\Pi_{x_j} \om (x_j), 
\]
so replacing   $\om(x_j)$ by $v_m (x_j)$ in the operator  leaves  the result unchanged.  Thus, 
	\[
	\widehat\Delta^k_{t,S_m}v_m = \bigl( \widehat\Delta^k_{t,S_m}\omega \bigr)\big|_{S_m}.
	\]
	Therefore, the discrete energy evaluates exactly to
	\[
	\mathcal Q_{m,t}(v_m) = \left\langle \widehat\Delta^k_{t,S_m}\omega, R_m^k\omega \right\rangle_m.
	\]
	
	Now we compute 
	\begin{align}
		&\Big | \left\langle \widehat\Delta^k_{t,S_m}\omega, R_m^k\omega \right\rangle_m -  \int_{M^n} \left\langle \Delta^k\omega, \omega \right\rangle \, d\mu  \Big |\nonumber \\
		&\le \sup_{x\in M^n} \Big | \left\langle \widehat\Delta^k_{t,S_m}\omega(x), P_x^{(k)}\om (x) \right\rangle  - \la \Delta^k \om (x), \om (x) \ra \Big |\nonumber \\
		&\quad+ \left | \frac{1}{m}\sum_{i=1}^m \la \Delta^k \om (x_i), \om (x_i) \ra - \int_{M^n} \left\langle \Delta^k\omega, \omega \right\rangle \, d\mu \right |.\label{eq:hmrec1}	
	\end{align}
	The first term of the right-hand side of \eqref{eq:hmrec1} tends uniformly to zero  since
	\begin{align*}
&\left|
\left\langle\widehat\Delta^k_{t,S_m}\omega(x),P_x^{(k)}\omega(x)\right\rangle
-\langle\Delta^k\omega(x),\omega(x)\rangle
\right|\\
&\qquad\le
\|(\widehat\Delta^k_{t,S_m}-\Delta^k)\omega\|_{C^0}\|\omega\|_{C^0}
+
\|\widehat\Delta^k_{t,S_m}\omega\|_{C^0}
\|P_x^{(k)}-\Pi_x^{(k)}\|_{\mathrm{op}}
\|\omega\|_{C^0}.
\end{align*}
	by the uniform consistency estimate of Theorem~\ref{thm:empiconvergence}
	and Proposition~\ref{prop:empitangent_clean}, noting that $\Pi_x\om (x) = \om (x)$.  The second  term  of the right-hand side of \eqref{eq:hmrec1} tends to zero by the law of large numbers. Hence 
	\begin{equation*}
		\mathcal Q_{m,t}(v_m) \longrightarrow \int_{M^n} \left\langle \Delta^k\omega, \omega \right\rangle \, d\mu = \mathcal Q_k(\omega).
	\end{equation*}
	Taking into account that
	\[\|v_m\|_m \to \| \om\|_{L^2(\mu)}\] 
	we conclude that
	\[ \Qq^C_{m, t} (v_m) \to \Qq_k (\om) + C\| \om\|^2_{L^2 (\mu)}.\]
	To make the recovery event independent of the particular form, fix a
	countable smooth core
	\[
	\mathcal D_k=\{\omega^\ell:\ell\ge1\}
	\subset C^\infty(\Lambda^kTM^n)
	\]
	dense in $H^1(\Lambda^kTM^n,\mu)$. For each $L$, intersect the
	universal scalar-Mosco and geometric good events with the uniform
	consistency and law-of-large-numbers events appearing above for
	$\omega^1,\ldots,\omega^L$. As in the scalar proof, choose an increasing
	sequence $M_L$ and define $N_m$ and $\mathcal G_m$ diagonally so that
	$N_m\to\infty$, $\mu^m(\mathcal G_m)\to1$, and all recovery errors for
	the first $N_m$ core forms tend to zero. Hence, for every deterministic
	sequence $S_m\in\mathcal G_m$ eventually, a slow diagonal selection
	from $\mathcal D_k$ yields a recovery sequence for every element of the
	form domain. This verifies condition~{\rm(ii)} in the precise sense of
	Remark~\ref{rem:prob_mosco}.
	
	\smallskip
	\noindent
	{\it Asymptotic compactness.}  
	If $\sup_m (\|v_m\|_m^2+\mathcal Q_{m,t}^C(v_m))<\infty$, then by \eqref{eq:q2limit}, the ambient diffusion energies $\mathcal{Q}_{\mathrm{diff},m,t}(v_m)$ are uniformly bounded. Proposition \ref{prop:scalar_gaussian_mosco}, applied to the finitely many ambient scalar components, yields precompactness in the transported $L^2(\mu)$-topology. Thus $\mathcal Q_{m,t}^C$ converges compactly in the Mosco sense to
	\[
	\mathcal Q_k^C = \mathcal Q_k+C\|\cdot\|_{L^2(\mu)}^2.
	\]
	Since strong transported $L^2$-convergence implies convergence of the corresponding norms, addition or subtraction of the discrete and continuum mass terms preserves Mosco convergence. Consequently, $\mathcal Q_{m,t}$ converges compactly in the Mosco sense to $\mathcal Q_k$. 
\end{proof}

\begin{lemma}[Self-adjointness of the discrete empirical Hodge Laplacian]\label{lem:self_adjoint}
	For every $t>0$ and every sample $S_m$, the restriction of
	$\widehat\Delta^k_{t,S_m}$ to $\Hh_{m,k}$ is self-adjoint with respect to the normalized inner product \eqref{eq:discrete-tangent-space}.
\end{lemma}

\begin{proof}
	By \eqref{eq:hodgeempi},
	\[
	\widehat\Delta^k_{t,S_m}
	=\widehat{\Lb}^{k}_{t,S_m}
	-\widehat{\Ww}_{t,S_m}
	-\widehat{\Rr}_{k,t,S_m}
	-  \widehat{\Rr}^{(1)}_{k,t,S_m} .
	\]
	We verify that the four summands are self-adjoint on $\Hh_{m,k}$.
	
	Write $P_i=P_i^{(k)}$.  For $i\ne j$, the $(i,j)$ block of the diffusion operator is
	\[
	W_{ij}
	=-\frac{\vol_g(M^n)}{mt}\,
	\Phi_t(x_i,x_j)\chi_\delta(x_i,x_j)P_iP_j.
	\]
	Since the scalar kernel is symmetric and $P_i^*=P_i$, we have
	\[
	W_{ij}^*
	=-\frac{\vol_g(M^n)}{mt}\,
	\Phi_t(x_i,x_j)\chi_\delta(x_i,x_j)P_jP_i
	=W_{ji}.
	\]
	The diagonal blocks are scalar multiples of $P_i$ and are therefore self-adjoint.  Hence $\widehat{\Lb}^{k}_{t,S_m}$ is self-adjoint.
	
	At a fixed sample point  $x_i$, write the empirical mean-curvature potential as
	\[
	\widehat{\mathcal W}_{t,S_m}(x_i)
	=\sum_{a,b=1}^n c_{ab}\,
	\varepsilon(\hat e_b)\,\iota(\hat e_a^*)
	\]
	where $c_{ab}=c_{ba}$, $\varepsilon(v)$ denotes exterior multiplication by $v$, and $\iota(v^*)$ denotes contraction. $\widehat{\mathcal W}_{t,S_m}(x_i)$ is being regarded as an operator on $E_i ^{(k)}$. Since
	\[
	\varepsilon(v)^*=\iota(v^*),
	\qquad
	\iota(v^*)^*=\varepsilon(v),
	\]
	we obtain
	\[
	\bigl(\varepsilon(\hat e_b)\iota(\hat e_a^*)\bigr)^*
	=\varepsilon(\hat e_a)\iota(\hat e_b^*).
	\]
	The symmetry $c_{ab}=c_{ba}$ therefore implies
	$\hat{\mathcal W}_{t,S_m}^*(x_i)=\hat{\mathcal W}_{t,S_m}(x_i)$.

Similarly, for the partial trace operator, set
\[
\rho^{(1)}_{bc}\coloneqq \sum_{a=1}^n
\widetilde{(\hat R_{t,S_m})}_{x_i}
(\hat e_a,\hat e_b,\hat e_a,\hat e_c).
\]
Because the empirical curvature tensor is defined by the Gauss formula from the symmetric form $\hat B^{\mathrm{sym}}_{t, S_m}$, we have
\[
\rho^{(1)}_{bc} = \sum_{a=1}^n \left( \la \hat B^{\mathrm{sym}}_{t, S_m} (\hat e_a, \hat e_c), \hat B^{\mathrm{sym}}_{t, S_m}(\hat e_a, \hat e_b)\ra - \la  \hat B^{\mathrm{sym}}_{t, S_m} (\hat e_a, \hat e_a),  \hat B^{\mathrm{sym}}_{t, S_m} (\hat e_b, \hat e_c) \ra \right).
\]
Since $\hat B^{\mathrm{sym}}_{t,S_m}$ is symmetric in its arguments, $\rho^{(1)}_{bc}=\rho^{(1)}_{cb}$. Thus
\[
\widehat\Rr^{(1)}_{k,t,S_m}
=\sum_{b,c=1}^n\rho^{(1)}_{bc}\,
\varepsilon(\hat e_b)\iota(\hat e_c^*)
\]
is self-adjoint by the exact same adjoint calculation used for $\widehat{\mathcal W}_{t,S_m}(x_i)$.

Finally, we evaluate the full Weitzenb\"ock potential $\widehat{\Rr}_{k,t,S_m}$. Expanding it as a four-index sum yields:
\[
\widehat{\Rr}_{k,t,S_m} = \sum_{a,b,c,d=1}^n \hat R_{abcd} \, \varepsilon(\hat e_a)\iota(\hat e_b^*)\varepsilon(\hat e_d)\iota(\hat e_c^*),
\]
where $\hat R_{abcd} \coloneqq \widetilde{(\hat R_{t,S_m})}_{x_i}(\hat e_a,\hat e_b,\hat e_c,\hat e_d)$. Taking the adjoint reverses the sequence of the creation and annihilation operators:
\[
\widehat{\Rr}_{k,t,S_m}^* = \sum_{a,b,c,d=1}^n \hat R_{abcd} \, \varepsilon(\hat e_c)\iota(\hat e_d^*)\varepsilon(\hat e_b)\iota(\hat e_a^*).
\]
Relabeling the dummy summation indices $(c \to a, d \to b, b \to d, a \to c)$ yields:
\[
\widehat{\Rr}_{k,t,S_m}^* = \sum_{a,b,c,d=1}^n \hat R_{cdab} \, \varepsilon(\hat e_a)\iota(\hat e_b^*)\varepsilon(\hat e_d)\iota(\hat e_c^*).
\]
By the empirical Gauss formula,
\[
\hat R_{cdab} = \la \hat B^{\mathrm{sym}}_{t,S_m}(\hat e_c, \hat e_b), \hat B^{\mathrm{sym}}_{t,S_m}(\hat e_d, \hat e_a) \ra - \la \hat B^{\mathrm{sym}}_{t,S_m}(\hat e_c, \hat e_a), \hat B^{\mathrm{sym}}_{t,S_m}(\hat e_d, \hat e_b) \ra.
\]
By the symmetry of $\hat B^{\mathrm{sym}}_{t,S_m}$ and the symmetry of the inner product, this evaluates to exactly $\hat R_{abcd}$. Hence $\hat R_{cdab} = \hat R_{abcd}$, which implies $\widehat{\Rr}_{k,t,S_m}^* = \widehat{\Rr}_{k,t,S_m}$.

All four summands of $\widehat\Delta^k_{t,S_m}$ are self-adjoint on $\Hh_{m,k}$, which completes the proof.
\end{proof}

\begin{theorem}[\bf Convergence of the empirical
	harmonic cluster]\
	\label{thm:harmonic_cluster}
	Let $(M^n,g)\subset\R^d$ be a closed,  connected,  orientable 
	$C^4$-smooth submanifold of dimension $n\ge3$. For
	$0\le k\le n$, let
	\[
	\Hh^k(M^n)=\ker\Delta^k,
	\qquad
	b_k=\dim\Hh^k(M^n),
	\]
	and let $\lambda_{+,k}>0$ be the first positive eigenvalue of
	$\Delta^k$. Set
	\[
	\eta_k=\frac{\lambda_{+,k}}2.
	\]
	
	Let $Q_m^k$ denote the discrete $L^2$-orthogonal spectral
	projection of $\widehat\Delta^k_{t,S_m}$ onto the empirical
	eigenvalues contained in
	\[
	(-\eta_k,\eta_k),
	\]
	and set $t=m^{-1/(2n)}$.   Then the following assertions hold.
	
	\begin{enumerate}
		\item With probability tending to $1$, the interval
		$(-\eta_k,\eta_k)$ contains exactly $b_k$ empirical eigenvalues,
		counted with multiplicity. In particular,
		\[
		\dim\operatorname{Ran}Q_m^k=b_k.
		\]
		
		\item On the event in the first assertion, there exists an isometry
		\[
		U_m^{k,\mathrm{disc}}:
		\bigl(\Hh^k(M^n), \la \cdot, \cdot \ra_{L^2(\mu)}\bigr)\longrightarrow \bigl(\operatorname{Ran}Q_m^k, \la \cdot, \cdot \ra_m\bigr)
		\]
	such that
		\[
		\sup_{\substack{\omega\in\Hh^k(M^n)\\
				\|\omega\|_{L^2(\mu)}=1}}
		\left\|
		U_m^{k,\mathrm{disc}}\omega-R_m^k\omega
		\right\|_m
		\longrightarrow0
		\]
		in probability.
	\end{enumerate}

\end{theorem}

\begin{proof}   Let
	\[
	\widehat\lambda_{m,1}^{(k)}
	\le
	\widehat\lambda_{m,2}^{(k)}
	\le\cdots
	\]
	be the empirical eigenvalues, counted with multiplicity. By Lemma \ref{lem:Hodge_Mosco}, Remark \ref{rem:prob_mosco}, and the spectral convergence theorem for
	compactly Mosco-convergent quadratic forms
	\cite{Kuwae2003}*{Section~5}, for every fixed $j$,
	\[
	\widehat\lambda_{m,j}^{(k)}
	\longrightarrow
	\lambda_j^{(k)}
	\]
	in probability.

	Let
	\[
	\{\omega_1,\ldots,\omega_{b_k}\}
	\]
	be an $L^2(\mu)$-orthonormal basis of
	$\mathcal H^k(M^n)$, and define
	\[
	v_{\omega_a}(x_i)
	\coloneqq
	P_{x_i}^{(k)}\omega_a(x_i).
	\]
	Set
	\[
	V_{m,k}
	\coloneqq
	\operatorname{span}_{\mathbb R}
	\{v_{\omega_1},\ldots,v_{\omega_{b_k}}\}.
	\]

	\underline{Step 1.} {\it  Spectral Dimension and No Pollution.} 
	Since the empirical operator depends only on the projected sample
	values, for every smooth ambient $k$-form $\omega$,
	\[
	\widehat\Delta^k_{t,S_m}
	\bigl(P^{(k)}\omega|_{S_m}\bigr)
	=	
	\bigl(\widehat\Delta^k_{t,S_m}\omega\bigr)|_{S_m}.
	\]
	Moreover, on the empirical-projector event,
	\[
	\sup_{x\in M^n}
	\|P_x^{(k)}-\Pi_x^{(k)}\|_{\mathrm{op}}
	\le C_kt.
	\]
	Consequently,
	\[
	\left|
	\langle v_{\omega_a},v_{\omega_b}\rangle_m
	-	
	\frac1m\sum_{i=1}^m
	\langle\omega_a(x_i),\omega_b(x_i)\rangle
	\right|
	\le C_kt.
	\]
Because the harmonic forms are smooth and bounded on the compact
manifold $M^n$, Hoeffding's inequality applies to the scalar functions
\[
x\longmapsto\langle\omega_a(x),\omega_b(x)\rangle
\]
and gives
\[
\langle v_{\omega_a},v_{\omega_b}\rangle_m\longrightarrow\delta_{ab}
\]
with exponentially high probability.	
	
	Since $\Hh^k(M^n)$ is finite-dimensional, all its $C^4$- and
	$L^2(\mu)$-norms are equivalent. Applying
	Theorem~\ref{thm:empiconvergence} to a fixed basis and intersecting
	the finitely many resulting events therefore implies
	\[
	\sup_{\substack{\omega\in\mathcal H^k(M^n)\\
			\|\omega\|_{L^2 (\mu)}=1}}
	\left\|
	\widehat\Delta^k_{t,S_m}
	\bigl(P^{(k)}\omega|_{S_m}\bigr)
	\right\|_m
	\longrightarrow0.
	\]
	Together with the Gram-matrix convergence, this yields
	\[
	\sup_{\substack{v\in V_{m,k},\\ \|v\|_m=1}}
	\|\widehat\Delta^k_{t,S_m}v\|_m
	\longrightarrow0.
	\]
	
	If the spectral subspace of
	$\widehat\Delta^k_{t,S_m}$ corresponding to
	$(-\eta_k,\eta_k)$ had dimension less than $b_k$, there would be a
	unit vector $v\in V_{m,k}$ orthogonal to that spectral subspace.
	Self-adjointness and the spectral theorem would then give
	\[
	\|\widehat\Delta^k_{t,S_m}v\|_m\ge\eta_k,
	\]
	contradicting the preceding estimate. Thus the empirical cluster
	contains at least $b_k$ eigenvalues, counted with multiplicity.

\textit{Conclusion of Step 1.}

Since
\[
\lambda_1^{(k)}
= \cdots
 = \lambda_{b_k}^{(k)}
= 0,
\qquad
\lambda_{b_k+1}^{(k)}
=
\lambda_{+,k},
\]
it follows, with
\[
\eta_k=\frac{\lambda_{+,k}}2,
\]
that, with probability tending to $1$,
\[
|\widehat\lambda_{m,j}^{(k)}|<\eta_k,
\qquad
1\le j\le b_k,
\]
and
\[
\widehat\lambda_{m,b_k+1}^{(k)}>\eta_k.
\]
Since the empirical eigenvalues are ordered, exactly $b_k$
eigenvalues lie in $(-\eta_k,\eta_k)$, counted with multiplicity.

\

\underline{Step 2.} {\it Alignment in transported $L^2$.}

Let
\[
R_m^k:\Hh^k(M^n)\longrightarrow\Hh_{m,k}
\]
be the empirical restriction map defined by
\[
(R_m^k\omega)(x_i) \coloneqq P_{x_i}^{(k)}\omega(x_i),
\]
and let $Q_m^k$ be the discrete $L^2$-orthogonal spectral projection
onto the empirical cluster corresponding to $(-\eta_k,\eta_k)$. On the
high-probability event from Step~1, this projection has rank $b_k$. We
now prove its alignment with the empirical restriction map directly.

For every $\omega\in\Hh^k(M^n)$, the spectral theorem gives
\[
\eta_k \left\| (I-Q_m^k)R_m^k\omega \right\|_m
\le
\left\| \widehat\Delta^k_{t,S_m} (I-Q_m^k)R_m^k\omega \right\|_m.
\]
Since $Q_m^k$ is a spectral projection, it commutes with $\widehat\Delta^k_{t,S_m}$. Because $(I-Q_m^k)$ is an orthogonal projection (and thus has operator norm bounded by 1), we obtain
\[
\left\| \widehat\Delta^k_{t,S_m} (I-Q_m^k)R_m^k\omega \right\|_m
=
\left\| (I-Q_m^k)\widehat\Delta^k_{t,S_m}R_m^k\omega \right\|_m
\le
\left\| \widehat\Delta^k_{t,S_m}R_m^k\omega \right\|_m.
\]
The uniform residual estimate established in Step~1 gives
\[
\sup_{\substack{\omega\in\Hh^k(M^n)\\
\|\omega\|_{L^2(\mu)}=1}}
\|\widehat\Delta^k_{t,S_m}R_m^k\omega\|_m\longrightarrow0.
\]
Because $\eta_k>0$ is fixed, this implies
\[
\left\| (I-Q_m^k)R_m^k \right\|_{\Hh^k(M^n)\to\Hh_{m,k}} \longrightarrow 0
\]
in probability.

Since $\Hh^k(M^n)$ is finite-dimensional and the discrete Gram matrices converge to the continuum Gram matrix,
\[
(R_m^k)^*R_m^k \longrightarrow I_{\Hh^k(M^n)}
\]
in operator norm.

Define
\[
J_m^k\coloneqq Q_m^kR_m^k.
\]
Then
\[
\begin{aligned}
	(J_m^k)^*J_m^k
	&=
	(R_m^k)^*Q_m^kR_m^k\\
	&=
	(R_m^k)^*R_m^k - (R_m^k)^*(I-Q_m^k)R_m^k \longrightarrow I_{\Hh^k(M^n)}
\end{aligned}
\]
in operator norm. Hence $J_m^k$ is injective for all sufficiently large $m$. Since
\[
\dim\Hh^k(M^n)=\dim\operatorname{Ran}Q_m^k=b_k,
\]
it follows that
\[
\operatorname{Ran}J_m^k=\operatorname{Ran}Q_m^k.
\]

Define
\[
U_m^{k,\mathrm{disc}} \coloneqq J_m^k \bigl((J_m^k)^*J_m^k\bigr)^{-1/2}.
\]
Then $U_m^{k,\mathrm{disc}}$ is an isometric isomorphism from $\Hh^k(M^n)$ onto $\operatorname{Ran}Q_m^k$. Moreover,
\[
\begin{aligned}
	\|U_m^{k,\mathrm{disc}}-R_m^k\|_{\op}
	&\le
	\|J_m^k-R_m^k\|_{\op}\\
	&\quad+
	\|J_m^k\|_{\op} \left\| \bigl((J_m^k)^*J_m^k\bigr)^{-1/2}-I \right\|_{\op} \longrightarrow0.
\end{aligned}
\]

Indeed,
\[
\|J_m^k-R_m^k\|_{\op} = \|(I-Q_m^k)R_m^k\|_{\op}\longrightarrow0,
\]
while $(J_m^k)^*J_m^k\to I$ implies, by continuous functional calculus,
\[
\bigl((J_m^k)^*J_m^k\bigr)^{-1/2}\longrightarrow I.
\]
This completes the proof of Theorem \ref{thm:harmonic_cluster}.

\end{proof}

\section{Recovering the Pontryagin classes and periods}\label{sec:learn}

In this section, we keep the assumption that $(M^n,g)$ is a closed
orientable $C^3$-smooth submanifold of dimension $n \ge 2$ of $\mathbb R^d$, that $\mu$ is the
uniform distribution on $M^n$. We apply the results of
the previous sections to show that, from $\mu$-i.i.d. point clouds, one
can consistently recover 
the Pontryagin characteristic forms and
periods (Theorems~\ref{thm:pontryagin} and
\ref{thm:persistent_cycles}).

For $p \in M^n$ define  $\Om_p \in  \Lambda ^2  T_p M^n \otimes \mathfrak{so} (T_p M^n)$ by
$$ \la  \Om _p  (X, Y)  Z, W \ra  =  R_p  (X, Y, Z, W),$$
 and  define $\tilde  \Om_p \in \Lambda ^2  \R^d \otimes \End (\R^d)$ by:
$$ \la \tilde \Om _p  (X, Y)  Z, W \ra  = \tilde R_p  (X, Y, Z, W).$$
Similarly, we define  its empirical version by
$$\la (\widehat {\Om}_{t, S_m})_ p (X, Y) Z, W \ra  = (\widetilde{\hat  R_{t, S_m}})_p (X, Y, Z, W).$$

\begin{theorem}[\bf Consistency of empirical Pontryagin forms and periods]
	\label{thm:pontryagin}
	
	\
	Let $(M^n,g)\subset\R^d$ be a closed orientable
	$C^3$-smooth submanifold of dimension $n\ge 2$ and assume that
	\[
	t=m^{-1/(2n)}.
	\]
	For $1\le r\le\lfloor n/4\rfloor$, let
	$\widetilde{p_r(\Om)}$ denote the ambient extension of the intrinsic
	$r$-th Pontryagin form. Let
	\[
	p_r(\widehat\Omega_{t,S_m})
	\]
	be the ambient $4r$-form obtained by applying the universal
	Chern--Weil polynomial to the empirical curvature $2$-form
	$\widehat\Omega_{t,S_m}$, defined by
	\[
	\left\langle
	\widehat\Omega_{t,S_m,p}(X,Y)Z,W
	\right\rangle
	=	
	\widetilde{(\widehat R_{t,S_m})_p}(X,Y,Z,W).
	\]
	
	Then there exists $C_r>0$, depending only on the geometry of $M^n$
	and on $r$, such that, for all sufficiently small $t$, with
	probability at least $1-3m^{-2/n}$,
	\begin{equation}
		\sup_{p\in M^n}
		\left\|
		p_r(\widehat\Omega_{t,S_m})(p)
		-		
		\widetilde{p_r(\Om)}(p)
		\right\|_{\mathrm{comass}}
		\le C_r\sqrt t.
		\label{eq:pontryagin_rate}
	\end{equation}
	Consequently, for every smooth singular $4r$-cycle $Z$ in $M^n$,
	\begin{equation}
		\left|
		\int_Zp_r(\widehat\Omega_{t,S_m})
		-		
		\int_Zp_r(\Om)
		\right|
		\le
		C_r\mathbf M(Z)\sqrt t.
	\end{equation}
\end{theorem}

\begin{proof}
	The $r$-th Pontryagin form is obtained from a universal homogeneous
	invariant polynomial $P_r$ of degree $2r$ in the curvature
	$2$-form:
	\[
	p_r(\Omega)=P_r(\Omega,\ldots,\Omega).
	\]
	By the telescoping identity for multilinear polynomial,
	\begin{align*}
		&P_r(\widehat\Omega,\ldots,\widehat\Omega)
		-	
		P_r(\widetilde\Omega,\ldots,\widetilde\Omega)
				&\qquad=
		\sum_{q=1}^{2r}
		P_r(\widehat\Omega,\ldots,\widehat\Omega,
		\widehat\Omega-\widetilde\Omega,
		\widetilde\Omega,\ldots,\widetilde\Omega),
	\end{align*}
	where the difference occupies the $q$-th slot.
	
	The true curvature is uniformly bounded on $M^n$. Moreover, by
	Theorem~\ref{thm:Rcurv}, with probability at least
	$1-3m^{-2/n}$,
	\[
	\sup_{p\in M^n}
	\|
	\widehat\Omega_{t,S_m,p}-\widetilde\Omega_p
	\|_{\mathrm{op}}
	\le C\sqrt t.
	\]
	It follows that $\widehat\Omega_{t,S_m}$ is also uniformly bounded
	on the same event. Since $P_r$ is a fixed multilinear polynomial,
	\[
	\sup_{p\in M^n}
	\left\|
	p_r(\widehat\Omega_{t,S_m})(p)
	-	\widetilde{p_r(\Om)}(p)
	\right\|_{\mathrm{comass}}
	\le C_r\sqrt t.
	\]
	
	Finally, the mass-comass inequality gives
\[
	\left|
	\int_Z
	\left(
	p_r(\widehat\Omega_{t,S_m})-p_r(\Om)
	\right)
	\right|
	\le
	\mathbf M(Z)
	\sup_{p\in M}
	\left\|
	p_r(\widehat\Omega_{t,S_m})(p)
	-
	\widetilde{p_r(\Om)}(p)
	\right\|_{\mathrm{comass}},
\]
	which proves the asserted estimate.
\end{proof}

To evaluate the empirical characteristic numbers, we require not only the empirical Pontryagin forms but also empirical domains of integration. In Topological Data Analysis (TDA), the topology of the underlying manifold $M^n$ can be recovered from the point cloud $S_m$ by constructing a geometric simplicial complex, such as the \v{C}ech or Vietoris-Rips complex, at an appropriate proximity scale. Foundational results by Niyogi, Smale, and Weinberger \cite{NSW2008} guarantee that, with high probability, the homology of this empirical complex is isomorphic to the homology of $M^n$ for sufficiently dense samples.

However, integrating differential forms requires a specific geometric representative (a chain or cycle) rather than an abstract homology class. While persistent homology algorithms return specific simplicial generator cycles $Z_{t, S_m}^{4r}$, establishing analytic convergence of the integrals requires these empirical cycles to converge geometrically to a true smooth cycle $Z^{4r}$ in $M^n$. The rigorous framework for evaluating the convergence of integration domains is the theory of integral currents and the Whitney flat norm, introduced by Federer and Fleming \cite{Federer1969}.

Crucially, flat norm convergence alone permits sequences of cycles with highly oscillating, zig-zagging boundaries whose total $4r$-dimensional volume (mass) diverges to infinity. If the mass is unbounded, the $O(\sqrt{t})$ pointwise error of the empirical forms $\hat{R}_{t, S_m}$ will be magnified to infinity during integration. Therefore, we must constrain the sequence to have uniformly bounded
mass. Such representatives can be sought through the Optimal Homologous
Cycle problem, a discrete Plateau-type minimization problem
\cite{DHK2011}. The optimization is a linear program for real chains and
for integral chains when the relevant boundary matrix is totally
unimodular; in general dimensions, the integral problem may require
integer linear programming. 

Motivated by these geometric and topological guarantees, we formulate the convergence of empirical characteristic numbers over persistent cycles.

\begin{theorem}[\bf Empirical characteristic numbers from persistent cycles]\ \label{thm:persistent_cycles}
	Assume the hypotheses of Theorem \ref{thm:pontryagin}.  Let $U=U_\rho(M^n)$, where $0<\rho<\tau_{M^n}$, be a tubular neighborhood of $M^n$  in $\R^d$   with nearest  projection $\pi:  U \to M^n$. 
Let $Z_{t,S_m}^{4r}$ be a family of simplicial cycles in        $U$  obtained from a persistent homology reconstruction of the point cloud $S_m$.   Assume that, as $m\to\infty$, the random currents
$Z_{t,S_m}^{4r}$ converge in the Whitney flat norm to a smooth singular
cycle $Z^{4r}\subset M^n$ in probability, and that their masses are
uniformly bounded in probability as  $m \to \infty$  (equivalently, as $t \to 0$). For every $\varepsilon>0$ there exists $K>0$ such that, for all
 sufficiently large $m$,
 \[
 \mu^m\{S_m:\mathbf M(Z_{t,S_m}^{4r})\le K\}\ge 1-\varepsilon.
 \]

	Let
\[
\widehat\omega_{t,S_m}\coloneqq\pi^*p_r(\widehat\Omega_{t,S_m}),
\qquad
\omega\coloneqq\pi^*p_r(\Omega),
\]
be the extensions of the empirical and true Pontryagin forms to $U$,
respectively. Then
	\[
	\int_{Z_{t,S_m}^{4r}} \hat{\omega}_{t, S_m} \longrightarrow \int_{Z^{4r}} \omega
	\]
	in probability as $m\to\infty$.
\end{theorem}

\begin{proof}
	We treat the cycles $Z_m \coloneqq Z_{t,S_m}^{4r}$ and $Z \coloneqq Z^{4r}$ as integral currents in $\R^d$. Using the triangle inequality, we split the integration error into two components:
	\begin{align}
		\left| \int_{Z_m} \hat{\omega}_{t, S_m} - \int_Z \omega \right| &\le \underbrace{\left| \int_{Z_m} (\hat{\omega}_{t, S_m} - \omega) \right|}_{\text{Form Error (I)}} + \underbrace{\left| \int_{Z_m} \omega - \int_Z \omega \right|}_{\text{Cycle Error (II)}}. \label{eq:persist_split}
	\end{align}
	
	\medskip
	\noindent
	{\it Bounding Term (I):}
	By the definition of integration over a current, the first term is bounded by the mass of the cycle times the uniform ($C^0$) norm of the differential form difference:
	\[
	\text{(I)} \le   \mathbf M(Z_m) \cdot \sup_{x \in U} \| \hat{\omega}_{t, S_m}(x) - \omega(x) \|_{\Lambda^{4r}(\R^d)^*}.
	\]

Let
\[
C_{\pi,r}
\coloneqq
\sup_{z\in U}\|\Lambda^{4r}d\pi_z\|_{\mathrm{op}}<\infty.
\]
Then
\[
\sup_{z\in U}
\|\widehat\omega_{t,S_m}(z)-\omega(z)\|_{\mathrm{comass}}
\le
C_{\pi,r}
\sup_{p\in M^n}
\|p_r(\widehat\Omega_{t,S_m})(p)-\widetilde{p_r(\Omega)}(p)\|_{\mathrm{comass}}.
\]
By Theorem~\ref{thm:pontryagin}, the last supremum is bounded by
$C_r\sqrt t$ with probability at least $1-3m^{-2/n}$.

To prove convergence of Term~{\rm(I)}, let $\varepsilon,\eta>0$.
Uniform boundedness of the masses in probability gives $K>0$ such that,
for all sufficiently large $m$,
\[
\mu^m\{\mathbf M(Z_m)>K\}<\eta/2.
\]
For all sufficiently large $m$, we also have
$KC_{\pi,r}C_r\sqrt t<\varepsilon$, while the failure probability in
Theorem~\ref{thm:pontryagin} is smaller than $\eta/2$. Hence
\[
\mu^m\{\text{Term {\rm(I)}}>\varepsilon\}<\eta,
\]
which proves that Term~{\rm(I)} converges to zero in probability.
	
	\medskip
	\noindent
	{\it Bounding Term (II):}
	This term evaluates the fixed smooth form $\omega$ on the converging
	sequence of cycles. By assumption,
	$\mathcal F(Z_m-Z)\to0$ in probability. The flat-norm inequality gives
	\[
	|(Z_m-Z)(\omega)|
	\le
	\mathcal F(Z_m-Z)
	\bigl(\|\omega\|_\infty+\|d\omega\|_\infty\bigr).
	\]
	The Chern--Weil form $p_r(\Omega)$ is closed, and therefore
	\[
	d\omega=d\bigl(\pi^*p_r(\Omega)\bigr)=\pi^*d p_r(\Omega)=0.
	\]
	It follows that Term~{\rm(II)}
	converges to zero in probability.
	
	\medskip
	Combining the convergence of Term (I) and Term (II), the sum goes to $0$ in probability, completing the proof.
\end{proof}
\begin{remark}[\bf Computational aspects of minimum-volume cycles]\label{rem:cycle_computation}
	While Theorem \ref{thm:persistent_cycles} requires the sequence of empirical persistent cycles $Z_{t, S_m}^{4r}$ to have uniformly bounded mass, finding such a minimum-volume representative within a homology class is fundamentally an optimization problem. For 1-cycles, the simplicial boundary matrix is totally unimodular, allowing the minimum-mass integral cycle to be found efficiently in polynomial time via linear programming \cite{DHK2011}. 
	
	However, for the $4r$-dimensional cycles required for Pontryagin numbers, the boundary matrix generally loses this total unimodularity, making the strict search for an optimal \emph{integral} cycle an NP-hard integer linear programming (ILP) problem. From a geometric perspective, this computational hurdle can be elegantly bypassed in two practical ways:
	\begin{enumerate}
		\item {\it Real Chains:} The integration of differential forms is well-defined over chains with real coefficients. By dropping the integer constraint, one can solve the continuous linear program (LP) relaxation in polynomial time.  The resulting minimum-mass real cycle is a natural candidate for
		satisfying the bounded-mass and flat-convergence hypotheses of the
		theorem. 
		\item {\it Top-Dimensional Fundamental Classes:} If the dimension of the manifold is exactly $n = 4r$, the topological invariant is the classical Pontryagin number evaluated over the entire manifold. In this case, the integration domain $Z_{t,S_m}^{n}$ is simply the fundamental class of the reconstructed complex, represented by the oriented sum of all top-dimensional simplices. This requires no optimization algorithm; the bounded-mass hypothesis is satisfied whenever the total $n$-dimensional volumes of the reconstructed complexes are uniformly bounded.
	\end{enumerate}
\end{remark}

\section{Conclusion and Final Remarks}\label{sec:fin}

In this paper, we have developed a rigorous framework for empirical
Hodge theory on closed submanifolds from uniformly sampled point-cloud
data. We constructed consistent empirical estimators of the tangent
projection, the second fundamental form, the Riemannian curvature
tensor, the Weitzenb\"ock curvature endomorphisms, and the Hodge
Laplacians. We also established consistent recovery of the Pontryagin
forms and the associated periods considered in this paper.

By combining geometric cutoff constructions with empirical-process
estimates for parametrized Lipschitz classes, we proved uniform
consistency of the geometric estimators and compact Mosco convergence
of the empirical Hodge quadratic forms. Consequently, the empirical
spectral cluster near zero recovers the Betti numbers and converges to
the corresponding harmonic spaces in the transported discrete
$L^2$-sense.

Although the Nystr\"om extension is not needed for the spectral
convergence theorem, every empirical eigenform in a bounded spectral
cluster admits an exact extension to a continuous
$\Lambda^k\mathbb R^d$-valued section on $M^n$. Establishing
uniform $C^0$-convergence of these extensions requires additional
low-energy regularity estimates and is left for future work.

A natural next direction is to recover the real homotopy type of a
closed submanifold $M^n\subset\mathbb R^d$, initially under the
simplifying assumption
\[
H^1(M^n;\mathbb R)=0.
\]
This condition is often referred to as cohomological simple
connectivity; in the real Sullivan homotopy category, it yields a
simply connected, and hence nilpotent, real homotopy type. The
algebraic structures developed in
\cites{Merkulov1999,FKLS2021,FL2025,Le2026}
provide a natural framework for this problem.

Chen's iterated-integral theory \cite{Chen1977} suggests a related connection with path
and loop spaces. Iterated integrals also underlie path signatures in
rough-path theory and provide a hierarchy of nonlinear features for
sequential data \cite{Lyons1998}. A possible direction is therefore to investigate
empirical iterated integrals and transferred higher operations on
low-energy empirical differential forms. Developing this connection
would first require suitable product, regularity, and stability results
beyond those proved in the present paper.

\begin{remark}  The scalar graph Laplacian of Belkin--Niyogi \cites{BN2003,BN2008} and the connection Laplacian of Singer--Wu \cites{SW2012,SW2017} are both rooted in heat-kernel approximations \cites{BGV1996,Rosenberg1997} of the intrinsic Laplacian. In contrast, our extrinsic Gaussian deformation bypasses the heat kernel and directly exploits the Euclidean geometry of the ambient space.
\end{remark}

\begin{remark}[\bf Relation with Singer--Wu connection Laplacians]
	The spectral convergence theorem of Singer and Wu for the connection
	Laplacian provides another possible route to empirical Hodge theory.
	Their construction applies not only to tangent vector fields but, more
	generally, to sections of vector bundles equipped with a connection.
	Applied to the exterior bundle $\Lambda^kT^*M$, it yields an empirical
	approximation of the rough Laplacian $\nabla^*\nabla$ on $k$-forms.
	
	In our sign convention, the Bochner--Weitzenb\"ock formula reads
	\[
	\Delta^k=\nabla^*\nabla-\mathcal R_k,
	\]
	where $\mathcal R_k$ is the algebraic curvature term. Therefore, once
	the Riemannian curvature tensor has been recovered from the point cloud,
	as in Theorem~\ref{thm:Rcurv}, one can also construct
	\[
	\widehat\Delta^k_{SW,t,S_m}
	\coloneqq
	\widehat{\nabla^*\nabla}_{SW,t,S_m}
	-
	\widehat{\mathcal R}_{k,t,S_m}.
	\]
	Thus, one might expect that spectral convergence of the Singer--Wu
	empirical connection Laplacian $\widehat{\nabla^*\nabla}_{SW,t,S_m}$,
	combined with the curvature convergence established in
	Theorem~\ref{thm:Rcurv}, would yield spectral convergence of the
	empirical Hodge Laplacian $\widehat\Delta^k_{SW,t,S_m}$. We regard this
	as a conjecture rather than an immediate consequence: the Singer--Wu
	construction and ours arise from genuinely different approximation
	schemes, and combining their spectral-convergence statement with our
	operator-norm estimate for $\widehat{\mathcal R}_{k,t,S_m}$ would
	require reconciling the sampling regimes and notions of convergence
	native to each framework.

	This suggests an alternative route to recovering harmonic forms.
	 The point at which our approach complements Singer--Wu is the
	direct recovery of the curvature tensor: estimating curvature from
	random samples is substantially more delicate than estimating tangent
	spaces or parallel transport, and our second-fundamental-form estimator
	provides the missing curvature input for the Weitzenb\"ock correction.
	
	Furthermore, while estimating empirical parallel transport across a point cloud graph involves aligning local tangent spaces via orthogonal Procrustes problems, our ambient projection framework directly yields the diffusion operator via matrix multiplication in $\mathbb{R}^d$, offering a distinct computational alternative.
\end{remark}

\begin{remark}[\bf Comparison with Cao et al. \cite{CLS2021}]
	In Riemannian geometry, the Weingarten map (shape operator) $A_\xi$ and the second fundamental form $B$ are related by metric duality (see \eqref{eq:shapeoperator}). Thus, estimating one effectively provides the other. However, Cao et al. use a two-step regression approach to estimate the Weingarten map. First, they estimate the tangent and normal spaces via Local PCA. Then, they fit the Weingarten map components by least-squares/quadratic regression of the normal displacements against the tangent coordinates. 
	
	Our method is an integral-based estimator. It extracts the curvature tensor directly through the first moment of the Gaussian kernel applied to the displacement vectors $(y-x) \otimes \omega(y)$. Integral estimators are inherently more stable under 
	high-frequency noise and do not require explicit local surface reconstruction or regression steps.  
\end{remark}

\begin{remark}[\bf Further Possible Generalizations] 
	\begin{enumerate}
		\item In a forthcoming paper, we shall extend the results of the present work to the setting of an arbitrary ambient Riemannian manifold of bounded geometry.
		\item It is possible to extend the method of this paper to learn the Dirac operator on a closed submanifold $M^n \subset \R^d$ that admits a spin structure. A necessary first step in this direction is recognizing whether $M^n$ admits a spin structure purely via point cloud data (e.g., via the vanishing of the empirical second Stiefel-Whitney class).
	
	\item By Novikov's theorem, rational Pontryagin classes are topological invariants, and the de Rham cohomology ring is a homotopy invariant. Therefore, we conjecture  that our framework can be extended to robustly recover both the Pontryagin classes and the cohomology ring even from point clouds corrupted by ambient noise.  The present framework still relies on strong smoothness and sampling assumptions, and extending these results to singular or highly noisy geometric settings remains an important open problem.
	\end{enumerate}
\end{remark}

\begin{remark}[\bf Computational Complexity and Low-Rank Implementations]\
	\label{rem:computational_complexity}
	While the extension of the empirical Hodge Laplacian $\hat{\Delta}_{t, S_m}^k$ to the ambient space $\Lambda^k \mathbb{R}^d$ yields significant theoretical and algebraic simplifications, a naive numerical implementation faces the ``curse of dimensionality.'' The dimension of the ambient exterior space is $\binom{d}{k}$. For a point cloud of $m$ samples, the explicit global Laplacian matrix would be of size $m\binom{d}{k} \times m\binom{d}{k}$, which becomes computationally intractable to store or diagonalize for large $d$.
	
	However, the empirical operator is highly degenerate by construction. The projection $\hat{\Pi}_x$ strictly constrains the active geometry to the $n$-dimensional empirical tangent space, meaning the local rank of the operator acting on $k$-forms is strictly bounded by $\binom{n}{k}$. 
	
	To efficiently compute the spectrum (e.g., the harmonic forms) in practice, one must avoid explicitly constructing the global matrix by employing matrix-free iterative eigensolvers (such as the Lanczos \cite{lanczos1950}, \cite{Golub2013} algorithm or LOBPCG  \cite{Knyazev2001}). These algorithms only require the evaluation of the Matrix-Vector Product (MVP) representing the action of $\hat{\Delta}_{t, S_m}^k$ on a discrete $k$-form $v \in L^2(S_m, \Lambda^k \mathbb{R}^d)$. 
	
	The MVP can be evaluated with high efficiency using factored local projections. Let $V_x \in \mathbb{R}^{d \times n}$ be the matrix whose columns are the orthonormal basis vectors of the empirical tangent space $\hat{T}_x M^n$. The empirical projection is exactly factored as $\hat{\Pi}_x = V_x V_x^T$. By the functoriality of exterior powers, the projection on $k$-forms factors as:
	\begin{equation}
		\Lambda^k \hat{\Pi}_x = (\Lambda^k V_x)(\Lambda^k V_x)^T,
	\end{equation}
	where $\Lambda^k V_x$ is a $\binom{d}{k} \times \binom{n}{k}$ matrix. 
	
	When applying the full operator $\hat{\Delta}_{t, S_m}^k$ to $v$, one never projects the full ambient vector directly. Instead, the operation $(\Lambda^k V_x)^T v(x)$ pulls the ambient $k$-form down into the $\binom{n}{k}$-dimensional intrinsic empirical tangent space. Crucially, both the neighborhood summations for the diffusion operator $\hat{\mathcal{L}}_{t, S_m}$ and the local evaluations of the zeroth-order potentials ($\widehat{\mathcal{W}}_{t, S_m}$, $\widehat{\Rr}_{k,t,S_m}$, and $\widehat{\Rr}^{(1)}_{k,t,S_m}$) occur entirely within this vastly reduced intrinsic space. The final result is only pushed back to the ambient space via $\Lambda^k V_x$ at the very end of the operation.
	
	The fixed spatial cutoff $\chi_\delta$ restricts the interaction to a
	geometric radius graph, but since $\delta$ is independent of $m$,
	this graph need not be sparse asymptotically. Thus the exact empirical
	operator may still contain $O(m^2)$ nonzero interaction blocks.
	
	A possible sparse implementation is to retain only a growing number
	$K_m$ of nearest neighbors at each sample point and to symmetrize the
	resulting neighbor relation. A natural regime suggested by the
	Gaussian decay is
	\[
	K_m
	\asymp
	m\,t^{n/2}(\log m)^{n/2}.
	\]
	Under the scaling $t=m^{-1/(2n)}$, this gives
	\[
	K_m
	\asymp
	m^{3/4}(\log m)^{n/2},
	\]
	which is sublinear in $m$. This scaling corresponds to the expected
	number of sample points in a ball of effective radius
	\[
	\sqrt{t\log m},
	\]
	outside which the Gaussian weights are polynomially small in $m$.
	Proving that this sparsified operator has the same spectral limit as
	the exact empirical Hodge operator requires additional uniform
	Gaussian-tail and projector-perturbation estimates and is left for
	future work.
\end{remark}

\section*{Declaration of the use of AI}
  The author    acknowledges      OpenAI's  ChatGPT  (5.5 High + 5.6  Sol High),  Google's  Gemini (3.1 Pro),   Anthropic's Claude (Sonnet 5 High), and DeepSeek   for   their help  in  the  improvement  of   readability  of the exposition.  The author  is     responsible  for the correctness of the paper.

\appendix

\section{Proof of Proposition \ref{prop:empitangent_clean}}\label{sec:empitangent}

 Assume the condition of Proposition \ref{prop:empitangent_clean}. The proof proceeds in three main steps: establishing a local coordinate representation, computing the eigengap of the population covariance operator $\Sigma_{t,\delta}(p)$, and applying concentration inequalities to bound the empirical deviations.

\medskip

\noindent \underline{Step 1}. {\it Local coordinate system and Taylor expansion.}
Fix a point $p \in M^n$. By translating and rotating our coordinate system, we assume $p = 0$ and that the tangent space $T_p M^n$ is aligned with the first $n$ coordinate axes, i.e., $T_p M^n = \mathbb{R}^n \times \{0\}^{d-n}$. 
For any point $y \in D_\delta(p)$, we decompose it into its tangent and normal components: 
$$y = v + u,$$
where $v \in T_p M^n$ and $u \in (T_p^\perp M^n)$. 
Since $M^n$ is a $C^3$-smooth submanifold with reach $\tau_M$, the normal component is governed by the second fundamental form $B_p$. Specifically, for $\delta < \tau_{M^n}/4$, any $y \in D_\delta(p) \subset M^n \subset \mathbb{R}^d$ can be uniquely parameterized by its tangent projection $v$.   By \cite{AL2019}*{Lemma 1}, we have:
\begin{equation}
	u = \frac{1}{2} B_p(v, v) + R(v), \quad \text{where } \|u\| \le \frac{1}{2\tau_{M^n}}\|v\|^2 \text{ and } \|R(v)\| \le C_1 \|v\|^3. \label{eq:al}
\end{equation}
Thus,  $Du(0)=0$ and $Du(v)=O(\|v\|)$. Hence the metric induced by
the graph parametrization $v\mapsto v+u(v)$ satisfies
\[
g(v)=I+(Du(v))^\top Du(v),
\]
and therefore
\[
\sqrt{\det g(v)}=1+O(\|v\|^2).
\]
Consequently,
\[
d\mu(y)
=
\frac1{\operatorname{vol}_g(M^n)}
\bigl(1+O(\|v\|^2)\bigr)\,dv
\]
where  $dv$ is the Lebesgue measure on $T_pM^n$.

\medskip

\noindent \underline{Step 2}. {\it Analysis of the population covariance $\Sigma_{t, \delta}(p)$.}
Define the expected localized covariance matrix:

\begin{equation} \Sigma_{t, \delta}(p) \coloneqq \E_{S_m \sim \mu^m}[\Sigma_{t,S_m}(p)] = \int_{M^n} \Phi_t(p, y) (y-p)(y-p)^\top \chi_\delta (p, y)\,d\mu(y). \label{eq:expectedlcm}
\end{equation}
We decompose this $d \times d$ matrix into blocks corresponding to the tangent space $T_p M^n$ and the normal space $(T_p^\perp M^n)$:
\[ \Sigma_{t, \delta}(p)= \begin{pmatrix} \Sigma_{T} & \Sigma_{TN} \\ \Sigma_{NT} & \Sigma_{N} \end{pmatrix}. \]
For the Gaussian kernel
$\Phi_t(p,y)=(4\pi t)^{-n/2}\exp(-\|y-p\|^2/4t)$, we use
\[
\|y-p\|^2=\|v\|^2+\|u\|^2.
\]
Since $\|u\|^2\le C\|v\|^4$, the leading Gaussian behavior is
controlled by the tangent displacement $v$; see Remark~\ref{rem:extint}.
Evaluating the blocks by integrating against the Gaussian measure on $\mathbb{R}^n$:
\begin{enumerate}
	\item {\it Tangent-Tangent Block ($\Sigma_T$):} The leading term of $(y-p)(y-p)^\top$ is $v v^\top$. By the spherical symmetry of the Gaussian, $\int_{\R^n} v v^\top e^{-\|v\|^2/4t} dv = c_0 t I_n$ for some constant $c_0 > 0$. (The constant $c_0>0$ absorbs the normalizing factor
	$\vol_g(M^n)^{-1}$ coming from the probability measure
	$d\mu=d\vol_g/\vol_g(M^n)$.) Factoring in the volume distortion $d\mu(y)$ and the $O(\|v\|^4/t)$ normal component in the exponent, we obtain $\Sigma_T = c_0 t I_n + O(t^2)$. 
	\item \label{pertnorm}{\it Tangent-Normal Block ($\Sigma_{TN}$):} The integrand is $v u^\top \Phi_t(p,y)$. The leading term of $v u^\top$ is $\frac{1}{2} v B_p(v,v)^\top$, which is cubic and odd in $v$. Therefore, its leading contribution vanishes by the spherical symmetry of the Gaussian measure, leaving only higher-order even terms. Integration yields a uniformly bounded norm $\|\Sigma_{TN}\|_{\mathrm{op}} \le C_3 t^2$.
	\item {\it Normal-Normal Block ($\Sigma_N$):} The integrand is $u u^\top \Phi_t(p,y)$. Bounded by $C_4 \|v\|^4 e^{-\|v\|^2/4t}$, this yields $\|\Sigma_N\|_{\mathrm{op}} \le C_5 t^2$.
\end{enumerate}
By the Davis-Kahan $\sin\Theta$ theorem \cites{DK1970,YWS2015}, the angle between the true tangent space $\Pi_p$ and the span of the top $n$ eigenvectors of $\Sigma_{t,\delta}(p)$ (denoted $\Pi_{t,\delta}$) is bounded by the ratio of the cross-term norm to the eigengap. 

To find the eigenvalues of the full matrix $\Sigma_{t,\delta}$, we treat it as a block-diagonal matrix perturbed by the cross-terms $\Sigma_{TN}$:
\[ \Sigma_{t, \delta}(p) = \underbrace{\begin{pmatrix} \Sigma_{T} & 0 \\ 0 & \Sigma_{N} \end{pmatrix}}_{\text{Base Matrix}} + \underbrace{\begin{pmatrix} 0 & \Sigma_{TN} \\ \Sigma_{NT} & 0 \end{pmatrix}}_{\text{Perturbation } E}. \]
We know from \eqref{pertnorm} that the norm of the perturbation is $\|E\|_{\text{op}} = \|\Sigma_{TN}\|_{\text{op}} = O(t^2)$. By Weyl's Inequality \cite{HJ2012}*{Theorem 4.3.1}, the eigenvalues of the full matrix cannot differ from the eigenvalues of the base matrix by more than $\|E\|_{\text{op}}$. Looking at the two groups of eigenvalues:
\begin{itemize}
	\item The top $n$ eigenvalues ($\lambda_1 \ge \dots \ge \lambda_n$) come from $\Sigma_T$. Since $\Sigma_T = c_0 t I_n + O(t^2)$, the base eigenvalues are tightly clustered around $c_0 t$. Adding the perturbation $E$, we get:
	\[ \lambda_n \ge c_0 t - O(t^2). \]
	\item The remaining $d-n$ eigenvalues ($\lambda_{n+1} \ge \dots \ge \lambda_d$) come from $\Sigma_N$. Since $\|\Sigma_N\| = O(t^2)$, the base eigenvalues are at most $O(t^2)$. Adding the perturbation $E$ yields:
	\[ \lambda_{n+1} \le O(t^2) + O(t^2) = O(t^2). \]
\end{itemize}
The eigengap $\mathfrak{g}$ is defined as:
\[ \mathfrak{g} = \lambda_n - \lambda_{n+1} \ge \Big(c_0 t - O(t^2)\Big) - O(t^2) = c_0 t - O(t^2). \]
For a sufficiently small bandwidth $t$, the $c_0 t$ term dominates. Thus, 
\begin{equation}
	 \mathfrak{g} \ge \frac{c_0 t}{2}.\label{eq:gappolbound}
	\end{equation}
By the Davis-Kahan theorem, the population bias satisfies:
\[ \|\Pi_{t,\delta} - \Pi_p\|_{\mathrm{op}} \le \frac{\sqrt{2}\|\Sigma_{TN}\|_{\mathrm{op}}}{\mathfrak{g}} \le \frac{\sqrt{2}(C_3 t^2)}{c_0 t / 2} = O(t). \]

\medskip

\noindent \underline{Step 3}. {\it Empirical concentration via matrix Bernstein.}
We now bound the stochastic fluctuation of the empirical covariance matrix $\Sigma_{t,S_m}(p)$. The matrix $\Sigma_{t,S_m}(p)$ is the average of $m$ independent random matrices 
\begin{equation} 
	Z_j(p) \coloneqq \Phi_t(p, x_j) (x_j-p)(x_j-p)^\top \chi_\delta(p, x_j). \label{eq:Yj}
\end{equation}
To apply the matrix Bernstein inequality, we bound the operator norm and the variance of the centered random matrices 
\[ Y_j(p) \coloneqq Z_j(p) - \Sigma_{t,\delta}(p). \]
First, we bound the uncentered term $Z_j$. Through standard calculus, the maximum of $r \mapsto r^2 e^{-r^2/4t}$ is attained exactly at $\| x_j-p\|_{\mathbb{R}^d} = 2\sqrt{t}$. Assuming $2\sqrt{t} \le \delta$, this yields:
\begin{equation}
	\|Z_j(p)\|_{\mathrm{op}} \le \sup_{y \in D_\delta(p)} \Phi_t(p, y) \|y-p\|^2 = 4 e^{-1} (4\pi)^{-n/2} t^{1-n/2} \coloneqq \frac{L}{2}. \label{eq:zbound}
\end{equation}
Because matrix Bernstein requires zero-mean matrices, we bound the centered variables using the triangle inequality and Jensen's inequality:
\[ \|Y_j(p)\|_{\mathrm{op}} \le \|Z_j(p)\|_{\mathrm{op}} + \|\E_{x_j \sim \mu}[Z_j(p)]\|_{\mathrm{op}} \le 2 \sup_{x_j \in D_{\delta}(p)} \|Z_j(p)\|_{\mathrm{op}} \stackrel{\eqref{eq:zbound}}{\le} L. \]

\begin{lemma}\label{lem:MBernstein}  Assume that  $t = m^{-\frac{1}{2n}}$ and $n \ge 2$.
	With probability at least $1 - m^{-2/n}$ we have
	\begin{equation}
		\Delta \coloneqq \sup_{p \in M^n} \|\Sigma_{t,S_m}(p) - \Sigma_{t,\delta}(p)\|_{\mathrm{op}} \le C_9t^2 \label{eq:MBernstein}
	\end{equation}
	for some constant $C_9$ depending only on the geometry of $M^n$.
\end{lemma}
\begin{proof}
	The proof consists of three steps: a matrix Bernstein estimate at a
	fixed point, discretization by an $\varepsilon$-net, and Lipschitz
	interpolation.
	
	\smallskip
	\noindent
	\underline{Step 1.} \textit{Matrix Bernstein estimate at a fixed point.}
	
	Fix $p\in M^n$, and write
	\[
	\Sigma_{t,S_m}(p)
	=	\frac{1}{m}\sum_{j =1} ^m Z_j  (p).
	\]
	By \eqref{eq:Yj},  \eqref{eq:expectedlcm}
	\[
	\E_{x_j \sim \mu} Y_j(p)=0.
	\]
	By \eqref{eq:zbound},
	\begin{equation}
		\|Y_j(p)\|_{\mathrm{op}}
		\le
		2\sup_{x_j}\|Z_j(p)\|_{\mathrm{op}}
		\le
		C_6t^{1-n/2}
		\coloneqq L.
		\label{eq:L}
	\end{equation}
	Moreover,
	\begin{equation}
		\left\|
		\E_{x_j\sim\mu}[Z_j(p)^2]
		\right\|_{\mathrm{op}}
		\le
		\int_{D_\delta(p)}
		\Phi_t(p,y)^2\|y-p\|^4\,d\mu(y)
		\le
		C_7t^{2-n/2}.
		\label{eq:sigmamb}
	\end{equation}
	Since
	\[
	\E[Y_j(p)^2]
	= \E[Z_j(p)^2]	-	\bigl(\E Z_j(p)\bigr)^2
	\]
	and both terms are symmetric positive semidefinite, it follows that
	\begin{equation}
		\left\|
		\E[Y_j(p)^2]
		\right\|_{\mathrm{op}}
		\le
		C_7t^{2-n/2}
		\coloneqq\nu.
		\label{eq:nu}
	\end{equation}
	
	The matrix Bernstein inequality therefore gives, for every $u>0$,
	\begin{equation}
		\mu^m
		\left\{
			S_m:
			|\Sigma_{t,S_m}(p)-\Sigma_{t,\delta}(p)|_{\mathrm{op}}>u
			\right\}
		\le
		2d\exp
		\left(
		-\frac{mu^2/2}{\nu+Lu/3}
		\right).
		\label{eq:Bersteind}
	\end{equation}
	
	\smallskip
	\noindent
	\underline{Step 2.} \textit{Discretization by an
		$\varepsilon$-net.}
	
	Fix
	\[
	A>\frac{n+3}{2},
	\qquad
	\varepsilon=t^A,
	\]
	and let $\mathcal N_\varepsilon\subset M^n$ be an
	$\varepsilon$-net. Since $M^n$ is compact and $n$-dimensional,
	\begin{equation}
		|\mathcal N_\varepsilon|
		\le
		C_M\varepsilon^{-n}
		=		
		C_Mt^{-An}.
		\label{eq:epsilonnet}
	\end{equation}
	Set
	\[
	\eta=m^{-2/n}
	\]
	and
	\[
	\gamma
	\coloneqq
	\log\left(
	\frac{2d|\mathcal N_\varepsilon|}{\eta}
	\right).
	\]
	Because $t=m^{-1/(2n)}$, we have
	\[
	|\log t|
	=	
	\frac{1}{2n}\log m.
	\]
	Consequently,
	\begin{align}
		\gamma
		&\le
		\log(2dC_M)
		+
		An|\log t|
		+
		\frac{2}{n}\log m
		\nonumber\\
		&=
		\log(2dC_M)
		+
		\left(
		\frac A2+\frac2n
		\right)\log m
		\le
		C_{\mathrm{net}}\log m
		\label{eq:cnet}
	\end{align}
	for all sufficiently large $m$.
	
	Choose
	\[
	u
	=	
	\frac{2L\gamma}{3m}
	+
	\sqrt{\frac{2\nu\gamma}{m}}.
	\]
	By \eqref{eq:Bersteind} and a union bound over
	$\mathcal N_\varepsilon$, with probability at least $1-\eta$,
	\begin{equation}
		\max_{p_k\in\mathcal N_\varepsilon}
		\|\Sigma_{t,S_m}(p_k)-\Sigma_{t,\delta}(p_k)\|_{\mathrm{op}}
		\le u.
		\label{eq:net_deviation}
	\end{equation}
	Using
	\[
	L=C_6t^{1-n/2},
	\qquad
	\nu=C_7t^{2-n/2},
	\qquad
	\gamma\le C_{\mathrm{net}}\log m,
	\]
	we obtain
	\[
	\begin{aligned}
		u
		&\le
		\frac{2C_6C_{\mathrm{net}}}{3}
		\frac{t^{1-n/2}\log m}{m}
		+
		\sqrt{2C_7C_{\mathrm{net}}}
		\sqrt{\frac{t^{2-n/2}\log m}{m}}\\
		&\le
		C\left(
		\frac{t^{1-n/2}\log m}{m}
		+
		\sqrt{\frac{t^{2-n/2}\log m}{m}}
		\right),
	\end{aligned}
	\]
	where $C>0$ depends only on $C_6$, $C_7$, and
	$C_{\mathrm{net}}$. Since $m=t^{-2n}$, this becomes
	\begin{equation}
		u
		\le
		C\left(
		t^{1+3n/2}\log m
		+
		t^{1+3n/4}\sqrt{\log m}
		\right).
		\label{eq:usolved}
	\end{equation}
	Since $n\ge2$,
	\[
	1+\frac{3n}{2}>2,
	\qquad
	1+\frac{3n}{4}>2.
	\]
	Furthermore, $\log m=2n|\log t|$, and every positive power of
	$t$ dominates every power of $|\log t|$ as $t\to0$. Therefore
	\[
	t^{1+3n/2}\log m=o(t^2),
	\qquad
	t^{1+3n/4}\sqrt{\log m}=o(t^2).
	\]
	Hence, for all sufficiently large $m$,
	\begin{equation}
		u\le Ct^2.
		\label{eq:u_t2}
	\end{equation}
	
	\smallskip
	\noindent
	\underline{Step 3.} \textit{Lipschitz interpolation.}
	
	For $p\in M^n$, let $p_k\in\mathcal N_\varepsilon$ satisfy
	$\|p-p_k\|_{\mathbb R^d}\le\varepsilon$. Then
	\begin{align}
		\Delta
		&\le
		\max_{p_k\in\mathcal N_\varepsilon}
		\|\Sigma_{t,S_m}(p_k)-\Sigma_{t,\delta}(p_k)\|_{\mathrm{op}}
		\nonumber\\
		&\qquad+
		\widehat{\mathrm{Lip}}(t)\varepsilon,
		\label{eq:delta_net}
	\end{align}
	where $\widehat{\mathrm{Lip}}(t)$ is the Lipschitz constant of the map
	\[
	p\longmapsto
	\Sigma_{t,S_m}(p)-\Sigma_{t,\delta}(p).
	\]
	
	Let $r=x_j-p$. Differentiating
	\[
	Z_j(p)
	=	
	\Phi_t(p,x_j)
	(x_j-p)(x_j-p)^\top
	\chi_\delta(p,x_j)
	\]
	with respect to $p$, we obtain
	\begin{align}
		\|D_pZ_j(p)\|_{\mathrm{op}}
		&\le
		\frac{1}{(4\pi t)^{n/2}}
		\left(
		\frac{|r|^3}{2t}
		+
		2|r|
		\right)
		e^{-|r|^2/(4t)}
		\nonumber\\
		&\quad+
		C_\delta
		\frac{1}{(4\pi t)^{n/2}}
		|r|^2e^{-|r|^2/(4t)}.
		\label{eq:supgrad}
	\end{align}
	For every $q\ge0$,
	\[
	\sup_{s\ge0}s^q e^{-s^2/(4t)}
	\le
	C_qt^{q/2}.
	\]
	It follows from \eqref{eq:supgrad} that
	\[
	\sup_{p,x_j}
	\|D_pZ_j(p)\|_{\mathrm{op}}
	\le
	Ct^{(1-n)/2}.
	\]
	Averaging over $j$ gives the same bound for
	$p\mapsto\Sigma_{t,S_m}(p)$, and taking expectations gives the same
	bound for $p\mapsto\Sigma_{t,\delta}(p)$. Hence
	\[
	\widehat{\mathrm{Lip}}(t)
	\le
	Ct^{(1-n)/2}.
	\]
	Since $\varepsilon=t^A$ and $A>(n+3)/2$,
	\[
	\widehat{\mathrm{Lip}}(t)\varepsilon
	\le
	Ct^{A+(1-n)/2}
	=	
	o(t^2).
	\]
	In particular, for all sufficiently small $t$,
	\begin{equation}
		\widehat{\mathrm{Lip}}(t)\varepsilon
		\le
		Ct^2.
		\label{eq:interpolation_t2}
	\end{equation}
	
	Combining \eqref{eq:net_deviation}, \eqref{eq:u_t2},
	\eqref{eq:delta_net}, and \eqref{eq:interpolation_t2}, we conclude
	that, with probability at least
	\[
	1-\eta=1-m^{-2/n},
	\]
	\[
	\Delta
	\le
	u+\widehat{\mathrm{Lip}}(t)\varepsilon
	\le
	C_9t^2.
	\]
	This proves \eqref{eq:MBernstein}.
\end{proof}

\noindent \underline{Step 4}. {\it Final Davis--Kahan argument.}

Recall that $\Pi_{t,\delta}(p)$ denotes the orthogonal projection onto the span
of the top $n$ eigenvectors of the population covariance
$\Sigma_{t,\delta}(p)$. From \eqref{eq:gappolbound}, the population eigengap satisfies
\begin{equation}
\mathfrak g(p)\ge \frac{c_0t}{2}\label{eq:eigengapp}
\end{equation}
uniformly in $p$, for all sufficiently small $t$. Moreover, the
population off-diagonal tangent-normal block satisfies
\[
\|\Sigma_{TN}(p)\|_{\mathrm{op}}\le Ct^2.
\]
Therefore, by the Davis--Kahan theorem,
\begin{equation}
\sup_{p\in M^n}\|\Pi_{t,\delta}(p)-\Pi_p\|_{\mathrm{op}}
\le Ct.\label{eq:dkpoptrue}
\end{equation}

By \eqref{eq:MBernstein},
\[
\Delta=O(t^2)=o(t).
\]
Since the population eigengap is bounded below by $c_0t/2$, Weyl's
inequality shows that, for all sufficiently small $t$, the empirical
top $n$-dimensional spectral cluster remains isolated, with eigengap
bounded below by $c_0t/4$. A second application of the
Davis--Kahan theorem therefore gives
\begin{equation}
\sup_{p\in M^n}
\|(\widehat\Pi_{t,S_m})_p-\Pi_{t,\delta}(p)\|_{\mathrm{op}}
\le
C\frac{\Delta}{t}
\stackrel{\eqref{eq:MBernstein}}{\le} Ct. \label{eq:dkempipop}
\end{equation}

Combining  \eqref{eq:dkpoptrue}  and \eqref{eq:dkempipop} yields
\[
\sup_{p\in M^n}
\|(\hat\Pi_{t,S_m})_p-\Pi_p\|_{\mathrm{op}}
\le Ct.
\]
This proves \eqref{eq:empoj_clean}.

\noindent \underline{Step 5}. {\it Continuity of the empirical projection map.}

On the same high-probability event, the empirical eigengap between the
$n$-th and $(n+1)$-st eigenvalues is uniformly bounded below by 
\begin{equation}
	\frac{c_0t}{2}-2 C_9 t^2 \ge \frac{c_0t}{4}.\label{eq:empieigengap}
\end{equation}
Because the two spectral clusters lie in uniformly separated intervals for all $p$ on the high-probability event, we choose a contour $\Gamma_t$ depending on $t$ but independent of $p$ and enclosing only the top $n$ eigenvalues. Then the matrix-valued map
\[
p\longmapsto \Sigma_{t,S_m}(p)
\]
is continuous because $\Phi_t$, $(x_j-p)(x_j-p)^\top$, and $\chi_\delta(p,x_j)$ are continuous in $p$. By classical finite-dimensional perturbation theory \cite{Kato1995}*{Section II.1.4}, the associated spectral projector is given by the Riesz contour integral
\begin{equation}
	(\hat\Pi_{t,S_m})_p
	=
	\frac{1}{2\pi i}
	\oint_{\Gamma_t}
	(zI-\Sigma_{t,S_m}(p))^{-1}\,dz, \label{eq:riesz1}
\end{equation}
which consequently depends continuously on $p$. Thus $p\mapsto(\hat\Pi_{t,S_m})_p$ is continuous on $M^n$.

\medskip
\noindent \underline{Step 6}. {\it Derivative concentration and Lipschitz control.}

For a matrix-valued map $A(p)$ on $M^n$, we use the norm
\[
\|D_pA(p)\|_{\mathrm{op}}
\coloneqq
\sup_{\xi\in T_pM^n,\ \|\xi\|=1}
\|D_pA(p)[\xi]\|_{\mathrm{op}}.
\]

\begin{lemma}[\bf Derivative concentration for empirical covariance]
	\label{lem:derivative_covariance_concentration}\
	Assume the hypotheses of Proposition~\ref{prop:empitangent_clean}, and
	assume $n\ge3$. Let
	\[
	\Delta_D
	\coloneqq
	\sup_{p\in M^n}
	\|D_p\Sigma_{t,S_m}(p)-D_p\Sigma_{t,\delta}(p)\|_{\mathrm{op}}.
	\]
	If $t=m^{-1/(2n)}$, then, for all sufficiently large $m$, with
	probability at least $1-m^{-2/n}$,
	\[
	\Delta_D\le C't^2.
	\]
\end{lemma}

\begin{proof}
	We apply the same Matrix Bernstein and net-interpolation argument as in
	Lemma~\ref{lem:MBernstein}, now to the differentiated random matrices
	\[
	D_pZ_j(p)[\xi],
	\qquad
	\xi\in T_pM^n,\quad \|\xi\|=1,
	\]
	where
	\[
	Z_j(p)=
	\Phi_t(p,x_j)(x_j-p)(x_j-p)^\top\chi_\delta(p,x_j).
	\]
	The supremum over unit tangent directions is handled by adding a fixed
	finite net in the unit sphere of $T_pM^n$ inside each coordinate chart;
	this changes only the constants in the logarithmic factor.
	
	Differentiating $Z_j(p)$ gives terms from the Gaussian factor, the
	quadratic factor $(x_j-p)(x_j-p)^\top$, and the cut-off. The dominant
	term comes from differentiating the Gaussian:
	\[
	D_p\Phi_t(p,x_j)[\xi]
	=
	\frac{\langle x_j-p,\xi\rangle}{2t}\Phi_t(p,x_j),
	\]
	which contributes an additional factor of order $t^{-1/2}$ on the
	kernel scale $\|x_j-p\|\sim\sqrt t$.   The derivative of the quadratic factor gives the same envelope order,
	and the derivative of the cut-off is lower order. Consequently,
	\[
	L_D\le C'_1t^{1/2-n/2},
	\qquad
	\nu_D\le C'_2t^{1-n/2}.
	\]
	
	The interpolation from the net to all $p\in M^n$ uses the corresponding
	bound for the second spatial derivatives of $Z_j(p)$. Differentiating
	once more gives an envelope of order
	\[
	\sup_{p,x_j}\|D_p^2Z_j(p)\|_{\mathrm{op}}\le Ct^{-n/2}
	\]
	in a finite atlas with fixed trivialization of $TM^n$.
	Choose
	\[
	\varepsilon=t^A,
	\qquad A>\frac n2+2.
	\]
	Then the interpolation error is
	\[
	Ct^{-n/2}\varepsilon=O(t^2).
	\]

Thus matrix Bernstein and the same net argument give
	\begin{equation}
	\Delta_D
	\le
	C'_3\left(
	\sqrt{\frac{t^{1-n/2}\log m}{m}}
	+
	\frac{t^{1/2-n/2}\log m}{m}
	\right)\label{eq:deltaD}
	\end{equation}
	Under $t=m^{-1/(2n)}$, equivalently $m=t^{-2n}$, the first term is
	\[
	\sqrt{t^{1-n/2}t^{2n}\log m}
	=
	t^{\frac{3n}{4}+\frac12}\sqrt{\log m},
	\]
	and the second term is
	\[
	t^{1/2-n/2}t^{2n}\log m
	=
	t^{\frac{3n+1}{2}}\log m.
	\]
	For $n\ge3$,
	\[
	\frac{3n}{4}+\frac12>2,
	\qquad
	\frac{3n+1}{2}>2.
	\]
	Hence both terms are $O(t^2)$ for sufficiently small $t$, after
	absorbing the logarithmic factors. Therefore $\Delta_D\le C't^2$.
\end{proof}
We now pass from derivative concentration of the covariance matrices to
derivative control of the corresponding spectral projectors. By
\eqref{eq:gappolbound}, \eqref{eq:MBernstein}, and
\eqref{eq:empieigengap}, on the common high-probability event there
exists a contour $\Gamma_t$, independent of $p$, enclosing the top
$n$ eigenvalues of both $\Sigma_{t,S_m}(p)$ and
$\Sigma_{t,\delta}(p)$, such that
\[
\operatorname{length}(\Gamma_t)\le Ct
\]
and
\[
\operatorname{dist}\left(
\Gamma_t,
\sigma(\Sigma_{t,S_m}(p))
\cup
\sigma(\Sigma_{t,\delta}(p))
\right)
\ge ct
\]
uniformly in $p\in M$.
By the Riesz formula,
\begin{equation}
	\hat\Pi_{t,S_m}(p)
	=
	\frac{1}{2\pi i}
	\oint_{\Gamma_t}
	(zI-\Sigma_{t,S_m}(p))^{-1}\,dz.\label{eq:riesz2}
\end{equation}
Differentiating \eqref{eq:riesz2} in $p$ gives
\begin{equation}
	D_p\hat\Pi_{t,S_m}
	=
	\frac{1}{2\pi i}
	\oint_{\Gamma_t}
	(zI-\Sigma_{t,S_m})^{-1}
	(D_p\Sigma_{t,S_m})
	(zI-\Sigma_{t,S_m})^{-1}
	\,dz.\label{eq:riesz1d}
\end{equation}
The same formula holds for $D\Pi_{t,\delta}(p)$. 
To bound the difference between the empirical and population derivatives, we analyze the integrands on the contour $\Gamma_t$.
Let
\[
R(z)\coloneqq(zI-\Sigma_{t,S_m})^{-1},
\qquad
R_0(z)\coloneqq(zI-\Sigma_{t,\delta})^{-1}.
\]
Then
\[
\|R(z)\|_{\mathrm{op}}
+
\|R_0(z)\|_{\mathrm{op}}
\le \frac Ct
\]
uniformly for $z\in\Gamma_t$.

Using the resolvent identity
\[
R(z)-R_0(z)
=R(z)\bigl(\Sigma_{t,S_m}-\Sigma_{t,\delta}\bigr)R_0(z),
\]
the difference of the differentiated Riesz integrands decomposes into
the three terms
\begin{align}
	&R(z) \big( D_p\Sigma_{t,S_m} - D_p\Sigma_{t,\delta} \big) R(z) \label{eq:res_term1}\\
	&+ R(z) (D_p\Sigma_{t,\delta}) R(z) \big( \Sigma_{t,S_m} - \Sigma_{t,\delta} \big) R_0(z) \label{eq:res_term2}\\
	&+ R(z) \big( \Sigma_{t,S_m} - \Sigma_{t,\delta} \big) R_0(z) (D_p\Sigma_{t,\delta}) R_0(z). \label{eq:res_term3}
\end{align}

Integrating the three terms \eqref{eq:res_term1}--\eqref{eq:res_term3} over $\Gamma_t$,  taking operator norms,
 the $1/(2\pi)$ factor and the contour length $\mathcal{O}(t)$ contribute one factor of $t$, compensating for one resolvent factor $t^{-1}$.
 
 Using the triangle inequality, the first term yields the derivative fluctuation bounded by $C t (1/t)^2 \Delta_D = C \Delta_D / t$. The second and third terms are bounded symmetrically, yielding $C t (1/t)^3 \|D_p\Sigma_{t,\delta}\|_{\mathrm{op}} \|\Sigma_{t,S_m} - \Sigma_{t,\delta}\|_{\mathrm{op}}$. Combining these estimates directly produces the bound:
\begin{equation}
	\|D_p\hat\Pi_{t,S_m}-D_p\Pi_{t,\delta}\|_{\mathrm{op}}
	\le
	C\frac{\Delta_D}{t}
	+
	C
	\frac{
		\|D_p\Sigma_{t,\delta}\|_{\mathrm{op}}
		\|\Sigma_{t,S_m}-\Sigma_{t,\delta}\|_{\mathrm{op}}
	}{t^2}.\label{eq:ste61}
\end{equation}

By Lemma~\ref{lem:derivative_covariance_concentration},  with  probability  at least $1 -m ^{-2/n}$
\begin{equation}
\Delta_D\le C't^2.\label{eq:dcc}
\end{equation}
Moreover, \eqref{eq:MBernstein} gives
\begin{equation}
\|\Sigma_{t,S_m}-\Sigma_{t,\delta}\|_{\mathrm{op}}\le Ct^2 \label{eq:mbernstein1}
\end{equation}
with   probability at least $1-  m ^{-2/n}$.

We claim that the population covariance admits the uniform
$C^1$-expansion
\begin{equation}
	\Sigma_{t,\delta}(p)
	=c_0t\,\Pi_p+t^2E_t(p),
	\qquad
	\sup_{0<t<t_0}\|E_t\|_{C^1}<\infty.
	\label{eq:populationC1}
\end{equation}
Indeed, this follows by differentiating under the integral sign and
applying the same local graph expansion and Gaussian moment
calculation as in Step 2. The $C^3$-regularity of $M^n$ provides the
uniform bounds on the first spatial derivatives of the remainder.

Since $p\mapsto\Pi_p$ is of class $C^2$ and $M^n$ is compact,
\eqref{eq:populationC1} gives
\[
D_p\Sigma_{t,\delta}(p)
=
c_0t\,D_p\Pi_p+t^2D_pE_t(p),
\]
and hence
\begin{equation}
	\sup_{p\in M^n}
	\|D_p\Sigma_{t,\delta}(p)\|_{\mathrm{op}}
	\le Ct.
	\label{eq:bounddsigma}
\end{equation}

Substituting \eqref{eq:dcc}, \eqref{eq:mbernstein1}, \eqref{eq:bounddsigma} into \eqref{eq:ste61}, the first term is bounded by $C(t^2)/t = \mathcal{O}(t)$, and the second term is bounded by $C(t)(t^2)/t^2 = \mathcal{O}(t)$. Thus,
\begin{equation*}
\|D_p\hat\Pi_{t,S_m}-D_p\Pi_{t,\delta}\|_{\mathrm{op}} \le Ct.
\end{equation*}
The uniform $C^1$-expansion \eqref{eq:populationC1}, together with
the spectral gap of order $t$, also implies
\begin{equation}
	\sup_{p\in M^n}
	\|D_p\Pi_{t,\delta}(p)-D_p\Pi_p\|_{\mathrm{op}}
	\le Ct.
	\label{eq:population_projector_derivative}
\end{equation}
Indeed, apply the differentiated Riesz projector formula to
$\Sigma_{t,\delta}(p)=c_0t\Pi_p+t^2E_t(p)$ and to the reference family
$c_0t\Pi_p$, $p \in M^n$. The resolvents are $O(t^{-1})$, the contour has length
$O(t)$, and both the perturbation and its first spatial derivative
are $O(t^2)$. The resolvent identity then yields
\eqref{eq:population_projector_derivative}.

Let 
\[
A_p(y)
\coloneqq
\hat\Pi_p\hat\Pi_y-\Pi_p\Pi_y.
\]

Using
\[
D_yA_p(y)
=
\widehat\Pi_p
\bigl(D_y\widehat\Pi_y-D_y\Pi_y\bigr)
+
(\widehat\Pi_p-\Pi_p)D_y\Pi_y,
\]
together with \eqref{eq:population_projector_derivative}, the uniform projection estimate,
and the boundedness of $D\Pi$, we obtain
\[
\sup_{p,y\in M}
\|D_yA_p(y)\|_{\mathrm{op}}\le Ct.
\]

If $y$ lies in a fixed sufficiently small neighborhood of $p$, then integration along a minimizing geodesic, together with the local comparison between intrinsic and extrinsic distances, gives
\[
\|A_p(y)-A_p(p)\|_{\mathrm{op}}
\le Ct\,d_M(p,y)
\le 2Ct\,\|y-p\|_{\mathbb R^d}.
\]
The neighborhood may be chosen uniformly in $p$, since $M$ is compact.

If $y$ lies outside this neighborhood, then
$\|y-p\|_{\mathbb R^d}$ is uniformly bounded below, whereas the uniform projection estimate gives
\[
\|A_p(y)-A_p(p)\|_{\mathrm{op}}\le Ct.
\]
After enlarging $C$, we therefore obtain, for all $p,y\in M$,
\begin{equation}
	\|A_p(y)-A_p(p)\|_{\mathrm{op}}
	\le
	Ct\|y-p\|_{\mathbb R^d}.
	\label{eq:apmean}
\end{equation}
Since
\[
A_p(p)=\widehat\Pi_p^2-\Pi_p^2,
\]
this proves the asserted transition estimate in the final part of
Proposition~\ref{prop:empitangent_clean}.

\section{Proofs of Lemmas \ref{lem:MC_kernel_vector}, \ref{lem:MC_scalar_kernel},  and \ref{lem:densitybounded}}\label{sec:density}
In this Appendix, we assume that $M^n$ is a smooth closed Riemannian submanifold in $\R^d$, $\Phi_t$ is the extrinsic Gaussian kernel defined in \eqref{eq:phit}:
$$ \Phi_t (x, y) = \frac{1}{(4\pi\,t)^{n/2}}\exp \Big ( - \frac{\|x-y\|^2_{\R^d}}{4t}\Big ).$$
Let $h_t: M^n \times M^n \to E$ denote a kernel taking values in a fixed
finite-dimensional Euclidean or tensor space $E$.  Denote by $\mu$ the uniform probability distribution on $M^n$. We study uniform
concentration of the empirical averages
\[
\frac1m\sum_{j=1}^m\Phi_t(p,x_j)h_t(p,x_j)
\]
around their expectations. Whenever $E\neq\mathbb R$, we reduce the
problem by duality to an associated class of real-valued functions and
apply the concentration inequality of Gin\'e and Guillou to that scalar
class.

We recall the uniform concentration inequality of Gin\'e and Guillou
\cite{GG2002}*{Theorem 2.1}, based on results of Talagrand
\cites{Talagrand1994,Talagrand1996} and Gin\'e--Guillou
\cite{GG2002}*{Proposition 2.2}. Let $\mathcal F$ be a bounded,
measurable, separable VC-subgraph class on $(M^n,\mu)$ with VC
characteristics $(A,v)$. Thus, for every probability measure $P$ on
$M^n$ and every $0<\tau<1$,
\begin{equation}
	N(\mathcal{F}, L_2(P), \tau \|F\|_{L_2(P)}) \le \Big( \frac{A}{\tau} \Big)^v, \label{eq:covering}
\end{equation}
where $N(T, d, \tau)$ denotes the $\tau$-covering number of the metric space $(T, d)$ (the smallest number of balls of radius $\tau$ needed to cover $T$), and $F \coloneqq \sup_{f \in \mathcal{F}} |f|$ is the measurable envelope of the family. In inequality \eqref{eq:covering}, the distance $d$ is the standard $L_2(P)$ metric.

Assume further that there are constants $U$ and $\sigma$ such that
\begin{equation} 
	U \ge \sup_{f\in \mathcal{F}} \|f\|_\infty, \label{eq:U}
\end{equation}
\begin{equation}
	\sigma^2 \ge \sup_{f \in \mathcal{F}} \Var_\mu(f), \label{eq:sigma}
\end{equation}
and which satisfy the relation
\begin{equation}
	0 < \sigma \le U. \label{eq:sigmau}
\end{equation}	
Theorem 2.1 of \cite{GG2002} states that there exist universal constants $C$ and $L$, depending only on the VC characteristics $(A, v)$ of $\mathcal{F}$, such that whenever
\begin{equation}
	\epsilon \ge C \left[ U \log \Big(\frac{AU}{\sigma} \Big) + \sqrt{m \sigma^2 \log \Big(\frac{AU}{\sigma}\Big)} \right], \label{eq:epsc}
\end{equation}
the following probability inequality holds for any $m \in \N^+$.

\medskip

\noindent{\it The probability inequality.}
For an i.i.d. sample
\[
S_m=(x_1,\ldots,x_m)\sim\mu^m,
\]
Gin\'e--Guillou's inequality \cite{GG2002}*{Eq. (2.3)} gives:
\begin{align}
&(\mu^m)^*\left\{
S_m:
\sup_{f\in\mathcal F}
\left|\sum_{i=1}^m\bigl(f(x_i)-\E_\mu f\bigr)\right|>\epsilon
\right\}\nonumber\\
&\qquad\le
L\exp\left\{
-\frac1L\frac{\epsilon}{U}
\log\left(
1+\frac{\epsilon U}
{L\bigl[\sqrt m\,\sigma+U\sqrt{\log(AU/\sigma)}\bigr]^2}
\right)
\right\}.
\label{eq:gg23}
\end{align}
In the remainder of this Appendix, we apply \eqref{eq:gg23} to prove Lemmas \ref{lem:MC_kernel_vector}, \ref{lem:MC_scalar_kernel}, and \ref{lem:densitybounded}.

In each application below, namely Lemmas~\ref{lem:MC_kernel_vector},
\ref{lem:MC_scalar_kernel}, and \ref{lem:densitybounded}, the relevant
function class $\Ff_t$ in \eqref{eq:localizedkernel}, $\Aa_t$ in
\eqref{eq:localizedscalarkernel}, or $\Gg_t$ in
\eqref{eq:singularkernel}, is a bounded finite-dimensional parametric
class. More precisely, for fixed $t>0$, the parameters range over
compact finite-dimensional spaces such as $M$,
$M\times S^{d-1}$, or
$M\times S^{d-1}\times S^{d^2-1}$, and the corresponding kernels
depend smoothly on the parameters. The smooth cutoff function
$\chi_\delta$ removes possible discontinuities at the boundary of the
localized region. Since $\chi_\delta$ is radial and nonincreasing,
the level sets of the localized kernels are controlled by Euclidean
balls. Together with the smooth finite-dimensional parametrization of
the remaining factors, this gives the uniform VC-subgraph property for
the classes $\Ff_t$, $\Aa_t$, and $\Gg_t$. Lemma~\ref{lem:vcuni}
below records that the corresponding VC-characteristics $(A,v)$ may
be chosen independently of $t$.

\begin{lemma}[\bf Uniform VC-type property of the localized classes]
	\label{lem:vcuni}
	For $0<t<t_0$, let $\mathcal C_t$ denote any of the scalar
	localized classes defined in
	\eqref{eq:localizedkernel}, \eqref{eq:localizedscalarkernel}, or
	\eqref{eq:singularkernel}. Then there exist constants $A\ge1$ and
	$v\ge1$, depending only on the finite-dimensional parameter spaces,
	the manifold, the cutoff function, and the particular type of class,
	but independent of $t$, such that \eqref{eq:covering} holds for every
	probability measure $P$ and every $0<\tau<1$.
\end{lemma}

\begin{proof}
Allow $p$, the auxiliary unit vectors or tensors, and $0<t<t_0$ to vary
simultaneously.  The Gaussian family
\[
\mathcal K=
\left\{
y\longmapsto (4\pi t)^{-n/2}
\exp\left(-\frac{\|y-p\|^2}{4t}\right):
p\in M,\ 0<t<t_0
\right\}
\]
is VC-subgraph with index independent of $t$.  Indeed, for $s>0$, the
subgraph inequality
\[
s<(4\pi t)^{-n/2}
\exp\left(-\frac{\|y-p\|^2}{4t}\right)
\]
is equivalent to
\[
\|y-p\|^2+4t\log s+2nt\log(4\pi t)<0.
\]
After expanding $\|y-p\|^2$, this is a linear-threshold condition in
the fixed finite collection of variables
\[
1,\quad y_1,\ldots,y_d,\quad \|y\|^2,\quad \log s.
\]
Restriction from $\R^d$ to $M$ does not increase the VC index.

The translated distance functions and the nonincreasing radial cutoff
form VC-subgraph classes governed by Euclidean balls.  The affine
factors $\langle y-p,u\rangle$ form a finite-dimensional linear class.
The family
\[
\left\{
	y\mapsto \frac1t\langle y-p,u\rangle:
	p\in M,\ |u|=1,\ 0<t\le t_0
	\right\}
\]
is contained in the fixed finite-dimensional affine-linear space
\[
\operatorname{span}\{1,y_1,\ldots,y_d\}.
\]
Indeed,
\[
\frac1t\langle y-p,u\rangle
= \sum_{i=1}^d\frac{u_i}{t}y_i
-  \frac{\langle p,u\rangle}{t}.
\]
Consequently, this family has a VC-subgraph index bounded independently
of $t$.
Moreover,
\[
\langle\Pi_y\Pi_p,A\rangle_{\mathrm{HS}}
\]
is, as a function of $y$, a linear combination of the fixed finitely
many coordinate functions $y\mapsto(\Pi_y)_{ab}$, with coefficients
depending on $(p,A)$.  The same observation applies to the coordinate
functions involving the fixed field $\Pi_y\omega(y)$ in the singular
kernel lemma.  These are finite-dimensional linear classes and hence
VC-subgraph classes \cite{VW1996}*{Lemma 2.6.15}.

The standard permanence properties of VC-subgraph and VC-type classes
under finite sums, products, scalar multiplication, and restriction
therefore show that each enlarged class is VC-type with fixed
characteristics.  Every fixed-$t$ class $\F_t$ is a subclass and has
the same characteristics.  The uniform covering estimate follows from
the VC covering theorem.
\end{proof}

We shall use   Lemma \ref{lem:vcuni}  without further comment in the proofs below.

\begin{lemma}[\bf Uniform Monte Carlo approximation for localized kernels]\label{lem:MC_kernel_vector}\
	Let $M^n \subset \mathbb{R}^d$ be a compact $C^3$ submanifold, $\mu$ the uniform probability distribution on $M^n$, and $x_1,\dots,x_m \sim \mu$ drawn i.i.d.
	Let $\delta \in  (0, \frac{\tau_{M^n}}{4})$. For $p \in M^n$, define the operator-valued kernel
	\[
	F_p(y) \coloneqq \Phi_t(p,y) \frac{(y-p)}{t} \otimes \Pi_y\Pi_p \chi_\delta (p,y).
	\]
	Then for $t = m^{-1/(2n)}$, there exists $C>0$ such that for all sufficiently small $t$, with probability at least $1 - m^{-2}$ over the sample $(x_1, \ldots, x_m)$, we have:
	\begin{equation}
		\sup_{p \in M} \left\| \frac{1}{m}\sum_{j=1}^m F_p(x_j) - \int_M F_p(y)\, d\mu(y) \right\|_{\mathrm{op}} \le C \sqrt{\frac{\log m}{m\, t^{n/2+1}}}. \label{eq:MC_kernel_bound}
	\end{equation}
	If $n \ge 2$, then this bound satisfies
	\begin{equation}
		\sqrt{\frac{\log m}{m\,t^{n/2+1}}} = o(\sqrt{t}). \label{eq:mcscalar2}	
	\end{equation}
\end{lemma}

\begin{proof} 
	First, we observe that for any fixed $t \in \R_+$ and sample $S_m = (x_1, \ldots, x_m) \in (M^n)^m$, the function
	\[
	\mathbf{F}_{t, S_m}: M^n \to \R, \quad p \mapsto \left\| \frac{1}{m}\sum_{j=1}^m F_p(x_j) - \int_M F_p(y)\, d\mu(y) \right\|_{\mathrm{op}},
	\]
	is continuous with respect to $p$. Because $M^n$ is a separable metric space, it contains a countable dense subset $\mathcal{D} \subset M^n$. Hence, for any $a > 0$, the supremum over the uncountable space $M^n$ equals the supremum over $\mathcal{D}$. The event
	\[ 
	\Om_{t, m}^F \coloneqq \left\{ S_m \in (M^n)^m : \sup_{p \in M^n} \mathbf{F}_{t, S_m}(p) \le a \right\} 
	\]
	is therefore a measurable subset of $(M^n)^m$ with respect to the product Borel $\sigma$-algebra, since it can be written as a countable intersection of measurable sets:
	\[ 
	\Om_{t, m}^F = \bigcap_{q \in \mathcal{D}} \left\{ S_m \in (M^n)^m : \mathbf{F}_{t, S_m}(q) \le a \right\}.
	\]
Thus  the outer measure  in \eqref{eq:gg23}   is replaced  by    $\mu^m$ for   the  measurable subset satisfying  \eqref{eq:MC_kernel_bound}.

	We apply the uniform concentration inequality for empirical processes from \cite{GG2002}*{Theorem 2.1}, formulated in \eqref{eq:gg23}, to a scalar-valued class associated to the operator-valued kernels $F_p$.

	The kernel $F_p(y)$ takes values in the finite-dimensional tensor
	space
	\[
	E\coloneqq
	\mathbb R^d\otimes\End(\R^d).
	\]
	We equip $E$ with the injective norm
	\[
	\|T\|_{\varepsilon}
	\coloneqq
	\sup_{\|u\|=1,\ \|A\|_{\mathrm{HS}}=1}
	\bigl|\langle T,u\otimes A\rangle\bigr|.
	\]
	For $p\in M$, $u\in \R^d$, and
	$A\in\End(\R^d)$, define
	\[
	f_{p,u,A}(y)
	\coloneqq
	\Phi_t(p,y)
	\frac{\langle y-p,u\rangle}{t}
	\langle\Pi_y\Pi_p,A\rangle_{\mathrm{HS}}
	\chi_\delta(p,y).
	\]
	Then
	\[
	f_{p,u,A}(y)=\langle F_p(y),u\otimes A\rangle,
	\]
	and hence the desired tensor-valued concentration estimate is
	equivalent to the corresponding scalar estimate uniformly over
	$\{\|u\|=1\}$ and $\{ \|A\|_{\mathrm{HS}}=1\}$.

	Let the function class be
	\begin{equation}
	\mathcal{F}_t \coloneqq \left\{ f_{p,u,A} : p\in M,\ \|u\|=1,\ \|A\|_{\mathrm{HS}}=1 \right\}.\label{eq:localizedkernel}
	\end{equation}

	\medskip
	
	\noindent\underline{Step 1.} {\it Envelope estimate.}
	
	For $r \ge 0$, consider the function
	\[
	\psi_t(r) \coloneqq r e^{-r^2/(4t)}.
	\]
	A direct computation gives
	\[
	\psi_t'(r) = e^{-r^2/(4t)} \left( 1 - \frac{r^2}{2t} \right).
	\]
	Hence, the function $\psi_t$ attains its maximum at $r = \sqrt{2t}$, yielding a maximal value of 
	\[
	\max_{r \ge 0}\psi_t(r) = \sqrt{2t}\,e^{-1/2}.
	\]
	Consequently, noting  that $\sqrt{2t} \le \delta$ for  sufficiently small $t$, we bound the spatial components of the kernel:
	\begin{equation}
		\Phi_t(p,y)\frac{|\langle y-p,u\rangle|}{t} \le C\, t^{-(n/2+1/2)}. \label{eq:phitypmax}
	\end{equation}
	Since the projection operators satisfy
	\[
	|\langle\Pi_y\Pi_p,A\rangle_{\mathrm{HS}}|
	\le \|\Pi_y\Pi_p\|_{\mathrm{HS}}\|A\|_{\mathrm{HS}}
	\le \sqrt d.
	\]

	we can construct the uniform envelope bound:
	\begin{equation}
		U_t \coloneqq C\, t^{-(n/2+1/2)} \ge \sup_{f\in\mathcal{F}_t}\|f\|_\infty. \label{eq:Ut_clean}
	\end{equation}
	Thus, the envelope condition \eqref{eq:U} is satisfied.
	
	\medskip
	
	\noindent\underline{Step 2.} {\it Variance estimate.}
	
	Since $\Var(f) \le \E(f^2)$, it suffices to estimate the second moment.
	Using normal coordinates centered at $p$, we write $y = \exp_p(\sqrt{t}\,v)$. By \eqref{eq:gray} and \eqref{eq:intrextr}, the volume element expands as

	\[
	d\mu(y) = 	\frac{t^{n/2}}{\vol_g(M^n)}(1+O(t\|v\|^2))\,dv.
	\]
	
	In normal coordinates $y=\exp_p(\sqrt t\,v)$,
	\[
	y-p=\sqrt t\,v+O(t\|v\|^2),
	\]
	uniformly in $p$.  Hence, for $\|u\|=1$,
	\[
	|\langle y-p,u\rangle|^2
	\le C\bigl(t\|v\|^2+t^2\|v\|^4\bigr).
	\]

	Therefore, bounding the integral over the normal coordinates:
	\begin{align}
		\E_\mu[f_{p,u,A}^2] &\le  \int_{D_\delta(p)} \Phi_t(p,y)^2 \frac{\langle y-p,u\rangle^2}{t^2} \langle \Pi_y\Pi_p,A\rangle^2 \,d\mu(y) \nonumber\\
		&\le C \int_{\mathbb{R}^n} \left( \frac{1}{(4\pi t)^{n/2}} e^{-\|v\|^2/4} \right)^2 \frac{t\|v\|^2}{t^2} t^{n/2}\,dv \nonumber\\
		&= C\, t^{-(n/2+1)}. \label{eq:variance_upper}
	\end{align}
	Hence, setting
	\begin{equation}
		\sigma_t^2 \coloneqq C\, t^{-(n/2+1)}, \label{eq:sigma_upper}
	\end{equation}
	and taking into account \eqref{eq:Ut_clean}, we conclude that the bounds \eqref{eq:sigma} and \eqref{eq:sigmau} are satisfied for $t$ sufficiently small.
	
	From \eqref{eq:Ut_clean} and \eqref{eq:sigma_upper}, the ratio between the envelope and standard deviation scales as:
	\begin{equation}
		\frac{U_t}{\sigma_t} \asymp t^{-n/4}. \label{eq:ratio_clean_final}
	\end{equation}
	Hence, the logarithmic VC penalty scales directly with the bandwidth:
	\begin{equation}
		\log\left(\frac{AU_t}{\sigma_t}\right) \asymp \log(1/t). \label{eq:logratio_final}
	\end{equation}
	Under our assumed scaling $t = m^{-1/(2n)}$, we obtain:
	\begin{equation}
		\log(1/t) \asymp \log m. \label{eq:logm_final}
	\end{equation}

	\medskip
	
	\noindent\underline{Step 3.} {\it Application of \eqref{eq:gg23} (\cite{GG2002}*{Theorem 2.1}).}
	
	We define the normalized target fluctuation $\eta_m$ and the unnormalized threshold $\epsilon_m = m\eta_m$:
	\[
	\eta_m \coloneqq C_0 \sqrt{ \frac{\log m}{m\, t^{n/2+1}} }, \quad \text{and} \quad \epsilon_m \coloneqq m\eta_m = C_0 \sqrt{ m\, t^{-(n/2+1)}\log m }.
	\]
	To apply \eqref{eq:gg23}, we must verify that the threshold condition \eqref{eq:epsc},
	\[
	\epsilon_m \ge C \left[ U_t \log\left(\frac{A U_t}{\sigma_t}\right) + \sqrt{m \sigma_t^2 \log\left(\frac{A U_t}{\sigma_t}\right)} \right],
	\]
	holds for some fixed positive constant $C$ if $C_0$ is large enough and $t$ is sufficiently small.
	Using \eqref{eq:logratio_final}, $\log(A U_t/\sigma_t) \asymp \log(t^{-n/4}) \asymp \log m$. 
	Hence, the required threshold condition \eqref{eq:epsc} becomes 
	\[
	\epsilon_m \ge C \left[ t^{-(n/2+1/2)}\log m + \sqrt{ m\, t^{-(n/2+1)}\log m } \right].
	\]
	Under the scaling $t = m^{-1/(2n)}$, the second term (the variance term) strictly dominates the first (the envelope term).
	Hence, for a sufficiently large choice of $C_0$, condition \eqref{eq:epsc} of \cite{GG2002}*{Theorem 2.1} is satisfied.
	
	To explicitly evaluate the exponential probability tail bound in \eqref{eq:gg23}, let $V_m \coloneqq \sqrt{m}\sigma_t + U_t \sqrt{\log(A U_t / \sigma_t)}$. Because the variance term dominates, we have $V_m^2 \asymp m \sigma_t^2$. We examine the argument of the logarithm in the exponent:
	\[
	x_m \coloneqq \frac{\epsilon_m U_t}{L V_m^2} \asymp \frac{C_0 \sqrt{m \sigma_t^2 \log m} \cdot U_t}{L m \sigma_t^2} = \frac{C_0 U_t}{L \sqrt{m \sigma_t^2}} \sqrt{\log m}.
	\]
	Substituting $U_t \asymp t^{-(n+1)/2}$ and $\sqrt{m \sigma_t^2} \asymp \sqrt{t^{-2n} t^{-(n/2+1)}} = t^{-(5n/4 + 1/2)}$, we find that $x_m \asymp t^{(3n)/4} \sqrt{\log(1/t)}$. Because $t \to 0$ and $n \ge 1$, we clearly have $x_m \to 0$. 
	Using the standard inequality $\log(1+x) \ge x/2$ for sufficiently small $x > 0$, the exponent in \eqref{eq:gg23} is bounded above by:
	\begin{align*}
		-\frac{1}{L} \frac{\epsilon_m}{U_t} \log(1 + x_m) &\le -\frac{1}{L} \frac{\epsilon_m}{U_t} \left(\frac{1}{2} \frac{\epsilon_m U_t}{L V_m^2} \right) = -\frac{\epsilon_m^2}{2 L^2 V_m^2} \\
		&\asymp -\frac{C_0^2 m \sigma_t^2 \log m}{2 L^2 m \sigma_t^2} = -C' C_0^2 \log m.
	\end{align*}
	Therefore, the probability on the right-hand side of
\eqref{eq:gg23} is bounded by
\[
L\exp(-C'C_0^2\log m)=Lm^{-C'C_0^2}.
\] 
	
	By choosing $C_0$ to be sufficiently large such that $C' C_0^2 \ge 3$, we conclude that with probability at least $1-m^{-2}$,
	\[
	\sup_{p\in M} \left\| \frac1m\sum_{j=1}^m F_p(x_j) - \int_M F_p(y)\,d\mu(y) \right\|_{\mathrm{op}} \le C_0 \sqrt{ \frac{\log m}{m\, t^{n/2+1}} }
	\]
	for $t$ sufficiently small. This proves the first assertion of Lemma \ref{lem:MC_kernel_vector}.
	
	To prove the second assertion, we evaluate the limit under the scaling constraint $m = t^{-2n}$:
	$$ \sqrt{\frac{\log m}{m\, t^{n/2 +1}}} = \sqrt{\frac{ 2n \log(1/t)}{t^{-2n} t^{n/2 + 1}}} = \sqrt{\frac{ 2n \log(1/t)}{t^{\frac{-3n + 2}{2}}}} = \sqrt{2n \log(1/t)}\; t^{\frac{3n - 2}{4}}. $$
	If $n \ge 2$, the exponent on $t$ satisfies $\frac{3n-2}{4} \ge \frac{4}{4} = 1$. Since $t \to 0$, it immediately follows that $t^1 \sqrt{\log(1/t)} = o(\sqrt{t})$.
	
	This completes the proof of Lemma \ref{lem:MC_kernel_vector}.
\end{proof}
\begin{lemma}[\bf Uniform Monte Carlo estimate for scalar Gaussian kernels]
	\label{lem:MC_scalar_kernel}
	Let $M^n\subset \R^d$ be a compact $C^3$-submanifold and $\mu$ the uniform probability distribution on $M^n$. Let $x_1,\dots,x_m\sim\mu$ be i.i.d.
	Let $\delta \in  (0, \frac{\tau_{M^n}}{4})$.
	For a fixed $t >0$ and $p\in M^n$, define
	\[
	A_p(y) \coloneqq \Phi_t(p,y)\frac{\|y-p\|}{t} \chi_\delta (p,y).
	\]
	Assume $t=m^{-1/(2n)}$. Then there exists $C>0$ such that for sufficiently small $t$, with probability at least $1-m^{-2}$ over the choice of $(x_1, \ldots, x_m)$,
	\begin{equation}
		\sup_{p\in M} \left| \frac1m\sum_{j=1}^m A_p(x_j) - \E_\mu[A_p] \right| \le C \sqrt{\frac{\log m}{m\,t^{n/2+1}}}.
		\label{eq:mcscalar}
	\end{equation}
	Furthermore, if $n \ge 2$, then under this scaling we have:
	\begin{equation}
		\sqrt{\frac{\log m}{m\,t^{n/2+1}}} = o(\sqrt{t}).	
	\end{equation}
\end{lemma}

\begin{proof}  As established in the proof of Lemma \ref{lem:MC_kernel_vector}, the continuity of the kernel for a fixed $t > 0$ over the separable space $M^n$ ensures that the supremum event is Borel measurable. Therefore, we may replace the outer measure $(\mu^m)^*$ in \eqref{eq:gg23} with the standard measure $\mu^m$ for the subset satisfying \eqref{eq:mcscalar}. 

	We apply \cite{GG2002}*{Theorem 2.1}, formulated in \eqref{eq:gg23}, to the scalar class
	\begin{equation}
	\mathcal{A}_t \coloneqq \{A_p : p \in M\}.\label{eq:localizedscalarkernel}
	\end{equation}
	\noindent
	\underline{Step 1.} {\it Envelope estimate.}
	
	Consider the function $\psi_t(r) = re^{-r^2/(4t)}$. 
	As shown in the proof of Lemma \ref{lem:MC_kernel_vector}, $\psi_t$ attains its maximum at $r=\sqrt{2t}$.
	Hence,  for sufficiently small $t$
	\[
	\sup_{p,y} A_p(y) \le C\,t^{-(n/2+1/2)}.
	\]
	Therefore, by setting
	\begin{equation}
		U_t \coloneqq C\,t^{-(n/2+1/2)} \ge \sup_{A_p \in \mathcal{A}_t} \|A_p\|_\infty, \label{eq:utscalar}
	\end{equation}
	the envelope condition \eqref{eq:U} is satisfied.	
	
	\medskip
	\noindent
	\underline{Step 2.} {\it Variance estimate.}
	
	Since $\Var(A_p)\le \E[A_p^2]$, it suffices to estimate the second moment.
	Using normal coordinates $y=\exp_p(\sqrt{t}u)$, exactly as in Step 2 of the proof of Lemma \ref{lem:MC_kernel_vector}, we obtain:
	\begin{align}
		\E_\mu[A_p^2] &\le \int_{D_\delta(p)} \Phi_t(p,y)^2 \frac{\|y-p\|^2}{t^2} \,d\mu(y) \nonumber\\
		&\le C \int_{\R^n} \left( \frac1{(4\pi t)^{n/2}} e^{-\|u\|^2/4} \right)^2 \frac{t\|u\|^2}{t^2} t^{n/2}\,du \nonumber\\
		&= C\,t^{-(n/2+1)}.
	\end{align}
	Therefore, setting
	\begin{equation}
		\sigma_t^2 \coloneqq C\,t^{-(n/2+1)} \ge \sup_{A_p \in \mathcal{A}_t} \Var_\mu(A_p),
	\end{equation}
	and taking into account \eqref{eq:utscalar}, we conclude that the variance conditions \eqref{eq:sigma} and \eqref{eq:sigmau} are satisfied for $t$ sufficiently small.

	\noindent
	\underline{Step 3.} {\it Application of \eqref{eq:gg23} (\cite{GG2002}*{Theorem 2.1}).}
	
	We observe that our envelope $U_t$ and variance bound $\sigma_t^2$ are identically equal to those established in the proof of Lemma \ref{lem:MC_kernel_vector}. Thus, the threshold condition verification and the probability tail bound expansion follow exactly the same algebraic steps. Importing the result of Step 3 and the asymptotic evaluation from Lemma \ref{lem:MC_kernel_vector} completes the proof of Lemma \ref{lem:MC_scalar_kernel}.
\end{proof}
\begin{lemma}[\bf Concentration for the singular empirical kernel]\label{lem:densitybounded} 
	Assume that $M^n \subset \R^d$ is a compact $C^3$ submanifold and $\mu$ is the uniform probability measure on $M^n$. Let $x_1, \ldots, x_m \sim \mu$ be i.i.d. Let $\delta \in  (0, \frac{\tau_{M^n}}{4})$.
	For a fixed $x \in M^n$ and $\om \in C^\infty(M^n, \Lambda^k \R^d)$, define for $t \in \R_+$ the vector-valued function:
	\[   
	f_x(y) \coloneqq \frac{1}{t} \Phi_t(x,y) \big(\Pi_x \omega(x) - \Pi_x\Pi_y \omega(y) \big)  \chi_\delta  (x, y),
	\]		
	where $\Phi_t$ is defined in \eqref{eq:phit}. Then there exists $C > 0$ such that for $t = m^{-1/(2n)}$ sufficiently small,
	\begin{equation}
		\sup_{x\in M} \left\| \frac{1}{m}\sum_{j=1}^m f_x(x_j) - \E_{y \sim \mu}[f_x(y)] \right\| \le C \sqrt{\frac{\log m}{m\, t^{n/2+1}}} \|\om\|_{C^1} \label{eq:mc_error}
	\end{equation}
	with probability at least $1-m^{-2}$ over i.i.d. $S_m \sim \mu^m$.
	
	Furthermore, under the scaling $t = m^{-1/(2n)}$, this statistical bound converges to zero as $m \to \infty$ for any dimension $n \ge 2$, satisfying:
	\begin{equation}
		\sqrt{\frac{\log m}{m\, t^{n/2+1}}} = o(t^{1/2}). \label{eq:density}
	\end{equation} 
\end{lemma}

\begin{proof} As in the proof of Lemma \ref{lem:MC_kernel_vector}, we observe that due to the continuity of the empirical functions over a separable metric space, the relevant supremum event is Borel measurable. Thus, we can replace the outer measure $(\mu^m)^*$ in \eqref{eq:gg23} with the standard measure $\mu^m$ for the subset satisfying \eqref{eq:mc_error}.  Since both occurrences of $\omega$ are preceded by the appropriate
	tangential projection, replacing $\omega$ by the smooth tangential
	field $p\mapsto\Pi_p\omega(p)$ leaves $f_x$ unchanged.  Moreover,
	compactness and smoothness of $p\mapsto\Pi_p$ give
	$\|\Pi\omega\|_{C^1}\le C\|\omega\|_{C^1}$.
	
	To bound the vector norm, we apply the concentration inequality \eqref{eq:gg23} to a dual scalar class. Define:
	\begin{equation}
	\mathcal{G}_t \coloneqq \left\{ y \mapsto \langle f_x(y), A \rangle \;\middle|\; x \in M^n, A \in \Lambda^k \R^d, \|A\| = 1 \right\}.\label{eq:singularkernel}
	\end{equation}
	By duality, the supremum of the absolute value over $\mathcal{G}_t$ controls the norm in \eqref{eq:mc_error}.
	
	\medskip
	\noindent
	\underline{Step 1.} {\it Envelope estimate.}
	
	Since $\omega$ is smooth and the orthogonal projections $\Pi_x, \Pi_y$ depend smoothly on $x, y \in M^n$,    with $\Pi_x ^2 = \Pi_x$,  we have 
	\[
	\Pi_x\Pi_y = \Pi_x  +  O (\|  x-y\| )
	\]
	Hence, for any $x, y \in M^n$:
	\[
	\|\Pi_x\omega(x) - \Pi_x\Pi_y\omega(y)\| \le C\|x-y\|\,\|\omega\|_{C^1}.
	\]
	A direct optimization of $r \mapsto r e^{-r^2/(4t)}$ yields a maximum of order $\sqrt{t}$. Thus, if $\sqrt t \le \delta$,  for any function $g \in \mathcal{G}_t$:
	\begin{align}
		|g(y)| &\le \frac{1}{t} \Phi_t(x,y) \|\Pi_x\omega(x) - \Pi_x\Pi_y\omega(y)\| \|A\| \nonumber\\
		&\le C \frac{1}{t} \Phi_t(x,y) \|x-y\|\,\|\omega\|_{C^1} \nonumber\\
		&\le C' \frac{1}{t} t^{-n/2} \sqrt{t}\,\|\omega\|_{C^1} = C' t^{-(n+1)/2}\|\omega\|_{C^1}. \label{eq:fxenvelope}
	\end{align}
	Hence, the class $\mathcal{G}_t$ has the envelope:
	\begin{equation}
		U_t \coloneqq C' t^{-(n+1)/2}\|\omega\|_{C^1}. \label{eq:ugt}
	\end{equation}
	
	\medskip
	\noindent
	\underline{Step 2.} {\it Variance estimate.}
	
	Using normal coordinates $y = \exp_x(\sqrt{t}v)$ and taking into account \eqref{eq:intrextr} and \eqref{eq:gray}, the volume element expands as $d\mu(y) =  (\vol_g (M^{n}) ^{-1})t^{n/2}(1 + O(t|v|^2))dv$. Therefore:
	\begin{align}
		\sup_{g \in \mathcal{G}_t} \int_{M^n} |g(y)|^2\,d\mu(y) &\le \sup_x \int_{D_\delta(x)} \frac{1}{t^2} \Phi_t(x,y)^2 \big( C \|x-y\|\,\|\omega\|_{C^1} \big)^2 d\mu(y) \nonumber\\
		&\le C^2 \|\omega\|_{C^1}^2 \int_{\R^n} \frac{1}{t^2} \left( \frac{1}{(4\pi t)^{n/2}} e^{-|v|^2/4} \right)^2 (t|v|^2)\, t^{n/2} dv \nonumber\\
		&\le C'' \frac{1}{t^2} t^{-n+1} \cdot t^{n/2}\, \|\omega\|_{C^1}^2 \nonumber\\
		&= C'' t^{-n/2-1}\|\omega\|_{C^1}^2.
	\end{align}
	Thus, by setting
	\begin{equation}
		\sigma_t^2 \coloneqq C'' t^{-n/2-1}\|\omega\|_{C^1}^2, \label{eq:sigmat_upper_density}
	\end{equation}
	and taking into account \eqref{eq:ugt}, we conclude that the conditions \eqref{eq:sigma} and \eqref{eq:sigmau} are satisfied for $t$ sufficiently small. 
	
	Note that
	\begin{equation}
		\sigma_t^2 \asymp t^{-n/2-1} \implies \sigma_t \asymp t^{-n/4-1/2}.
	\end{equation}

	\noindent
	\underline{Step 3.} {\it Application of \eqref{eq:gg23} (\cite{GG2002}*{Theorem 2.1}).}
	
	We define the normalized target fluctuation $\eta_m$ and the unnormalized threshold $\epsilon_m = m\eta_m$:
	\begin{equation}
		\eta_m = C_0 \sqrt{ \frac{\log m}{m\, t^{n/2+1}} } \|\om\|_{C^1} \implies \epsilon_m = C_0 \sqrt{ m\, t^{-n/2-1} \log m } \|\om\|_{C^1}.
	\end{equation}
	To apply \eqref{eq:gg23}, we verify the threshold condition \eqref{eq:epsc}:
	\[
	\epsilon_m \ge K \left[ U_t \log\left(\frac{A U_t}{\sigma_t}\right) + \sqrt{m \sigma_t^2 \log\left(\frac{A U_t}{\sigma_t}\right)} \right]
	\]
	for some positive constant $K$. Using our bounds $U_t \asymp t^{-(n+1)/2}$ and $\sigma_t^2 \asymp t^{-n/2-1}$, the logarithmic ratio scales as $\frac{U_t}{\sigma_t} \asymp \frac{t^{-(n+1)/2}}{t^{-n/4-1/2}} = t^{-n/4}$, which is exactly the same as in the unscaled case, so $\log(A U_t/\sigma_t) \asymp \log(t^{-n/4}) \asymp \log m$. 
	
	We check if the variance term strictly dominates the envelope term on the right-hand side:
	\[
	\frac{\sqrt{m \sigma_t^2 \log m}}{U_t \log m} \asymp \frac{\sqrt{m t^{-n/2-1} \log m}}{t^{-(n+1)/2} \log m} = \sqrt{\frac{m t^{n/2}}{\log m}}.
	\]
	Under the scaling $t = m^{-1/(2n)}$, we have $t^{n/2} = m^{-1/4}$, so the ratio scales as $m^{3/8} / \sqrt{\log m}$. Since $m^{3/8} \gg \sqrt{\log m}$ as $m \to \infty$, the variance term  dominates. By choosing $C_0$ sufficiently large, $\epsilon_m$ easily exceeds the threshold condition.
	
	To explicitly evaluate the probability tail bound in \eqref{eq:gg23}, let $V_m \coloneqq \sqrt{m}\sigma_t + U_t \sqrt{\log(A U_t / \sigma_t)}$. Because the variance term  dominates, $V_m^2 \asymp m \sigma_t^2$. We examine the argument of the logarithm in the exponent:
	\[
	x_m \coloneqq \frac{\epsilon_m U_t}{L V_m^2} \asymp \frac{C_0 \sqrt{m \sigma_t^2 \log m} \cdot U_t}{L m \sigma_t^2} = \frac{C_0 U_t}{L \sqrt{m \sigma_t^2}} \sqrt{\log m}.
	\]
	Substituting $U_t \asymp t^{-(n+1)/2}$ and $\sqrt{m \sigma_t^2} \asymp \sqrt{t^{-2n} t^{-n/2-1}} = t^{-5n/4-1/2}$, we find that $x_m \asymp t^{3n/4} \sqrt{\log(1/t)}$. Since $t \to 0$, $x_m \to 0$. 
	
	Using the inequality $\log(1+x) \ge x/2$ for sufficiently small $x > 0$, the exponent in \eqref{eq:gg23} is bounded above by:
	\begin{align*}
		-\frac{1}{L} \frac{\epsilon_m}{U_t} \log(1 + x_m) &\le -\frac{1}{L} \frac{\epsilon_m}{U_t} \left(\frac{1}{2} \frac{\epsilon_m U_t}{L V_m^2} \right) = -\frac{\epsilon_m^2}{2 L^2 V_m^2} \\
		&\asymp -\frac{C_0^2 m \sigma_t^2 \log m}{2 L^2 m \sigma_t^2} = -C' C_0^2 \log m.
	\end{align*}
	Therefore, the right-hand side of \eqref{eq:gg23} is bounded by $L \exp(-C' C_0^2 \log m) = L m^{-C' C_0^2}$. 
	By choosing $C_0$ sufficiently large such that $C' C_0^2 \ge 3$, we conclude that with probability at least $1 - m^{-2}$:
	\[
	\sup_{g \in \mathcal{G}_t} \left| \frac{1}{m}\sum_{j=1}^m g(x_j) - \E_\mu[g] \right| \le C_0 \sqrt{ \frac{\log m}{m\, t^{n/2+1}} } \|\om\|_{C^1}.
	\]
	This completes the proof of the first assertion \eqref{eq:mc_error}.
	
	To prove the second assertion \eqref{eq:density}, we substitute $m = t^{-2n}$ into the rate:
	\begin{equation}
		\sqrt{\frac{\log m}{m\, t^{n/2+1}}} = \sqrt{\frac{\log(t^{-2n})}{t^{-2n} t^{n/2+1}}} = \sqrt{t^{2n - n/2 - 1} \log(t^{-2n})} = \sqrt{t^{\frac{3n}{2} - 1} \log(t^{-2n})}.
	\end{equation}
	For any manifold dimension $n \ge 2$, the exponent is $\frac{3n}{2} - 1 \ge \frac{3(2)}{2} - 1 = 2$.
	Therefore, the term is bounded by $\mathcal{O}(t \sqrt{\log(1/t)})$, which is  $o(t^{1/2})$.
	This completes the proof of Lemma \ref{lem:densitybounded}.
\end{proof}


\begin{thebibliography}{99999}
	
	\bibitem{AB2006}  S. B. Alexander,  R. L. Bishop,  Gauss equation and injectivity radii for subspaces in spaces of curvature bounded above. Geom. Dedicata 117(2006), 65--84.
	
	\bibitem{AL2019} E. Aamari, C. Levrard, Nonasymptotic rates for manifold, tangent space and curvature estimation. Ann. Statist. 47(1): 177-204 (2019). DOI: 10.1214/18-AOS1685, long version \url{https://doi.org/10.1214/18-AOS1685}
	
	\bibitem{BGV1996} N. Berline, E. Getzler, M. Vergne, Heat Kernels and Dirac Operators. 2nd Edition, Springer, 1996.	

	\bibitem{BN2003} M. Belkin, P. Niyogi, Laplacian eigenmaps for dimensionality reduction and data representation. Neural computation, 15(6), 1373--1396 (2003).
	
	\bibitem{BN2008} M. Belkin, P. Niyogi, Towards a theoretical foundation for Laplacian-based manifold methods. Journal of Computer and System Sciences, 74(8), 1289--1308 (2008).
	
	\bibitem{BN20062008} M. Belkin, P. Niyogi, Convergence of Laplacian Eigenmaps. Advances in Neural Information Processing Systems 19 (NIPS 2006), 129-136 \url{https://proceedings.neurips.cc/paper_files/paper/2006/file/5848ad959570f87753a60ce8be1567f3-Paper.pdf}, long version: \url{https://misha.belkin-wang.org/papers/CLEM_08.pdf}

	\bibitem{CLS2021} Y. Cao, D. Li, H. Sun, et al., Efficient Weingarten map and curvature estimation on manifolds. Mach Learn. 110, 1319--1344   (2021). \url{https://doi.org/10.1007/s10994-021-05953-4}
	

\bibitem{Chen1977}    K.T., Chen,   Iterated path integrals. Bull. Am. Math. Soc. 83 (1977), 831--879. 




	\bibitem{DK1970} C. Davis, W.M. Kahan, The rotation of eigenvectors by a perturbation. III, SIAM Journal on Numerical Analysis, vol. 7 (1970) N 1, 1--46.
	
	
	\bibitem{DHK2011}  T.K. Dey, A.N.  Hirani,  and B. Krishnamoorthy, 
	Optimal homologous cycles, total unimodularity, and linear programming.
	SIAM Journal on Computing, 40, N4(2011), SIAM, 1026--1044.

	 
	 \bibitem{Federer1969}  H. Federer, Geometric Measure Theory,
	 Die Grundlehren der mathematischen Wissenschaften, 153  (1969),
	 Springer-Verlag, New York.
	 
	
	\bibitem{Federer1959} H. Federer, Curvature measures. Trans. Amer. Math. Soc. 93(1959) 418--491.
	
	\bibitem{FKLS2021} D. Fiorenza, K. Kawai, H. V. L\^e and L. Schwachh\"ofer, Almost formality of manifolds of low dimension. Ann. Sc. Norm. Super. Pisa Cl. Sci. (5), vol. XXII (2021), 79-107.
	
	\bibitem{FL2025} D. Fiorenza,  H. V. L\^e,  	
	Unital $C_\infty$-algebras and the real homotopy type of $(r-1)$-connected compact manifolds of dimension $\leq \ell(r-1)+2$, https://arxiv.org/abs/2310.19506,  Ann. Sc. Norm. Super. Pisa,  \url{https://doi.org/10.2422/2036-2145.202401_003}
	\newblock 2025.
	
\bibitem{GarciaTrillos2015} N. Garc{\'\i}a Trillos  and D. Slep{\v{c}}ev,  A variational approach to the consistency of spectral clustering, Applied and Computational Harmonic Analysis,  45, (2018) No. 2, 239--281.

	
	\bibitem{GG2002} E. Gin\'e, A. Guillou, Rates of strong uniform consistency for multivariate kernel density estimators. Annales de l'Institut Henri Poincar\'e, Probabilités et Statistiques, Vol. 38, No. 6, (2002), pp. 907--921.

	
\bibitem{Golub2013}	G. H. Golub,  C. F. Van Loan.
	Matrix Computations, 4th  edition,
Johns Hopkins University Press,  (2013)

	
	\bibitem{Gray1974} A. Gray, The volume of a small geodesic ball of a Riemannian manifold. Michigan Mathematical Journal, vol.20, no.4, 329--344 (1974).
	
	\bibitem{HJ2012} R. A. Horn, C. R. Johnson, Matrix analysis (2nd ed.). Cambridge University Press. (2012).
	
	\bibitem{Jost2017} J. Jost, Riemannian Geometry and Geometric Analysis. 7th Edition, Springer, 2017.
	
	\bibitem{Kato1995} T. Kato, Perturbation Theory for Linear Operators. (Classics in Mathematics), Springer (1995).


\bibitem{Knyazev2001}, A. V. Knyazev, Toward the optimal preconditioned eigensolver: Locally optimal block preconditioned conjugate gradient method,
SIAM Journal on Scientific Computing, 23(2001), Nr. 2, 517--541,


	\bibitem{KN1963} S. Kobayashi, K. Nomizu, Foundations of Differential Geometry. vol. I, Interscience Publishers, 1963.
	
	\bibitem{KN1969} S. Kobayashi, K. Nomizu, Foundations of Differential Geometry. vol. II, Interscience Publishers, 1969.
	
\bibitem{Kuwae2003}	K.	Kuwae,  and T. Shioya,Convergence of spectral structures: a functional
analytic theory and its applications to spectral
geometry.  Communications in
analysis and geometry
Volume 11(2003), Number 4, 599-673.
		
		
\bibitem{lanczos1950} C.  Lanczos, An iteration method for the solution of the eigenvalue problem of linear differential and integral operators,
Journal of Research of the National Bureau of Standards, 45 (1950), 4, 255--282,

	
	\bibitem{Le2026} H. V. L\^e, Minimal Unital Cyclic $C_\infty$-Algebras and the Real and Rational Homotopy Type of Closed Manifolds. \url{https://arxiv.org/abs/2603.01219}
	
	\bibitem{LMPT2026}	H. V. L\^e, H. Q. Minh, F. Protin, W. Tuschmann, Mathematical Foundations of Machine Learning, Springer 2026 (to appear).
	
	\bibitem{Merkulov1999} S. A. Merkulov, Strong homotopy algebras of a K\"ahler manifold. Int. Math. Res. Not. IMRN (1999), Nr 3, 153--164.
	
	\bibitem{Lyons1998}  T. J. Lyons, Differential equations driven by rough signals.
	Revista Matem{\'a}tica Iberoamericana, vol. 14(1998), Nr. 2, 215--310,1998,
	doi:10.4171/RMI/240.
	
	\bibitem{NSW2008} P. Niyogi, S. Smale, S. Weinberger, Finding the homology of submanifolds with high confidence from random samples. Discrete and Computational Geometry, vol. 39, nos. 1--3, 419-441 (2008) \url{https://doi.org/10.1007/s00454-008-9053-2}
	
	\bibitem{Rosenberg1997} S. Rosenberg, The Laplacian on a Riemannian manifold : an introduction to analysis on manifolds. London Mathematical Society student texts. Cambridge University Press, Cambridge, U.K., New York, NY, USA, 1997.
	
	\bibitem{SW2011} A. Singer, H.-T. Wu, Orientability and diffusion map. Appl. Comput. Harmon. Anal., 31(2011), 44--58.
	
	\bibitem{SW2017} A. Singer, H.-T. Wu, Spectral convergence of the connection Laplacian from random samples. Information and Inference: A Journal of the IMA (2017) 6, 58--123 \url{https://doi:10.1093/imaiai/iaw016}
	
	\bibitem{SW2012} A. Singer, H.-T. Wu, Vector diffusion maps and the connection Laplacian. Comm. Pure Appl. Math., 65(2012), 1067--1144 \url{https://doi.org/10.1002/cpa.21395}
	

	\bibitem{Talagrand1994} M. Talagrand, Sharper bounds for Gaussian and empirical processes. Ann. Probab. 22 (1994) 28--76.
	
	\bibitem{Talagrand1996} M. Talagrand, New concentration inequalities in product spaces. Invent. Math. 126 (1996) 505--563.
	
\bibitem{Trillos2020}    N. 	Garcia Trillos, M.  Gerlach,  M.  Hein,  and D. Slep\v cev, Error estimates for spectral convergence of the graph Laplacian on random geometric graphs toward the Laplace--Beltrami operator. Foundations of Computational Mathematics, 20(4), 827--887 (2020).
	
	\bibitem{VW1996} A. W. van der Vaart, J.A. Wellner, Weak convergence and Empirical Processes. 2nd Edition. Springer (1996).
	
	\bibitem{YWS2015} Y. Yu, T. Wang, R.J. Samworth, A useful variant of the Davis--Kahan theorem for statisticians. Biometrika, 102(2015), N.2, 315-323.
\end{thebibliography}
\end{document}